\documentclass[12pt]{article}
\usepackage{amssymb, amsfonts,amsmath,amsthm, float, url, bm,mathrsfs}
\usepackage{multirow}
\usepackage{tkz-euclide}
\newcommand{\dmn}{\Delta_{M,N}}
\newcommand{\djn}{\Delta_{[J],N}}

\newcommand{\Xn}{X^{(N)}}
\newcommand{\bXn}{{\bf X}^{(N)}}

\newcommand{\bx}{{\bf x}}
\newcommand{\bu}{{\bf u}}
\newcommand{\bv}{{\bf v}}
\newcommand{\bw}{{\bf w}}
\newcommand{\bff}{{\bf f}}
\newcommand{\bGamma}{{\bf \Gamma}}
\newcommand{\bvarphi}{{\bm \varphi}}
\newcommand{\bUpsilon}{{\bm \Upsilon}}
\newcommand{\bSigma}{{\bm \Sigma}}
\newcommand{\bTheta}{{\bm \Theta}}
\newcommand{\bpsi}{{\bm \psi}}
\newcommand{\bnu}{{\bm \nu}}
\newcommand{\by}{{\bf y}}
\newcommand{\bz}{{\bf z}}
\newcommand{\bA}{{\bf A}}
\newcommand{\bB}{{\bf B}}
\newcommand{\bD}{{\bf D}}
\newcommand{\bH}{{\bf H}}
\newcommand{\bg}{{\bf g}}
\newcommand{\calun}{\calu^{(N)}}

\newcommand{\Zn}{Z^{(N)}}
\newcommand{\bZn}{{\bf Z}^{(N)}}
\newcommand{\zx}{\zeta^{(\bx)}}
\newcommand{\xx}{\xi^{(\bx)}}

\newcommand{\bxxk}{ {\bm \xi}_{k,i}^{(\bx)}}

\newcommand{\bYn}{{\bf Y}^{(N)}}

\newcommand{\bun}{{\bf u}^{(N)}}

\newcommand{\rnd}{\mbox{\rm and}}
\newcommand{\ol}{\overline}
\newcommand{\one}[1]{\mbox {\bf 1}_{\{#1\}}}
\newcommand{\odin}{\mbox {\bf 1}}

\newcommand{\witi}{\widetilde}

\newcommand{\zz}{{\mathbb Z}}
\newcommand{\nn}{{\mathbb N}}

\newcommand{\rr}{{\mathbb R}}

\newcommand{\cala}{{\mathcal A}}
\newcommand{\calc}{{\mathcal C}}

\newcommand{\cale}{{\mathcal E}}

\newcommand{\calr}{{\mathcal R}}

\newcommand{\calm}{{\mathcal M}}
\newcommand{\caln}{{\mathcal N}}

\newcommand{\calu}{{\mathcal U}}
\newcommand{\calf}{{\mathcal F}}

\newcommand{\mz}{{\chi_{eq} }}
\newcommand{\bchi}{{\bm \chi }}
\newcommand{\bmz}{{\bm {\chi_{eq}}}}
\newcommand{\bmza}{{\bm {\chi^{(1)}_{eq}}}}
\newcommand{\bmzb}{{\bm {\chi^{(2)}_{eq}}}}
\newcommand{\veps}{\varepsilon}

\newcommand{\be}{{\bf e}}

\newcommand{\bb}{{\bf b}}

\newcommand{\bR}{{\bf R}}
\newcommand{\bX}{{\bf X}}

\newcommand{\bU}{{\bf U}}

\newcommand{\bo}{{\bf 0}}

\newcommand{\beq}{\begin{eqnarray*}}
	\newcommand{\feq}{\end{eqnarray*}}
\newcommand{\beqn}{\begin{eqnarray}}
\newcommand{\feqn}{\end{eqnarray}}

\newcommand{\as}{a.\,s.}

\newtheorem{theorem}{Theorem}
\makeatletter \@addtoreset{theorem}{section}\makeatother

\newcounter{vadik}
\newcounter{reza}
\newtheorem{definition}[theorem]{Definition}
\newtheorem{lemma}[theorem]{Lemma}
\newtheorem{assume}[theorem]{Assumption}
\newtheorem*{theorema*}{Theorem~A}
\newtheorem{conj}[theorem]{Conjecture}
\newtheorem*{surjt*}{Surjectivity Theorem \cite{garay}}
\newtheorem*{theoremb*}{Theorem~B}
\newtheorem*{cld*}{Condition $\mbox{LD}_d$}
\newtheorem*{theorem*}{Theorem}
\newtheorem{theorema}[vadik]{Theorem}
\newtheorem{theoremc}[reza]{Theorem}

\newtheorem{proposition}[theorem]{Proposition}
\newtheorem{corollary}[theorem]{Corollary}
\newtheorem{remark}[theorem]{Remark}
\newtheorem{example}[theorem]{Example}
\title{Extinction scenarios in evolutionary processes: A Multinomial Wright-Fisher approach}
\date{December 2, 2019}
\author{Alexander~Roitershtein \thanks{Department of Statistics, Texas A\&M University, College Station, TX 77843, USA;
\newline e-mail: alexander@stat.tamu.edu; ORCID ID: 0000-0001-8207-4289}
\and
Reza~Rastegar\thanks{Occidental Petroleum Corporation, Houston, TX 77046 and Departments of Mathematics and Engineering, University of Tulsa, OK 74104, USA - Adjunct Professor;
e-mail:  reza\_rastegar2@oxy.com; ORCID ID: 0000-0003zz-1011-651X}
\and Robert S. Chapkin\thanks{Department of Nutrition and Food Science - Program in Integrative Nutrition and Complex Diseases, Texas A\&M University, College Station, TX 77843; e-mail: r-chapkin@tamu.edu; ORCID ID: 0000-0002-6515-3898}\and Ivan Ivanov\thanks{Corresponding author: Department of Veterinary Physiology and Pharmacology, Texas A\&M University, College Station, TX 77843; e-mail: iivanov@cvm.tamu.edu; ORCID ID: 0000-0002-5323-8325}
}

\begin{document}
\maketitle
	
\begin{abstract}
We study a generalized discrete-time multi-type Wright-–Fisher population process.
The mean-field dynamics of the stochastic process is induced by a general replicator difference equation.
We prove several results regarding the asymptotic behavior of the model, focusing on the impact of the
mean-field dynamics on it. One of the results is a limit theorem that describes sufficient conditions
for an almost certain path to extinction, first eliminating the type which is the least fit at the mean-field equilibrium.
The effect is explained by the metastability of the stochastic system, which under the conditions of the theorem spends
almost all time before the extinction event in a neighborhood of the equilibrium.  In addition to limit theorems, we
propose a variation of Fisher's maximization principle, fundamental theorem
of natural selection, for a completely general deterministic replicator dynamics and study implications of the
deterministic maximization principle for the stochastic model.
\end{abstract}
\section{Introduction}
In this paper, we study a generalized multi-type Wright-Fisher process, modeling the
evolution of a population of cells or microorganisms in discrete time. Each cell belongs to one of the several given types,
and its type evolves stochastically while the total number of cells does not change with time.
The process is a time-homogeneous Markov chain that exhibits a rich and complex behavior.
\par
Due to the apparent universality of its dynamics, availability of effective approximation schemes, and its relative amenability
to numerical simulations, the Wright-Fisher process constitutes a popular modelling tool in applied biological research.
In particular, there is a growing body of literature using the Wright-Fisher model of evolutionary game theory \cite{wfnowak}
and its variations along with other tools of game theory and evolutionary biology to study compartmental models of
complex microbial communities and general metabolic networks \cite{mgame9, mgame7, mreview, sinergy3, mb5wf, mgame3}.
For instance, recent studies \cite{mgame3} incorporate evolutionary game theory and concepts from behavioral economics into their hybrid models of metabolic
networks including microbial communities with so-called Black Queen functions. See also \cite{ad4, mgame7, ad1} for related research.
The Wright-Fisher model has been employed to study the evolution of host-associated microbial communities \cite{mb5wf}.
In a number of studies, the Wright-Fisher model was applied to examine adaptive dynamics in living organisms, for instance viral and immune populations perpetually adapting \cite{stat47} and response to T-cell-mediated immune pressure \cite{stat24, stat21} in HIV-infected patients. Another interesting application of the generalized Wright-Fisher model to the study of evolution of viruses has been previously reported \cite{stat45}.
\par
Other biological applications of the Wright-Fisher model supported by experimental data include, for instance, substitutions in protein-coding genes \cite{stat28}, analysis of the single-nucleotide polymorphism differentiation between populations \cite{stat35}, correlations between genetic diversity in the soil and above-ground population \cite{stat46}, and persistence of decease-associated gene alleles in wild populations \cite{wfbio}. In cancer-related research, the Wright-Fisher model has been used to model tumor progression \cite{wfc, wfc1},
clonal interference \cite{stat18, clonal}, cancer risk across species (Peto's paradox) \cite{petopgenes}, and intra-tumoral heterogeneity \cite{iwasa}. Not surprisingly,
the Wright-Fisher model is also extensively used to model phenomena in the theoretical evolutionary biology \cite{stat33, indirect, stat62, iwf, dnoise1, stat19, wfctunelling, stat63, wfmratchet}.
\par
We now describe the generalized Wright-Fisher model and state our results in a semi-formal setting. For a population with $N$ cells, we define $M$-dimensional vectors
\beq
\bZn_k:=\bigl(\Zn_k(1),\Zn_k(2),\ldots,\Zn_k(M)\bigr), \qquad k=0,1,2,\ldots,
\feq
where $\Zn_k(i)$ represents the number of type $i$ cells in the generation $k.$
We denote by
\beq
\bXn_k:=\bigl(\Xn_k(1),\Xn_k(2),\ldots,\Xn_k(M)\bigr)
\feq
the frequency vectors of the population. That is,
\beqn
\label{s1}
\Xn_k(i)=\frac{\Zn_k(i)}{N}=\frac{1}{N}\cdot \#\{\mbox{\rm cells of type $i$ in the $k$-th generation}\},
\feqn
for all $1\leq i\leq M.$ We refer to the frequency vector $\bXn_k$ as the \emph{profile of the population at time $k.$}
\par
With biological applications in mind, the terms ``particles" and ``cells" are used interchangeably throughout this paper.
We assume that the sequence $\bZn_k$ forms a time-homogeneous Markov chain with a Wright-Fisher frequency-dependent transition mechanism.
That is, conditioned on $\bZn_k,$ the next generation vector $\bZn_{k+1}$ has a multinomial distribution with
the profile-dependent parameter $\bGamma\big(\bXn_k\big),$ where $\bGamma$ is a vector field that shapes the fitness landscape
of the population (see Section~\ref{forma} for details). In Example~\ref{exm} we list several examples of the fitness taken
from either biological or biologically inspired mathematical literature and briefly indicate a rationale behind their introduction.
\par
It is commonly assumed in the evolutionary population biology that particles reproduce or adopt their behavior according to their fitness,
which depends on the population composition through a parameter representing utility of a random interaction within the population.
In Section~\ref{filand}, in  order  to  relate  the  stochastic dynamics  of  the  generalized  Wrighte-Fisher  model
to \textit{first biological principles}, we interpret the mean-field model in terms of the reproductive fitness of the stochastic process
and study some basic property of the fitness. Our main contribution here is a variant of Fisher's evolutionary maximization principle
(``fundamental theorem of natural selection") for a general deterministic replicator dynamics.
An informal statement of this contribution is as follows. We refer to Theorem~\ref{mdns} for a formal statement.
We refer to a vector $\bx$ in $\rr^m$ as a \textit{population profile vector} if all its components are non-negative and their sum is one (i.\,e.,
$\bx$ is a probability vector).
\begin{theoremc}
\label{ac}
Assume that $\bGamma$ is a continuous vector-field. Then there exists a ``nice" (non-decreasing, bounded, complete Lyapunov, see Definition~\ref{lyapunov})
function $h$ mapping population profiles into the set of reals $\rr,$ such that:
\item [(i)]  $h(\bx)$ is the average $\sum_i x(i)\varphi_i(\bx)$ of an ordinal
Darwinian fitness given by
\beqn
\label{ford}
\varphi_i(\bx)=\frac{\Gamma_i(\bx)}{x(i)}h(\bx),\qquad \qquad i=1,\ldots,M.
\feqn
\item [(ii)] $h\big(\bGamma(\bx)\big)\geq h(\bx)$ for any population profile $\bx.$
\item [(iii)] $h\big(\bGamma(\bx)\big)>h(\bx)$ for any population profile $\bx$ which is, loosely speaking, ``not recurrent" (see Definition~\ref{conley} for the
definition of the ``recurrence").
\end{theoremc}
The theorem is a simple consequence of a ``fundamental theorem of dynamical systems" due to Conley \cite{conley}, see Section~\ref{arfa} for details.
\par
Let $\bx$ be an arbitrary population profile, set $\bpsi_0=\bx,$ and define recursively,
\beqn
\label{psi}
\bpsi_{k+1}=\bGamma(\bpsi_k),\qquad \qquad k\geq 0.
\feqn
We will occasionally use the notation $\bpsi_k(\bx)$ to indicate the initial state $\bpsi_0=\bx.$ It turns out that dynamical system $\bpsi_k$ serves as
mean-field approximation for the stochastic model (see Section~\ref{forma} and Theorem~\ref{da} for details).
It follows from Theorem~\ref{ac} that the average fitness $h$ is non-decreasing along the orbits of $\bpsi_k:$
\beq
h(\bpsi_{k+1})\geq h(\bpsi_k),
\feq
and, moreover, (again, informally) ``typically" we have a strong inequality $h(\bpsi_{k+1})>h(\bpsi_k).$
Furthermore, since $h$ is a bounded function,
\beq
\lim_{k\to\infty} h(\bpsi_k)=\sup_k h(\bpsi_k)
\feq
exists for all initial conditions $\bpsi_0.$
\par
In fact, the average Darwinian fitness generically remains constant during the evolution of any
Wright-Fisher model driven by a replicator mean-field dynamics (see Section~\ref{r} for details). Therefore,
from the perspective view of Fisher fundamental theorem, it is desirable to develop an alternative notion of fitness.
To that end, Theorem~\ref{ac} offers a distinguished ordinal utility that is compatible with the Fisher's fundamental theorem.
While the construction of this fitness is not explicit enough to be of a practical significance in applications,
the identity of this distinguished fitness is implied to be used by Nature in the way predicted
by Fisher's fundamental principle. Given that multiplication of the fitness vector-function by any scalar function does not change the
model, the fitness is an ordinal concept in our model (see Section~\ref{r} for details). This is seemingly paradoxical with respect to Theorem~\ref{ac} choosing a specific (cardinal) fitness. However, we argue that the fitness
is a more fundamental and intrinsic concept  than the evolution of phenotype described by our model, and while any ordinal fitness would yield the same update rule $\bGamma,$ there
exists a specific fitness form (perhaps more then one) that should be considered as \textit{the fitness} of the system in order to ensure that Fisher's maximization principle holds true.
\par
In contrast to the most of the known results of this type, our maximization principle 1) is genuinely universal,
such that it can be applied to any continuous replicator dynamics; and 2) refers to maximization of the entire reproduction fitness, rather than of its part.
Though our result is practically an existence theorem, namely it rather asserts the existence of a proper fitness function than offers its
explicit construction, we believe it makes an important contribution to the evolutionary population theory by establishing a direct
link (biologically meaningful and mathematically rigorous) between Wright's metaphor of fitness landscapes and Fisher's theorem of natural selection,
two of the most influential and arguably most controversial concepts of the evolutionary biology \cite{birch, fisher5,
fmoran, fisher11, elf, gavr, rugged, forr, fitn1}.
\par
In Section~\ref{exits} we study fixation probabilities of a certain class of Wright-Fisher models, induced by
a classical linear-fractional replicator mean-field dynamics with a symmetric payoff matrix which corresponds to a
partnership normal-form game with a unique evolutionary stable  equilibrium.
Understanding the route to fixation of multivariate population models is of a great importance to many branches of biology
\cite{app17, gavr, iwf, bionowak, rugged, wfmratchet}. For instance, extinction risk of biologically important components,
regulated by metabolism in living cells, has been the focus
of intensive research and is a key factor in numerous biotechnologically and biomedically relevant applications \cite{e,e1}.
Our contribution can be summarized as follows.  Firstly, Theorem~\ref{xit} offers an interpretation of the notion
``the least fit" which is in the direct accordance to the law of competitive exclusion and Fisher's fundamental theorem of natural selection.
Secondly, on a more practical note, the theorem suggests that for the underlying class of population models,
the path to extinction is almost deterministic when the population size is large. We remark that for a different
class of models, a similar fixation mechanism has been described in \cite{assaf, assaf1, metapark}.
However, our proof is based on a different idea, and is mathematically rigorous in contrast to the previous work
(where, in order to establish the result, analytical methods are combined with some heuristic arguments).
\par
Informally, the result in Theorem~\ref{xit} can be stated as follows:
\begin{theoremc}
\label{cb}
Suppose that $\bGamma$ satisfies the following set of conditions:
\begin{itemize}
\item [(i)] There is $\bmz \in\rr^M$ such that $\mz(i)>0$ for all $i=1,\ldots,M$ and
$\lim_{k\to\infty} \bpsi_k=\bmz$ whenever $\psi_0(i)>0$  for all $i=1,\ldots,M.$
\item [(ii)] There are no mutations, that is $\Gamma_i(x)=0$ if $x(i)=0.$
\item [(iii)] The Markov chain $\bXn_k$ is metastable, that is it spends a long time
confined in a small neighborhood of the deterministic equilibrium $\bmz$ before a large
(catastrophic) fluctuation occurs and one of the types gets instantly extinct.
\end{itemize}
Then, the probability that the least fit at the equilibrium $\bmz$ type will extinct first converges to one when $N$ goes to infinity.
\end{theoremc}
\textit{Metastability} or \textit{stable coexistence} is described in \cite{negac} as ``the long-term persistence of multiple competing species,
without any species being competitively excluded by the others, and with each species able to recover from perturbation to low density."
Various mathematical conditions equivalent to a variant of the above concept can be found in \cite{leto,bovier,landim}, see also \cite{assaf}.
\par
The benefit of the kind of results stated in Theorem~\ref{cb} for applied biological research is that they imply that the ultimate outcome of
a complex stochastic evolution is practically deterministic, and thus can be robustly predicted in a general accordance with
a simple variation of the \textit{competitive exclusion principle}.
\par
In Section~\ref{exits}, in addition to our main result stated in Theorem~\ref{xit}, we prove an auxiliary Proposition~\ref{radius} which complements
 \cite{losakin} and explains the mechanism of  convergence to the equilibrium under the conditions of Theorem~\ref{xit}.
The proposition establishes a strong contraction near the equilibrium property for the systems considered in Theorem~\ref{xit}.
The claim is a ``reverse engineering" in the sense that it would immediately imply Theorem~\ref{cb}-(i), whereas the latter has been
obtained in \cite{losakin} in a different way and is used here to prove the claim in Proposition~\ref{radius}.
\par
The Wright-Fisher Markov chains are known to exhibit complex multi-scale dynamics, partially because of the non-smooth boundary of their natural phase space. The analytic intractability of Wright-Fisher models has motivated an extensive literature on the approximation of these stochastic processes \cite{scale, ewens, fitn, iwf, app1}. These approximations ``illuminate the evolutionary process by focusing attention on the appropriate scales for the parameters, random variables, and time and by exhibiting the parameter combinations on which the limiting process depends" \cite{app1}. Our results in Section~\ref{exits} concerning the model's path to extinction rely partially on the Gaussian approximation constructed in Theorem~\ref{da}
and error estimates proved in Section~\ref{dea}. The construction is fairly standard and has been used to study other stochastic population models.
Our results in Sections~\ref{infi} and~\ref{dea} can be informally summarized as follows:
\begin{theoremc}
Assuming that the vector field $\bGamma$ is smooth enough, the following is true:
\item [(i)] $\bXn_k=\bpsi_k+\frac{\bun_k}{\sqrt{N}},$ where $\bpsi_k$ is the deterministic mean-field sequence defined in \eqref{psi}
(in particular, independent of $N$) and $\bun_k$ is a random sequence that, as $N$ goes to infinity, converges in distribution to a solution of certain random linear recursion.
\item [(ii)] For a fixed $N$ and any given $\veps>0,$ with an overwhelmingly large probability, the norm of the noise perturbation vector $\frac{\bun_k}{\sqrt{N}}$ will not exceed $\veps$
for an order of $e^{\alpha \veps^2 N}$ first steps of the Markov chain, where $\alpha>0$ is a constant independent of $N$ and $\veps.$
\end{theoremc}
\par
The theorem implies that when the population size $N$ is large, the stochastic Wright-Fisher model
can be interpreted as a perturbation of the mean-field dynamical system $\bGamma$ with a small random noise.
\par
Due to the particular geometry of the phase space boundary, the stochastic dynamics of the Wright-Fisher model often exhibits structural patterns (specifically, invariant and quasi-invariant sets) inherited from the mean-field model. Appendix~A of this paper is devoted to the exploration of the peculiar connection between invariant sets of the mean-field dynamical system and the form of the equilibria of the stochastic model when mutation is allowed. The results in the appendix aim to contribute to the general theory of Wright-Fisher processes with mutations and are not used directly anywhere in the paper.
\par
The rest of the paper is organized as follows. The mathematical model is formally introduced in Section \ref{modell}. Section~\ref{filand} is concerned with the geometry of the fitness landscape and fitness maximization principle for the deterministic mean-field model.
Our main results for the stochastic dynamics are stated and discussed in Section \ref{mainr}, with proofs deferred to Section~\ref{proofs}. Appendix~A ties the geometric properties of the fitness to the form of stochastic equilibria of the Wright-Fisher process. In order to facilitate the analysis of stochastic persistence in the Wright-Fisher model considered in Section~\ref{mainr},  we survey in Appendix~B a mathematical theory of coexistence for a class of deterministic interacting species models, known as permanent (or uniform persistent)
dynamical systems.  Finally, conclusions are outlined in Section~\ref{conclude}.
\section{The model}
\label{modell}
The aim of this section is to formally present the generalized Wright-Fisher model and introduce some of the key
notions pertinent to the generalized framework. The section is divided into two subsections. Section~\ref{geom} introduces a broader geometric context and thus sets the stage for a formal definition of the model given in Section~\ref{forma}.
\subsection{Geometric setup, notation}
\label{geom}
Fix $M\in\nn$ and define $S_M=\{1,2,...,M\}.$ Let $\zz_+$ and $\rr_+$ denote, respectively, the set of non-negative integers and the set of non-negative reals. Given a vector $\bx\in\rr^M,$ we usually use the notation $x(i)$ to denote its $i$-th coordinate. The only exception is in the context of vector fields, whose $i$-th component will typically be indicated with the subscript $i$.
\par
Let $\be_i=(0,\ldots,0,1,0,\ldots,0)$ denote the unit vector in the direction of the $i$-th axis in $\rr^M.$ We define the $(M-1)$-dimensional simplex
\beq
\Delta_M=\Big\{\bx=\big(x(1),\ldots,x(M)\big)\in\rr_+^M\,:\,\sum_{i\in S_M} x(i)=1\Big\}
\feq
and denote by $V_M$ the set of its vertices:
\beqn
\label{vm}
V_M=\{\be_j: j\in S_M\}\subset \rr^M.
\feqn
$\Delta_M$ is a natural state space for the entire collection of stochastic processes $(\bXn)_{N\in\nn}$ of \eqref{s1}. We equip $\Delta_M$ with the usual Euclidean topology inherited from $\rr^M.$ For an arbitrary closed subset $\Delta$ of $\Delta_M,$ we denote by $\mbox{\rm bd}(\Delta)$ its topological boundary and by $\mbox{\rm Int}(\Delta)$ its
topological interior $\Delta\backslash\mbox{\rm bd}(\Delta).$ For any non-empty subset $J$ of $S_M$ we denote by $\Delta_{[J]}$ the simplex built on the vertices corresponding to
$J.$ That is,
\beq
\Delta_{[J]}=\Big\{\bx\in\rr_+^M\,:\,\sum_{i\in J} x(i)=1\Big\} \subset \mbox{\rm bd}(\Delta_M).
\feq
We denote by $\partial(\Delta_{[J]})$ the manifold boundary of $\Delta_{[J]}$ and by $\Delta_{[J]}^\circ$ the corresponding
interior $\Delta_{[J]}\backslash\partial(\Delta_{[J]}).$ That is,
\beq
\partial(\Delta_{[J]})=\big\{\bx\in \Delta_{[J]}:x(j)=0~\mbox{\rm for some}~j\in J\big\}
\feq
and
\beq
\Delta_{[J]}^\circ=\big\{\bx\in \Delta_{[J]}:x(j)>0~\mbox{\rm for all}~j\in J\big\}.
\feq
We write $\partial(\Delta_M)$ and $\Delta_M^\circ$ for, respectively, $\partial(\Delta_{[S_M]})$ and $\Delta_{[S_M]}^\circ.$ Note that
$\mbox{\rm  bd}(\Delta_{[J]})=\Delta_{[J]},$ and hence $\partial(\Delta_{[J]})\neq \mbox{\rm bd}(\Delta_{[J]})$ for any \emph{proper} subset $J$ of $S_M.$
\par
Furthermore, let
\beq
\dmn=\Bigl\{0,\frac{1}{N},\frac{2}{N},...,\frac{N-1}{N},1\Bigr\}^M.
\feq
Note that $\bXn_k\in\dmn$ for all $k\in\zz_+$ and $N\in\nn.$
For any set of indices $J\subset S_M$ we write
\beqn
\label{sjn}
\djn=\Delta_{[J]}\bigcap \dmn,\quad \partial(\djn)=\partial(\Delta_{[J]})\bigcap \dmn, \quad \djn^\circ=\Delta_{[J]}^\circ\bigcap \dmn.
\feqn
We simplify the notations $\partial(\Delta_{[S_M],N})$ and $\Delta_{[S_M],N}^\circ$ to
$\partial(\dmn)$ and $\dmn^\circ,$ respectively.
\par
We use the maximum ($L^\infty$) norms for vectors in $\Delta_M$ and functions in $C(\Delta_M),$ the space of real-valued continuous functions on $\Delta_M:$
\beq
\|\bx\|:=\max_{i\in S_M}|x(i)|~\,\mbox{\rm for}~\bx\in \Delta_M
\qquad \rnd\qquad
\|f\|:=\max_{\bx\in \Delta_M}|f(\bx)|~\,\mbox{\rm for}~f\in C(\Delta_M).
\feq
The Hadamard (element-wise) product of two vectors $\bu,\bv\in\rr^M$ is denoted by $\bu\circ \bv.$ That is, $\bu\circ \bv \in\rr^M$ and
\beqn
\label{ep}
(u\circ v)(i)=u(i)v(i)\qquad\qquad \forall\, i\in S_M.
\feqn
The usual dot-notation is used to denote the regular scalar product of two vectors in $\rr^M$.
\subsection{Formulation of the model}
\label{forma}
We now proceed with the formal definition of the Markov chain $\bXn_k.$ Let $P$ denote the underlying probability measure defined in some large probability space that includes all random variables considered in this paper. We use $E$ and $COV$ to denote the corresponding expectation and covariance operators. For a random vector $\bX=(X(1),X(2),\ldots,X(M))\in\rr^M,$ we set $E(\bX):=\big(E\big(X(1)\big),E\big(X(2)\big),\ldots,E\big(X(M)\big)\big).$
\par
Let $\bGamma:\Delta_M\to \Delta_M$ be a mapping of the probability simplex $\Delta_M$ into itself. Suppose that the  population profile at time $k\in\zz_+,$
is given by a vector $\bx=\bigl(x(1),\ldots,x(M)\bigr)\in\dmn,$ i.e.  $\bXn_k=\bx.$  Then the next generation's profile, $\bXn_{k+1}$,
is determined by a multinomial trial where the probability of producing an $i$-th type particle is equal to $\bGamma_i (\bx).$
In other words, $\bXn$ is a Markov chain whose transition probability matrix is defined as
\beqn
\label{model}
P\bigg(\bXn_{k+1}=\Bigl(\frac{j_1}{N},...,\frac{j_M}{N}\Bigr)\,\bigg|\,\bXn_k=\bx\bigg)=\frac{N!}{\prod\limits_{i=1}^M j_i !}
~~\prod\limits_{i=1}^M\Gamma_i(\bx)^{j_i}
\feqn
for all $k\in\zz_+$ and $\Bigl(\frac{j_1}{N},...,\frac{j_M}{N}\Bigr)\in \Delta_{M,N}.$  Note that for any state $\bx\in \dmn$ and time $k\in \zz_+$ we have
\beqn
\label{eq1}
E\big(\bXn_{k+1}\,\big|\,\bXn_k=\bx\big)=\bGamma(\bx)
\feqn
and
\beqn
\label{var1}
COV\big(\Xn_{k+1}(i)\Xn_{k+1}(j)\,\big|\,\Xn_k=x\big)=\left\{
\begin{array}{ll}
-\frac{1}{N}\Gamma_i(\bx)\Gamma_j(\bx)&\mbox{\rm if}~i\neq j,
\\
[4pt]
\frac{1}{N}\Gamma_i(\bx)\bigl(1-\Gamma_i(\bx)\bigr)&\mbox{\rm if}~i=j.
\end{array}
\right.
\feqn
Extending the terminology of \cite{app17} from a binomial to multinomial transition kernel, we will refer to
any Markov chain $\bXn$ on $\dmn$ that satisfies \eqref{model} as
a \textit{Generalized Wright-Fisher model associated with the update rule $\bGamma.$}
\par
We denote transition kernel of $\bXn$ on $\dmn$ by $P_N.$ That is, $P_N(\bx,\by)$ is the conditional probability in \eqref{model} with $y(i)=j_i/N.$
For $k\in\nn,$ we denote by $P_N^k$ the $k$-th iteration of the kernel. Thus, $$P_N^k(\bx,\by)=P\big(\bXn_{m+k}=\by\,\big|\,\bXn_m=\bx\big)$$
for any $m\in\zz_+.$
\par
The generalized Wright-Fisher model is completely characterized by the choice of the parameters $M,N$ and
the update rule $\bGamma(\bx).$ Theorem~\ref{da} below states that when $N$ approaches infinity while $M$ and $\bGamma$ remain fixed, the update
rule emerges as the model's mean-field map, in that the path of the Markov chain $\bXn$ with $\bXn_0=\bx$ converges to the orbit of the (deterministic) discrete-time dynamical system
$\bpsi_k$ introduced in \eqref{psi}. Thus, for the large population size $N,$ the Wright-Fisher process can be perceived as a small-noise perturbation of the deterministic
dynamical system $\bpsi_k.$
\par
Throughout the paper we assume that the mean-field dynamics is induced by a replicator equation
\beqn
\label{replica}
\Gamma_i(\bx)=\frac{x(i)\varphi_i(\bx)}{\sum_{j\in S_M}x(j)\varphi_j (\bx)}\qquad \qquad  \forall\, i\in S_M,
\feqn
where the vector-field $\bvarphi:\Delta_M\to\rr_+^M$ serves as a fitness landscape of the model (see Section~\ref{filand} for details). For biological motivations of this definition, see \cite{ctao, rewhs, zig}. The replicator dynamics can be viewed as a natural discretization
scheme for its continuous-time counterpart described in the form of a differential equation \cite{garay, HS}. The normalization factor $\sum_{j\in S_m} x(j) \varphi_j(\bx)$ in \eqref{replica} is the total population fitness and plays a role similar to that of the partition function in statistical mechanics \cite{statph1}.
\begin{example}
\label{exm}
\item [(i)] One standard choice of the fitness $\bvarphi$ in the replicator equation \eqref{replica} is
\cite{hoffa, HS, iwf, wfnowak}:
\beqn
\label{lf4}
\bvarphi (\bx)=(1-\omega)\bb+\omega \bA \bx,
\feqn
where $\bA=(A_{ij})_{i,j\in S_M}$ is an $M\times M$ payoff matrix,
$\bb$ is a constant $M$-vector, and $\omega\in (0,1)$ is a selection parameter. The evolutionary game theory
interprets $(Ax)(i)=\sum_{j\in S_M}A_{ij}x(j)$ as the expected payoff of a type $i$ particle in a single interaction (game) with a
particle randomly chosen from the population profile $\bx$ when the utility of the interaction with a type $j$ particle is $A_{ij}.$
The constant vector $\bb$ represents a baseline fitness and the selection parameter $w$ is a proxy for modeling the strength
of the effect of interactions on the evolution of the population compared to inheritance \cite{wfnowak}.
\item [(ii)] In replicator models associated with multiplayer games \cite{nlinear, multip},
the pairwise interaction of case (i) (two-person game) is replaced by a more complex interaction within a randomly
selected group (multiplayer game). In this case, the fitness $\bvarphi$ is typically a (multivariate) polynomial of $\bx$ of a degree higher than one.
\item  [(iii)] An evolutionary game dynamics with non-uniform interaction rates has been considered in \cite{taylor}.
Mathematically, their assumptions amount to replacing \eqref{lf4} with
\beq
\varphi_i(\bx)=(1-\omega)b(i)+\omega \frac{Bx(i)}{Rx(i)},
\feq
where $\bR$ is the matrix of rates reflecting the frequency of pair-wise interaction between different types, and $\bB=\bA\circ \bR,$ the Hadamard product
of the payoff matrix $\bA$ and $\bR.$
\item [(iv)] Negative frequency dependence of the fitness can be captured by, for instance, allowing matrix $\bA$ in \eqref{lf4} to have negative entries \cite{negac}.
\item [(v)] A common choice of a non-linear fitness is an exponential function \cite{HS, econegt}:
\beqn
\label{ef4}
\varphi_i(\bx)=e^{\beta Ax(i)},
\feqn
where, as before, $\bA$ is a payoff matrix, and $\beta>0$ is a ``cooling" parameter (analogue of the inverse temperature in statistical mechanics
\cite{statph1}).
\end{example}
\section{Mean-field model: fitness landscape}
\label{filand}
In Section~\ref{r} we relate the update rule $\Gamma$ to a family of fitness landscapes of the mean-field model in the absence of mutations.
All the fitness functions in the family yield the same \textit{relative fitness} and differ from each other only by a scalar normalization function.
Section~\ref{arfa} contains a version of Fisher's maximization principle for the deterministic mean-field model. In Section~\ref{mutants}, we incorporate mutations into the framework presented in Sections~\ref{r} and~\ref{arfa}.
\par
The main result of this section is Theorem~\ref{mdns}, a variation of Fisher's fundamental theorem for a general deterministic replicator dynamics.
The theorem asserts the existence of a non-constant reproductive fitness which is maximized during the evolutionary process. The result is an immediate consequence of a version of Conley's ``fundamental theorem of dynamical systems" \cite{conley}, a celebrated mathematical result that appears to be underutilized in biological applications. The maximization principle stated in Theorem~\ref{mdns} applies to any continuous replicator dynamics, and it establishes a direct link between Wright's metaphor of fitness landscapes and Fisher's theorem of natural selection.
\subsection{Reproductive fitness and replicator dynamics}
\label{r}
The main goal of this section is to introduce a notion of fitness suitable for supporting a formalization of Fisher's fitness maximization principle
stated below in Section~\ref{arfa} as Theorem~\ref{mdns}. Proposition~\ref{agree} provides a useful characterization of the fitness, while Theorem~\ref{gt} (which is a citation from \cite{garay}) addresses some of its more technical properties, all related to the characterization of the support of the fitness function.
\par
Let $\calc:\Delta_M\to 2^{S_M}$ be the \textit{composition set-function} defined as
\beqn
\label{calf}
\calc(\bx)=\{j\in S_M: x(j)>0\}=\mbox{\rm ``support of vector $\bx$"}.
\feqn
For instance, $C\big((0,1/2,1/3,1/6,0)\big)=\{2,3,4\}$ and  $C\big((0,0,0,1/2,1/2)\big)=\{4,5\}.$
\par
According to \eqref{eq1}, for a population profile $\bx\in \dmn,$ the $M$-dimensional vector $\bff(\bx)$ with $i$-th component equal to
\beqn
\label{fdar}
f_i(\bx)=\frac{\Gamma_i(\bx)}{x(i)}\qquad \mbox{\rm if} \quad i\in \calc(\bx)
\feqn
represents the \emph{Darwinian fitness} of the profile $\bx$ \cite{fitn, forr}, see also \cite{fmoran, fitn3}. We will adopt the terminology of \cite{kimurac,fitco} and call a vector field $\bvarphi: \Delta_M\to \rr_+^M$ a \textit{reproductive fitness} of the population if for all $\bx\in \Delta_M$ we have:
\beqn
\label{fit}
\sum_{j\in \calc(\bx)}x(j)\varphi_j (\bx)>0 \qquad \rnd \qquad
\Gamma_i(\bx)=\frac{x(i)\varphi_i(\bx)}{\sum_{j\in S_M}x(j)\varphi_j (\bx)}\quad  \forall\, i\in \calc(\bx).
\feqn
Notice that even though we assume that both $\bvarphi(\bx)$ and $\bff(\bx)$ are vectors in $\rr_+^M,$ the definitions \eqref{fdar} and \eqref{fit} only consider $j\in \calc(\bx),$
imposing no restriction on the rest of the vectors components.
\par
Both concepts of fitness are pertinent to the amount of reproductive success of particles of a certain type per capita of the population. We remark that in \cite{pgoods},
$\bvarphi$ is referred to as the \textit{reproductive capacity} rather than the reproductive fitness.
A similar definition of reproductive fitness, but with the particular normalization $\bx\cdot \bvarphi=1$ and the restriction
of the domain of $\bGamma$ to the interior $\Delta_M^\circ,$ is given in \cite[p.~160]{econegt}.
\par
In the following straightforward extension of Lemma~1 in \cite{kimurac}, we show that two notions of fitness essentially agree under the
following no-mutations condition:
\beqn
\label{nom}
\sum_{j\in \calc(\bx)} \Gamma_j(\bx)=1 \qquad \mbox{\rm or, equivalently,}\qquad \bGamma(\bx)\in \Delta_{[\calc(\bx)]}\quad \forall~\bx\in \Delta_M.
\feqn
\begin{proposition}
\label{agree}
The following is true for any given $\bGamma:\Delta_M\to\Delta_M:$
\item [(i)] A reproductive fitness exists if and only if condition \eqref{nom} is satisfied.
\item [(ii)] Suppose that condition \eqref{nom} is satisfied and let $h: \Delta_M\to \rr_+$ be an arbitrary non-negative
function defined on the simplex. Then any $\bvarphi:\Delta_M\to \rr_+^M$ such that
\beqn
\label{sex}
\varphi_i(\bx)=\frac{\Gamma_i(\bx)}{x(i)}h(\bx)\qquad \forall\, i\in \calc(\bx),
\feqn
is a reproductive fitness.
\item [(iii)] Conversely, if $\bvarphi: \Delta_M\to \rr_+^M$ is a reproductive fitness, then
\eqref{sex} holds for $h(\bx)=\bx\cdot \bvarphi(\bx).$
\item [(iv)] If $\bvarphi:\Delta_M\to\rr_+^M$ is a reproductive fitness, $\bx\in \Delta_M$ and $i\in \calc(\bx),$
then $\varphi_i(\bx)>0$ if and only if $\Gamma_i(\bx)>0.$
\end{proposition}
Note that \eqref{nom} is equivalent to the condition that $\bGamma\big(\Delta_{[I]}^\circ\big)\subset \Delta_{[I]}$ for all $I\subset S_M.$
Therefore, if $\bGamma$ is a continuous map, then \eqref{nom} is equivalent to the condition
\beq
\bGamma\big(\Delta_{[I]}\big)\subset \Delta_{[I]} \qquad \qquad \forall\, I\subset S_M.
\feq
For continuous fitness landscapes, the following refined version of Proposition~\ref{agree} has been obtained in \cite{garay}.
Recall that $\bGamma$ is called a homeomorphism of $\Delta_M$ if it is a continuous bijection of $\Delta_M$ and its inverse function is also
continuous. Intuitively, a bijection $\bGamma:\Delta_M\to\Delta_M$ is a homeomorphism when the distance $\|\bGamma(x)-\bGamma(y)\|$ is
small if and only if $\|\bx-\by\|$ is small. $\bGamma$ is a diffeomorphism  if it is a homeomorphism and both $\bGamma$ and $\bGamma^{-1}$ are continuously differentiable.
\begin{theorema}
[\cite{garay}]
\label{gt}
Suppose that $\bGamma=\bx\circ \bvarphi$ for a continuous $\bvarphi:\Delta_M\to\rr_+^M$ and $\bGamma$ is a homeomorphism of $\Delta_M.$
Then:
\item [(i)] We have:
\beq
\Gamma\big(\Delta_{[I]}^\circ\big)=\Delta_{[I]}^\circ \quad \rnd \quad \Gamma\big(\partial(\Delta_{[I]})\big)=\partial\big(\Delta_{[I]}\big),\qquad  \forall\,I\subset S_M.
\feq
\item [(ii)] If, in addition, $\bGamma$ is a diffeomorphism of $\Delta_M$ and each $\varphi_i,$ $i\in S_M,$ can be extended to an open neighborhood
$U$ of $\Delta_M$ in $\rr^M$ in such a way that the extension is continuously differentiable on $U,$ then $\varphi_i(\bx)>0$ for all $\bx\in\Delta_M$ and
$i\in S_M.$
\end{theorema}
The first part of the above theorem is the Surjectivity Theorem on p.~1022 of \cite{garay} and the second part is Corollary~6.1 there.
The assumption that $\bGamma$ is a bijection, i.\,e. one-to-one and onto on $\Delta_M,$ seems to be quite restrictive in biological applications \cite{fisher5, landscape, gavr, rugged}. Nevertheless, the theorem is of a considerable interest for us because it turns out that its conditions are satisfied for $\bGamma$ in \eqref{lf4} and \eqref{ef4}.
\subsection{Fisher's fundamental theorem for the reproductive fitness}
\label{arfa}
The aim of this section is to introduce a variation of Fisher's fundamental theorem, suitable for the deterministic replicator
dynamics and reproductive fitness. The result is stated in Theorem~\ref{mdns}. An example of application to the stochastic Wright-Fisher model
is given in Proposition~\ref{cava}.
\par
For the Darwinian fitness we have $\ol \bff(\bx):=\sum_{i\in S_M} x_if_i(\bx)=
\sum_{i\in S_M} \Gamma_i(\bx)=1.$ Thus the average Darwinian fitness is independent of the population profile, and in particular doesn't change with time.
In contrary, the average reproductive fitness $\ol \bvarphi(\bx):= \sum_{i\in S_M} x_i\varphi_i(\bx)$ is equal to the (arbitrary) normalization factor $h(\bx)$ in \eqref{sex},
and hence in principle can carry more structure. If the mean-field dynamics is within the domain of applicability of Fisher's fundamental
theorem of natural selection, we should have \cite{birch, fisher11, elf, fitn1, ska}:
\beqn
\label{sc}
h\big(\bGamma(\bx)\big)\geq h(\bx)\qquad \forall\,\bx\in\Delta_M.
\feqn
In order to elaborate on this point, we need to introduce the notion of chain recurrence.
\begin{definition}
[\cite{conley}]
\label{conley}
An $\veps$-chain from $\bx\in \Delta_M$ to $\by\in\Delta_M$ for $\bGamma$ is a sequence of points in $\Delta_M,$
$\bx = \bx_0, \bx_1, \ldots , \bx_n = \by,$ with $n\in\nn,$ such that
$\big\|\bGamma(\bx_i)-\bx_{i+1}\big\| < \veps$ for $0 \leq i \leq n-1.$
A point $\bx \in \Delta_M$ is called chain recurrent if for every $\veps > 0$ there is an $\veps$-chain
from $x$ to itself. The set $\calr(\bGamma)$ of chain recurrent points is called the chain recurrent set of $\bGamma.$
\end{definition}
Write $\bx \looparrowright \by$ if for every $\veps>0$ there is an $\veps$-chain from $\bx$ to $\by$ and  $\bx\sim \by$ if both $\bx \looparrowright \by$ and $\by \looparrowright \bx$ hold true.
It is easy to see that $\sim$ is an equivalence relation on $\calr(\bGamma).$ For basic properties of the chain recurrent set we refer the reader to \cite{conley} and \cite{p3,am, p1}. In particular,  Conley showed (see Proposition~2.1 in \cite{p1}) that equivalence classes of $\calr(\bGamma)$ are precisely its connected components.
\par
Somewhat surprisingly, the following general result is true:
\begin{theorem}
\label{mdns}
Assume that $\bGamma:\Delta_M\to\Delta_M$ is continuous. Then there exists a reproductive fitness function $\bvarphi$
such that
\begin{itemize}
\item [(i)] $\bvarphi$ is continuous in $\Delta_{[I]}^\circ$ for all $I\subset \Delta_M.$
\item [(ii)] The average $h(\bx)=\bx\cdot \bvarphi(\bx)$ satisfies \eqref{sc}
and, moreover, is a complete Lyapunov function for $\bGamma$ (see Definition~\ref{lyapunov}).
\end{itemize}
In particular, the difference $h\big(\bGamma(\bx)\big)-h(\bx)$ is strictly increasing on $\Delta_M\backslash \calr(\bGamma).$
\end{theorem}
Theorem~\ref{mdns} is an immediate consequence of a well-known mathematical result which
is sometimes called a ``fundamental theorem of dynamical systems"  \cite{franks, norton, fisheri1}. The following definition and theorem are adapted from \cite{hurley1}. To simplify notation, in contrast to the conventional definition, we define Lyapunov function as a function non-decreasing (rather than non-increasing) along the orbits of $\bpsi_k.$
\begin{definition}
\label{lyapunov}
A complete Lyapunov function for $\bGamma:\Delta_M\to\Delta_M$ is a continuous function $h : \Delta_M \to [0,1]$
with the following properties:
\item [1.] If $\bx\in \Delta_M\backslash\calr(\bGamma),$ then $h\big(\bGamma(\bx)\big) > h(\bx).$
\item [2.] If $\bx, \by \in\calr(\bGamma)$ and $ \bx \looparrowright \by,$ then $h(\bx)\leq h(\by).$ Moreover, $h(\bx) = h(\by)$ if and only if $\bx \sim \by.$
\item [3.] The image $h\big(\calr(\bGamma)\big)$ is a compact nowhere dense subset of $[0,1].$
\end{definition}
\begin{theorema}
[Fundamental theorem of dynamical systems]
\label{mt}
If $\bGamma:\Delta_M\to\Delta_M$ is continuous, then there is a complete Lyapunov function for $\bGamma.$
\end{theorema}
\par
To derive Theorem~\ref{mdns} from this fundamental result, set $\varphi_i(\bx)=\frac{\Gamma_i(\bx)}{x_i}h(\bx)$ for $i\in \calc(\bx)$
and, for instance, $\varphi_i(\bx)=0$ for $i\not\in\calc(\bx).$
\par
Theorem~\ref{mt} was established by Conley for continuous-time dynamical systems on a compact space \cite{conley}.
The above discrete-time version of the theorem is taken from \cite{hurley1}, where it is proven for an arbitrary separable metric state space.
For recent progress and refinements of the theorem see \cite{dns1, dns6, dns3} and references therein.
\par
The Lyapunov functions constructed in \cite{conley} and \cite{hurley1} map $\calr(\bGamma)$ into a subset of the Cantor middle-third set. Thus,
typically, $\Delta_M\backslash \calr(\bGamma)$ is a large set (cf. Section~6.2 in \cite{conley}, Theorem~A in \cite{p3}, and an extension of Theorem~\ref{mt} in \cite{am}). For example, for the class of Wright-Fisher models discussed in Section~\ref{exits}, $\calr(\bGamma)$ consists
of the union of a single interior point (asymptotically stable global attractor in $\Delta_M^\circ$) and the boundary of $\Delta_M$ (repeller).
For a related more general case, see Appendix~B and Example~\ref{price} below.
\par
Recall that a point $\bx\in\Delta_M$ is called recurrent if the orbit $(\bpsi_k)_{k\in\zz_+}$ with $\bpsi_0=\bx$ visits  any open neighborhood of $\bx$ infinitely often. Any recurrent point is chain recurrent, but the converse is not necessarily true. It is easy to verify that the strict inequality
$h(\Gamma(\bx)) < h(\bx)$ holds for any recurrent point $\bx$. The following example builds upon this observation.
\begin{example}
The core of this example is the content of Theorem~4 in \cite{rfit}. Assume that \eqref{sc} holds for a continuous function $h:\Delta_M\to\rr_+$ in \eqref{sex}.
Suppose in addition that the mean-field dynamics is chaotic. More precisely, assume that $\bGamma$ is topologically transitive, that is
$\bGamma$ is continuous and for any two non-empty open subsets $U$ and $V$ of $\Delta_M,$ there exists $k\in\nn$ such that $\bpsi_k(U)\bigcap V\neq \emptyset$
(this is not a standard definition of a chaotic dynamical system, we omitted from a standard Devaney's set of conditions
a part which is not essential for our purpose, cf. \cite{deva, rfit, silver}). Equivalently \cite{silver}, there is $\bx_0\in\Delta_M$ such that the orbit
$\big(\bpsi_k(\bx_0)\big)_{k\in\zz_+}$ is dense in $\Delta_M.$
\par
Let $\bx=\bpsi_k(\bx_0)$ be an arbitrary point on this orbit. To get a contradiction,  assume that $\veps:=\big[h\big(\bGamma(\bx)\big)-h(\bx)\big]>0.$ Let $\delta>0$ be so small that $|h(\bx)-h(\by)|<\veps/2$
whenever $\|\bx-\by\|<\delta.$ By our assumption, there exists $m>k$ such that $\|\bpsi_m(\bx_0)-\bx\|<\delta,$ and hence $h\big(\bpsi_m(\bx_0)\big)<h\big(\bGamma(\bx)\big)=h\big(\bpsi_{k+1}(\bx_0)\big),$ which is impossible in view of the monotonicity condition \eqref{sc}. Thus $h$ is constant along the forward orbit of $\bx_0.$
Since the normalization function $h(\bx)$ is constant on a dense subset of $\Delta_M$ and $\Delta_M$ is compact, $h(\bx)$ is a constant function.
In other words, for a chaotic mean-field dynamics a constant function is the only choice of a continuous normalization $h(\bx)$
that satisfies \eqref{sc}.
\par
A common example of a chaotic dynamical system is the iterations of a one-dimensional logistic map. In our framework this can be translated into the two-dimensional
case $M=2$ and $\bGamma(x,y)=\big(4x(1-x), 1-4x(1-x)\big)=(4xy,1-4xy).$ In principle, the example can be extended to higher-dimensions using, for instance, techniques of \cite{chaose}.
\end{example}
\begin{example}
\label{price}
Suppose that $\bGamma$ is a replicator function defined in \eqref{replica}. If $\bA$ is an $M\times M$ matrix that satisfies the conditions of Theorem~\ref{ept} given in Appendix~B and $\bvarphi$ in \eqref{replica} is defined by either \eqref{lf4} with small enough $\omega>0$ or by \eqref{ef4}, then $\calr(\bGamma) \subset K\bigcup \partial(\Delta_M),$
where $K$ is a compact subset of $\Delta_M^\circ.$ In particular, the complete Lyapunov function $h$ constructed in Theorem~\ref{mt} is strictly increasing on the non-empty open set $\Delta_M^\circ \backslash K.$ Notice that while the average $\bx\cdot \bvarphi(\bx)$ may not be a complete Lyapunov function for $\bGamma,$ the average $h(\bx)$ of the reproductive fitness $\witi \bvarphi(\bx):=h(\bx)\bvarphi(\bx)$ is.
\end{example}
The existence of a non-decreasing reproductive fitness doesn't translate straightforwardly into a counterpart for the stochastic dynamics. Indeed, in the absence of mutations $\bXn_k$ converges almost surely to an absorbing state, i.\,e. to a random vertex of the simplex $\Delta_M.$ Therefore, by the bounded convergence theorem,
\beq
\lim_{k\to\infty} E\big(h(\bXn_k)\,\big|\,\bXn_0=\bx\big)=\sum_{j\in S_M}p_j(\bx)h(e_j),
\feq
where $p_j(\bx)=\lim_{k\to\infty} P\big(\bXn_k=\be_j\,\big|\,\bXn_0=\bx\big).$
It then follows from part 2 of Definition~\ref{lyapunov} that if the boundary $\partial\big(\dmn\big)$ is a repeller for
$\bGamma$ (cf. see Appendix~B), then generically the limit is not equal to the
$\sup_{k\in\nn} E\big(h(\bXn_k)\,\big|\,\bXn_0=\bx\big).$
\par
The following example shows that under some additional convexity assumptions, the stochastic average of a certain continuous reproductive fitness
is a strictly increasing function of time. More precisely, for the class of Wright-Fisher models in the example,
$h\big(\bXn_k)$ is a supermartingale for an explicit average reproductive fitness $h.$
\begin{example}
Let $h:\Delta_M\to\rr_+$ be an arbitrary convex (and hence, continuous) function that satisfies \eqref{sc},
not necessarily the one given by Theorem~\ref{mdns}. Then, by the multivariate Jensen's inequality \cite[p.~76]{fergi},
with probability one, for all $k\in\zz_+,$
\beqn
\label{infact}
E\big(h\big(\bXn_{k+1}\big)\,\big|\,\bXn_k\big)\geq h\big(E\big(\bXn_{k+1}\,\big|\,\bXn_k\big)\big)=h\big(\bGamma\big(\bXn_k\big)\big)\geq h\big(\bXn_k).
\feqn
Thus, $M_k=h\big(\bXn_k\big)$ is a submartingale in the filtration $\calf_k=\sigma\big(\bXn_0,\ldots,\bXn_k\big).$
In particular, the sequence ${\mathfrak f}_k=E(M_k)$ is non-decreasing and converges to its supremum as $k\to\infty.$ In fact, since the first inequality in \eqref{infact} is strict whenever
$\bXn_k\not\in V_M,$ the stochastic average of the reproductive fitness is strictly increasing under the no-mutations assumption \eqref{nom}:
\beq
0\leq E\big(h\big(\bXn_k\big)\big)<E\big(h\big(\bXn_{k+1}\big)\big)\leq 1
\feq
as long as $P\big(\bXn_0\in \dmn^\circ\big)=1.$
\par
The effect of the fitness convexity assumption on the evolution of biological populations has been discussed by several authors, see, for instance, \cite{c,c1,c3} and references therein.
Generally speaking, the inequality $h\big(\alpha \bx+(1-\alpha)\by\big)\leq \alpha h(\bx)+(1-\alpha)h(\by),$ $\alpha\in[0,1],$
$\bx,\by\in \Delta_M,$ that characterizes convexity, manifests an evolutionary disadvantage of the intermediate population $\alpha \bx+(1-\alpha)\by$
in comparison to the extremes $\bx,\by.$ A typical biological example of such a situation is competitive foraging \cite{c1}.
\end{example}
It can be shown (see Section~6 in \cite{hoffa} or \cite{losakin}) that if $\bA$ is a non-negative symmetric $M\times M$ matrix, then \eqref{sc} holds with
$h(\bx)=\bx^T\bA\bx$ for \eqref{lf4} and \eqref{ef4}. Here and henceforth, $\bA^T$ for a matrix $\bA$ (of arbitrary dimensions)
denotes the transpose of $\bA.$  Moreover, the inequality is strict unless $\bx$ is a game equilibrium. The evolutionary games with $\bA=\bA^T$ are referred to in \cite{hoffa} as partnership games, such games are a particular case of so-called potential games (see \cite{rewhs, econegt} and references therein).  In the case of the replicator mean-field dynamics \eqref{replica} and potential games, the potential function $h(\bx)$ coincides with the average reproductive fitness $\bx\cdot \bvarphi(\bx).$ Let
\beqn
\label{seq}
W_M=\Big\{\bx\in\rr^M:\sum_{i\in S_M} x(i)=0\Big\}.
\feqn
It is easy to check that if $\bA$ is positive definite on $W_M,$ that is
\beqn
\label{pdef}
\bw^T\bA\bw>0\qquad \qquad \forall\,\bw\in W_M,\bw\neq \bo,
\feqn
then $\bff(\bx)=\bx^T\bA\bx$ is a convex function on $\Delta_M,$ that is
\beq
\frac{1}{2}\bx^T\bA\bx+\frac{1}{2}\by^T\bA\by>\frac{(\bx+\by)^T}{2}\bA\frac{\bx+\by}{2}\qquad \qquad \forall\,\bx,\by\in \Delta_M,\bx\neq \by.
\feq
We thus have:
\begin{proposition}
\label{cava}
Suppose that $\bA$ is a symmetric invertible $M\times M$ matrix with positive entries such that \eqref{pdef} holds true (i.\,e. $\bA$ is positive-definite on $W_M$).
Then \eqref{infact} holds with $M_k=\big(\bXn_k\big)^T\bA\bXn_k.$
\end{proposition}
It is shown in \cite{losakin} that under the conditions of the proposition, $\bpsi_k$ converges, as $k$ tends to infinity, to an equilibrium
on the boundary of the simplex $\Delta_M$ (cf. Theorem~\ref{bsa} in Appendix~B). The proposition thus indicates that when $\bA$ is positive-definite on $W_M,$
the stochastic and deterministic mean-field model might have very similar asymptotic behavior when $N$ is large. This observation partially
motivated Conjecture~\ref{conja} stated below in Section~\ref{exits}.
\subsection{Incorporating mutations}
\label{mutants}
In order to incorporate mutations into the interpretation of the update rule $\bGamma$ in terms of the population's
fitness landscape, we adapt the Wright-Fisher model with neutral selection and mutation of \cite{stat62}. Namely, for the purpose
of this discussion we will make the following assumption:
\begin{assume}
\label{asm}
The update rule $\bGamma$ can be represented in the form $\bGamma(\bx)=\bUpsilon(\bx^T\bTheta),$ where
$\bx^T$ denotes the transpose of $\bx,$ and
\begin{enumerate}
\item $\bUpsilon: \Delta_M\to\Delta_M$ is a vector field that satisfies the ``no-mutations" condition
\beqn
\label{nm1}
x(j)>0~\Longrightarrow~\Upsilon_j(\bx)>0
\feqn
for all $\bx\in\Delta_M$ and $j\in S_M.$
\item $\bTheta=(\Theta_{ij})_{i,j\in S_M}$ is $M\times M$ stochastic matrix, that is $\Theta_{ij}\geq 0$ for all $i,j\in S_m$ and
$\sum_{j\in S_M}\Theta_{ij}=1$ for all $i\in S_M.$
\end{enumerate}
\end{assume}
The interpretation is that while the entry $\Theta_{ij}$ of the \textit{mutation matrix} $\bTheta$ specifies the probability of mutation
(mutation rate) of a type $i$ particle into a particle of type $j,$ the update rule $\bUpsilon$ encompasses the effect of evolutional forces other than mutation, for instance genetic drift and selection. The total probability of mutation of a type $i$ particle is $1-\Theta_{ii},$
and, accordingly, $\bGamma$ satisfies the no-mutation condition if $\bTheta$ is the unit matrix. Note that, since $\bTheta$ is assumed to be stochastic, the linear transformation
$\bx\mapsto \bx^T\bTheta$ leaves the simplex $\Delta_M$ invariant, that is $\bx^T\bTheta\in \Delta_M$ for all $\bx\in \Delta_M.$ Both the major implicit assumptions forming a foundation for writing $\bGamma$ as the composition of a non-mutative map $\bUpsilon$ and the action of a mutation matrix $\bTheta,$ namely that mutation can be explicitly separated from other evolutionary forces and that mutations happen at the beginning of each cycle before other evolutionary forces take effect, are standard \cite{mweaks, stat62, iwf} even though at least the latter is a clear idealization \cite{mb8fit}.
Certain instances of Fisher's fundamental theorem of natural selection have been extended to include mutations in \cite{fmuta, fishm}.
\par
We now can extend the definition \eqref{fit} as follows:
\begin{definition}
Let Assumption~\ref{asm} hold. A vector field $\bvarphi: \Delta_M\to \rr_+^M$ is called the reproductive fitness landscape of the model
if for all $\bx\in \Delta_M$ we have:
\beqn
\label{mfit}
\bx^T\bTheta\bvarphi (\bx)>0 \qquad \rnd \qquad
\Gamma_i(\bx)=\frac{(x^T\Theta)(i)\cdot \varphi_i(\bx)}{\bx^T\bTheta\bvarphi (\bx)}\quad  \forall\, i\in S_M,
\feqn
where $\bx^T$ is the transpose of $\bx.$
\end{definition}
An analogue of Proposition~\ref{agree} for a model with mutation reads:
\begin{proposition}
Let Assumption~\ref{asm} hold. Then $\bvarphi:\Delta_M\to\Delta_M$ is a reproductive fitness landscape if and only if
\beqn
\label{land}
\varphi_i(\bx)=
\left\{
\begin{array}{ll}
\frac{\Gamma_i(\bx)}{(x^T\Theta)(i)}h(\bx)~&~\mbox{\rm if}~i\in \calc\big(\bx^T\bTheta\big),\\
[4pt]
0~&~\mbox{\rm if}~i\not\in \calc\big(\bx^T\bTheta\big)
\end{array}
\right.
\feqn
for some function $h:\Delta_M\to (0,\infty).$
\end{proposition}
\begin{proof}
For the ``if part" of the proposition, observe that if $\bvarphi$ is given by \eqref{land}, then
\beq
\bx^T\bTheta\bvarphi (\bx)=h(\bx)\cdot\sum_{i\in \calc(\bx^T\bTheta)}\Gamma_i(\bx)=h(\bx)\cdot\sum_{i\in \calc(\bGamma(\bx))}\Gamma_i(\bx)=h(\bx),
\feq
where in the second step we used the fact that $\calc\big(\bGamma(\bx)\big)\subset \calc\big(\bx^T\bTheta\big)$ by virtue of \eqref{nm1}.
Thus $\bvarphi$ is a solution to \eqref{mfit}.
\par
For the ``only if part" of the proposition, note that if $\bvarphi$ is a reproductive fitness landscape,
then \eqref{mfit} implies that \eqref{land} holds with $h(\bx)=\bx^T\bTheta\bvarphi (\bx).$
\end{proof}
If the weighted average $\bx^T\bTheta\bvarphi (\bx)$ is adopted as the definition of the mean fitness,
Theorem~\ref{mdns} can be carried over verbatim to the setup with mutations.
\section{Stochastic dynamics: main results}
\label{mainr}
In this section we present our main results for the stochastic model. The proofs are deferred to Section~\ref{proofs}.
\par
First, we prove two different approximation results. The Markov chain $\bXn$, even for the classical two-allele haploid model of genetics, is fairly complex. In fact, most of the known results about the asymptotic behavior of Wright-Fisher Markov chains were obtained through a comparison to various limit processes, including the mean-field deterministic system, branching process, diffusions and Gaussian processes \cite{scale, ewens, fitn, iwf, app1}.
\par
In Section~\ref{infi} we are concerned with the so-called Gaussian approximation, which can be considered as an intermediate between the deterministic and the diffusion approximations. This approximation process can be thought of as the mean-field iteration scheme perturbed by an additive Gaussian noise. Theorem~\ref{da} constructs such an approximation for a generalized Wright-Fisher model. Results of this type are well-known in literature, and our proof of Theorem~\ref{da} is a routine adaptation of classical proof methods.
\par
In Section~\ref{dea}, we study the longitudinal error of the deterministic approximation by the mean-field model. Specifically, Theorem~\ref{th1} provides an exponential lower bound on the decoupling time, namely the first time when an error of approximation exceeds a given threshold. A variation of the result is used later in the proof of Theorem~\ref{xit}, the main result of this section.
\par
Conceptually, Theorem~\ref{da} is a limit theorem suggesting that for large values of $N,$ the stochastic component is the major factor in determination of the asymptotic behavior of the stochastic model. Theorem~\ref{dea} then complements this limit theorem by quantifying this intuition for the stochastic model of a given (though large) population size $N$ (versus the infinite-population-limit result of Theorem~\ref{da}).
\par
The impact of the mean-field dynamics on the underlying stochastic process is further studied in Section~\ref{exits}. The main result is Theorem~\ref{xit} which specifies the route to extinction of the stochastic system in a situation when the mean-field model is ``strongly permanent" (cf. Appendix~B) and has a unique interior equilibrium. In general, the problem of determining the route to extinction for a multi-type stochastic population model is of an obvious importance to biological sciences, but is notoriously difficult in practice.
\par
The proof of Theorem~\ref{xit} highlights a mechanism leading to an almost deterministic  route to extinction for the model. Intuitively, under the conditions of Theorem~\ref{xit} the system is trapped in a quasi-equilibrium stochastic state, fluctuating near the deterministic equilibrium for a very long time, eventually ``untying the Gordian knot", and instantly jumping to the boundary. A similar mechanism of extinction was previously described for different stochastic population models in \cite{assaf, assaf1, metapark}. We note that analysis in these papers involves a mixture of rigorous mathematical and heuristic arguments, and is different in nature from our approach.
\par
\subsection{Gaussian approximation of the Wright-Fisher model}
\label{infi}
Let $\overset{P}{\to}$ denote convergence in probability as the population size $N$ goes to infinity.
The following theorem is an adaptation of Theorems~1 and~3 in \cite{bdet} for our multivariate setup. We also refer to \cite{knerman, nagy, app1, dapr4} for earlier related results and to \cite[p.~527]{rao} for a rigorous construction of a degenerate Gaussian distribution.
\par
Recall $\bpsi_k$ from \eqref{psi}. We have:
\begin{theorem}
\label{da}
Suppose that
\beqn
\label{xnk1}
\bXn_0\overset{P}{\to}\bpsi_0
\feqn
for some $\bpsi_0\in \Delta_M.$ Then the following holds true:
\item [(a)]  $\bXn_k\overset{P}{\to} \bpsi_k$ for all $k\in\zz_+.$
\item [(b)] Let
\beq
\bun_k=\sqrt{N}\bigl(\bXn_n-\bpsi_k\bigr),\qquad N\in\nn,\,k\in\zz_+,
\feq
and $\bSigma(\bx)$ be the matrix $M\times M$ with entries (cf. \eqref{var1} in Section~\ref{forma})
\beq
\Sigma_{i,j}(\bx):=
\left\{
\begin{array}{ll}
-\Gamma_i(\bx)\Gamma_j(\bx)&\mbox{\rm if}~i\neq j,
\\
\Gamma_i(\bx)\bigl(1-\Gamma_i(\bx)\bigr)&\mbox{\rm if}~i=j
\end{array}
\right.,
\qquad \bx\in \Delta_M.
\feq
Suppose that in addition to \eqref{xnk1},  $\bGamma$ is twice continuously differentiable and $\bun_0$ converges weakly, as $N$ goes to infinity, to some (possibly random) $\bu_0.$
\par
Then the sequence $\bun:=\bigl(\bun_k\bigr)_{k\in\zz_+}$ converges in distribution, as $N$ goes to infinity,
to a time-inhomogeneous Gaussian $AR(1)$ sequence $(\bU_k)_{k\in\zz_+}$ defined by
\beqn
\label{ar1}
\bU_{k+1}=\bD_x(\bpsi_k)\bU_k+\bg_k,
\feqn
where $\bD_x(\psi_k)$ denotes the Jacobian matrix of $\bGamma$ evaluated at $\bpsi_k,$ and $\bg_k,$ $k\in\zz_+,$ are  independent degenerate Gaussian
vectors, each $\bg_k$ distributed as $N\bigl(0,\bSigma(\bpsi_k)\bigr).$
\end{theorem}
The proof of the theorem is included in Section~\ref{pda}. It is not hard to prove (cf. Remark (v) on p.~61 of \cite{bdet}, see also \cite{knerman})
that if $\bmz$ is a unique global stable point of $\bGamma,$ $\bpsi_0\overset{P}{\to} \bmz$ in the statement of Theorem~\ref{da}
and, in addition, $\bun_0$ converges weakly, as $N$ goes to infinity, to some $\bu_0,$ then the linear recursion \eqref{ar1} can be replaced with
\beqn
\label{vepr}
\bU_{k+1}=\bD_x(\bmz)\bU_k+\witi \bg_k,
\feqn
where $\witi \bg_k,$ $k\in\zz_+,$ are i.\,i.\,d. degenerate Gaussian vectors in $\rr^M,$ each $\witi \bg_k$ distributed as $N\bigl(0,\bSigma(\bmz)\bigr).$
One then can show that if the spectral radius of $\bD_x(\bmz)$ is strictly less than one, Markov chain $\bU_k$ has
a stationary distribution, see \cite[p.~61]{bdet} for more details. In the case when $\bGamma$ has a unique global stable point,
\eqref{vepr} was obtained in \cite{dapr4} by different from ours, analytic methods.
\subsection{Difference equation approximation}
\label{dea}
Theorem~\ref{da} indicates that when $N$ is large,
the trajectory of the deterministic dynamical system $\psi_k$ can serve as a good approximation to the path of the Markov chain $\bXn_k.$ The following
theorem offers some insight into the duration of the time for which the deterministic and the stochastic paths stay fairly close to each other,
before separating significantly for the first time. The theorem is a suitable modification of some results
of \cite{ozgur} in a one-dimensional setup.
\par
For $\veps>0,$ let
\beq
\tau_N(\veps)=\inf\bigl\{k\in\zz_+:\bigl\|\bXn_k-\bpsi_k\bigr\| >\veps \bigr\},
\feq
where $\bpsi_0=\bXn_0$ and $\bpsi_k$ is the sequence introduced in \eqref{psi}.
Thus $\tau_N(\veps)$ is the first time when the path
of the Markov chain $\bXn_k$ deviates from the trajectory of the deterministic dynamical system \eqref{psi} by more than a given threshold $\veps>0.$
We have:
\begin{theorem}
\label{th1}
Suppose that the map $\bGamma:\Delta_M \to \Delta_M$ is  Lipschitz continuous, that is there exists a constant $\rho>0$ such that
\beqn
\label{lip}
\|\bGamma(\bx)-\bGamma(\by)\|\leq \rho \|\bx-\by\|
\feqn
for all $\bx,\by\in \Delta_M.$ Then the following holds for any $\veps>0,$ $K\in\nn,$ and $N\in\nn:$
\beqn
\label{first}
P\bigl( \tau_N(\veps)\leq K \bigr)\,
\leq
2KM\exp\Bigl(-\frac{\veps^2c_K^2}{2}N\Bigr),
\feqn
where
\beqn
\label{ck}
c_K=
\left\{
\begin{array}{ll}
\frac{1-\rho}{1-\rho^K}&\mbox{\rm if}~\rho\neq 1,
\\
[6pt]
\frac{1}{K}&\mbox{\rm if}~\rho=1.
\end{array}
\right.
\feqn
In particular, if $\rho<1$ we have:
\beq
P\bigl( \tau_N(\veps)\leq K \bigr)\,\leq
2KM\exp\Bigl(-\frac{\veps^2(1-\rho)^2}{2}N\Bigr).
\feq
\end{theorem}
The proof of the theorem is given in Section~\ref{pth1}. Note that the upper bound in \eqref{first} is meaningful for large
values of $N$ even when $\rho\geq 1.$ When $\bGamma$ is a contraction, i.e. $\rho<1,$ the Banach fixed point theorem
implies that there is a unique point $\bmz\in \Delta_M$ such that $\bGamma(\bmz)=\bmz.$ Furthermore, for any $\bx\in \Delta_M$ and $k\in\nn$ we have
\beq
\bigl\|\bpsi_k-\bmz\bigr\|\leq \rho^{k-1}\|\bXn_0-\bmz\|,
\feq
that is, as $k$ goes o infinity, $\bpsi_k$ converges exponentially fast to $\bmz.$
The following theorem is a suitable modification of Theorem~\ref{th1} for contraction maps.
\begin{theorem}
\label{th3}
Assume that there exist a closed subset $\cale$ of $\Delta_M$ and a constant $\rho\in (0,1)$ such that \eqref{lip} holds
for any $\bx,\by\in\cale.$
Then the following holds for any $\veps>0,$ $K\in\nn,$ $N\in\nn$ such that $1-2M\exp\bigl(-\frac{(1-\rho)^2\veps^2 N}{2}\bigr)>0,$
and $\bXn_0\in \cale:$
\beq
P\bigl( \tau_N(\veps)>K\bigr)
\geq
\Bigl[1-2M\exp\Bigl(-\frac{(1-\rho)^2\veps^2 N}{2}\Bigr)\Bigr]^K.
\feq
In particular,
\beq
E\bigl( \tau_N(\veps)\bigr)=\sum_{K=0}^\infty P\bigl( \tau_N(\veps)>K\bigr)\geq \frac{1}{2M}\exp\Bigl(\frac{(1-\rho)^2\veps^2 N}{2}\Bigr).
\feq
\end{theorem}
The proof of the theorem is given in Section~\ref{pth3}. We remark that relaxing the assumption that $\bGamma$ is Lipschitz continuous
on the whole simplex $\Delta_M$ allows, for instance, to cover the one-dimensional  setup of \cite{ozgur} where $M=2,$ the underlying Markov chain is
binomial, and $\bGamma$ is a contraction only on a subset of $\Delta_2.$
\subsection{Metastability and elimination of types}
\label{exits}
In this section, we consider the Wright-Fisher model with $\bGamma$ satisfying the conditions of Theorem~\ref{bsa} (see page \pageref{bsa} within Appendix~B)
and a large population. In particular, $\Gamma_j(\bx)=0$ if and only if $x_j=0,$ and hence the Markov chain $\bXn$ is absorbing and will eventually fixate on one of the simplex vertices (cf. Appendix~A).
\par
The main result of the section is Theorem~\ref{xit}, which intuitively says that if the push away from the boundary and toward the equilibrium is strong enough and conditions of Theorem~\ref{bsa} hold, with a high degree of certainty, the type with the lowest fitness at the equilibrium is the  one that goes extinct first. The result provides an interesting link between the interior equilibrium of the mean-field dynamics and the asymptotic behavior of the stochastic system which is eventually absorbed at the boundary. A similar phenomenon for a different
type of stochastic dynamics has been described in physics papers \cite{assaf, assaf1, metapark}, see also \cite{leto}.
Key generic features that are used to prove the result are:
\begin{enumerate}
\item The mean-field dynamics is permanent.
\item There is a unique interior equilibrium $\bmz$. In addition, $\lim_{k\to\infty} \bpsi_k=\bmz$ for any $\bpsi_0\in\Delta_M^\circ.$
\item The push toward equilibrium is strong enough to force the stochastic system to follow the deterministic dynamics closely
and get trapped for a prolonged time in a ``metastable" state near the equilibrium.
\item Finally, the same strong push toward the equilibrium yields a large deviation principle which ensures that the stochastic
system is much likely to jump to the boundary directly from the neighborhood of the equilibrium by one giant fluctuation than
by a sequence of small consecutive steps.
\end{enumerate}
\par
With these features, one would expect the following. If the starting point $\bXn_0$ lies within the interior of $\Delta_M,$ the mean-field dynamical system
$\bpsi_k$ converges to the unique equilibrium, $\bmz\in\Delta_M^\circ.$ It can be shown that the rate of the convergence is exponential.
In view of Theorem~\ref{bsa} in Appendix~B and Theorem~\ref{da}, one can expect
that the trajectory of the stochastic system will stay close to the deterministic mean-field path for a long enough time, so that the stochastic model will follow the mean-field path to a small neighborhood of $\bmz.$ Then the probability of an immediate elimination of a type $i$ will be of order $\Gamma_i(\bmz)^N.$ Assuming that
\beqn
\label{heun}
\mbox{\rm ``it is easier to return to a small neighborhood of $\bmz$ than to leave it"}
\feqn
we expect that the stochastic system will enter a metastable state governed by its quasi-stationary distribution and will be trapped in a neighborhood of the equilibrium
for a very long time. In fact, since by virtue of Theorem~\ref{da}, $\bXn$ can be considered as a random perturbation of the dynamical system $\bpsi_k$ with a small noise, the quasi-stationary distribution should converge as $N$ goes to infinity to the
degenerate probability measure supported by the single point $\bmz.$ Finally, general results in the theory of large deviations \cite{quasi, eagle} suggest that under a suitable condition \eqref{heun}, with an overwhelming probability, the system will eventually jump to the boundary by one giant fluctuation. It follows from the above heuristic argument, that the probability of such an instant jump to the facet where, say, $x(j)=0$ and the rest of coordinates are positive is roughly $\Gamma_j(\bmz)^N.$ For large values of $N$ this practically guarantees that the stochastic system will hit the boundary through the interior of $\Delta_{[J^*]},$ where $J^*=S_M\backslash\{j^*\},$ and $j^*\in \arg\min_{i\in S_M}\{\Gamma_i(\bmz):i\in S_M\}.$
In other words, under the conditions of Theorem~\ref{bsa}, when $N$ is large, the elimination of the least fitted and survival of the remaining types is almost certain (see below results of a numerical simulation supporting the hypothesis).
\par
We now turn to the statement of our results. The following set of conditions is borrowed from \cite{losakin} (see Theorem~\ref{bsa} in Appendix~B).
\begin{assume}
\label{assume7}
Suppose that $\bGamma$ is defined by \eqref{replica} and \eqref{lf4}.
Furthermore, assume that $\bb=(1,1,\ldots,1)$ and $\bA$ is a symmetric invertible $M\times M$ matrix with positive entries such that the following two conditions hold:
\beqn
\label{seq1}
\bw^T\bA\bw<0\qquad \qquad \forall\, \bw\in W_M,\,\bw\neq \bo,
\feqn
and there are $\bmz\in \Delta_M^\circ$ and $c>0$ such that
\beqn
\label{seq3}
\bA\bmz=c\be,
\feqn
where $\be:=(1,\ldots,1)\in\rr^M.$
\end{assume}
First, we state the following technical result. Recall \eqref{psi} and Definition~\ref{conley}.
\begin{proposition}
\label{key}
Let Assumption~\ref{assume7} hold. Then, for any neighborhood $\caln$ of the interior equilibrium $\bmz$ and
any compact set $K\subset \Delta_M$ there exist a real constant $\eta=\eta(\caln,K)>0$ and an integer $T=T(\caln,K)\in\nn$ such that
if $\veps\in (0,\eta),$ $\bx_0\in K,$ and $\bx_0,\bx_1,\ldots,\bx_T$ is an $\veps$-chain of length $T$ for the dynamical system $\bpsi_k,$ then $\bx_T\in\caln.$
\end{proposition}
The proof of the proposition is deferred to Section~\ref{keyp}. The proposition is a general property of dynamical system converging to an
asymptotically stable equilibrium (see Remark~\ref{rp1} at the end of the proof). Since we were unable to find
a reference to this claim in the literature, we provided a short self-contained proof in Section~\ref{keyp}.
\par
We are now ready to state the main result of this section. Suppose $\bx\in \Delta_M$ such that at least one component $\Gamma_i(\bx)$ is not equal to $1/M.$ Then, we define
\beq
\alpha_\bx=\min_{j\in S_M} \Gamma_j(\bx), \qquad J^*_\bx=\{j\in S_M:\Gamma_j(\bx)=\alpha_\bx\},
\qquad \mbox{\rm and}\qquad \beta_\bx=\min_{j\not\in J^*_\bx}\Gamma_j(\bx).
\feq
We set $\alpha:=\alpha_\bmz,$ $\beta:=\beta_\bmz,$ and $J^*:=J^*_\bmz.$
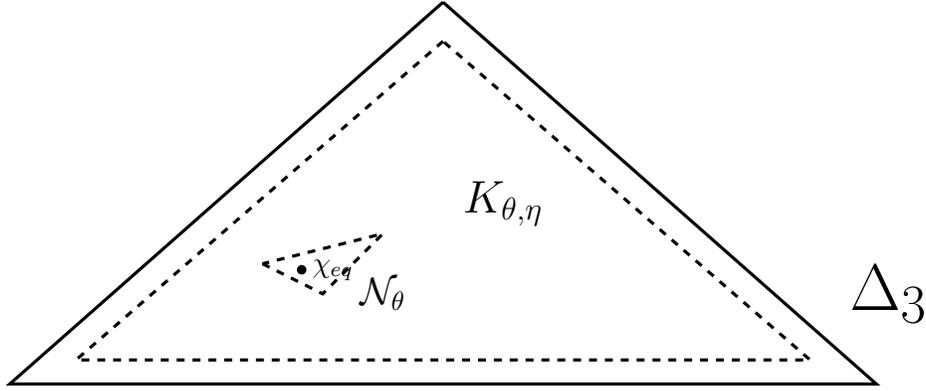
\begin{figure}[!tb]
\begin{center}
\begin{tikzpicture} [scale = 0.8]
\draw[very thick, dashed] (0,7.2) -- (-6.1,1.9) -- (6.1,1.9) -- (0,7.2) ;
\draw[very thick] (0,7.85) -- (-7.2,1.5) -- (7.2,1.5) -- (0,7.85) ;
\draw[very thick, dashed] (-3,3.5) -- (-2,3) -- (-1,4) -- (-3,3.5) ;
\draw (7.4,3) node {\huge $\Delta_3$} ;
\draw (1,4.5) node {\Large $K_{\theta,\eta}$} ;
\draw (-1,3) node {\large $\caln_\theta$} ;
\draw[fill] (-2.35,3.4) circle (2pt) node[right] {\small $\mz$} ;
\end{tikzpicture}
\caption{
Schematic sketch (in the planar barycentric coordinates) of the simplex $\Delta_3,$ the interior equilibrium $\bmz,$ its neighborhood $\caln_\theta,$ and a compact set $K_{\theta,\eta}.$ }
\label{fig1}
\end{center}
\end{figure}
\par
Additionally, for any $\theta\in \big(0,\frac{\beta-\alpha}{2}\big)$ and $\eta>0$ such that
\beqn
\label{eta1}
1-\alpha-\theta>(1-\beta+\theta)^{1-\eta},
\feqn
we define (see Fig~\ref{fig1}):
\beqn
\label{nd}
\caln_\theta=\big\{\bx\in \Delta_M^\circ: \alpha_x<\alpha+\theta,\,\beta_x>\beta-\theta\big\}
\feqn
and
\beqn
\label{etad}
K_{\theta,\eta}:=\big\{\bx\in\Delta_M^\circ\backslash\caln_\theta:\max_{j\not\in J^*}x(j)\geq \eta\big\}.
\feqn
In the infinite population limit we have:
\begin{theorem}
\label{xit}
Assume that $\bGamma$ satisfies Assumption~\ref{assume7} and $\bXn_0\overset{P}{\to}\bx_0$ for some $\bx_0\in\Delta_M^\circ.$
Suppose in addition that there exists $\theta\in \big(0,\frac{\beta-\alpha}{2}\big)$ and
$\Big(0,1-\frac{\log(1-\alpha-\theta)}{\log(1-\beta+\theta)}\Big)$ such that
\beqn
\label{hc}
1-\alpha-\theta>e^{-\frac{\veps_\theta^2}{2}},
\feqn
where $\veps_\theta=\veps(\caln_\theta,K_{\theta,\eta}),$ and $\veps(\caln,K)$ is introduced in Proposition~\ref{key}.
\par
Let
\beqn
\label{nu}
\nu_N:=\inf\big\{k>0:\bXn_k\in \partial(\Delta_M)\big\}
\feqn
and $\cale_N$ be the event that
\begin{enumerate}
\item[1)] $\bXn_{\nu_N}$ has exactly $M-1$ non-zero components ($\big|C\big(\bXn_{\nu_N}\big)\big|=M-1$ in notation of \eqref{calf}); and
\item[2)] $\Xn_{\nu_N}(i)=0$ for some $i\in J^*.$
\end{enumerate}
Then, the following holds true:
\item [(i)] $\lim_{N\to\infty} P (\cale_N)=1.$
\item [(ii)] $\lim_{N\to\infty} E(\nu_N)=+\infty.$
\end{theorem}
The proof of Theorem~\ref{xit} is included in Section~\ref{exit}. In words, part (i) of the theorem is the claim that with overwhelming probability,
at the moment that the Markov chain hits the boundary, exactly one of its components is zero, and, moreover, the index $i$ of the zero component
$\bXn_{\nu_n}(i)$ belongs to the set $J^*.$ A similar metastability phenomenon for continuous-time population models has been considered in \cite{assaf} and \cite{assaf1, leto, metapark}, however their proofs rely on a mixture of rigorous computational and heuristic ``physicist arguments".
\par
Condition \eqref{hc} is a rigorous version of the heuristic condition \eqref{heun}.
We note that several more general but less explicit and verifiable conditions of this type can be found
in mathematical literature, see, for instance, \cite{quasi,imix} and \cite{assaf, metapark}. We conjecture the following:
\begin{conj}
\label{conja}
With a suitable modification of the condition \eqref{hc}, a similar statement to that of Theorem~\ref{xit} hold true
for a general positive symmetric matrix $\bA.$ We indeed believe the negative-definite condition \eqref{seq1} and the existence
of an interior equilibrium condition \eqref{seq3} are not required in general.
\end{conj}
It is shown in \cite{losakin} that in this more general situation, the mean-field
dynamical system $\psi_k$ still converges to an equilibrium which might be on the boundary and depends on the initial condition $\psi_0.$
\begin{remark}
Recall \eqref{calf}. A classical result of Karlin and McGregor \cite{kms} implies that under the conditions of Theorem~\ref{xit},
the types vanish from the population with an exponential rate. That is, $\lim_{k\to\infty} \frac{1}{k}\log E\big(\big|C\big(\bXn_k\big)\big|\big)$ exists and is finite and strictly negative. However, heuristically, the limit is extremely small under the conditions of the theorem
(the logarithm of the expected extinction time is anticipated to be of order $M\cdot N$ \cite{imix}),
and hence the convergence to boundary states is extremely slow for any practical purpose. The situation
when the system escapes from a quasi-equilibrium fast but with a rate of escape quickly converging to zero when the population size increases
to infinity, is typical for metastable population systems \cite{leto, negac, seneta1, q}.
\end{remark}
We will next illustrate the (hypothesized) important role of the threshold condition \eqref{hc} with a numerical example. To that end, we consider two examples (see matrices $\bA_1$ and $\bA_2$, bellow) with $N=500$ and $M=3$. The choice of the low dimension $M=3$ is due to the computational intensity
involved with tracking the evolution of a metastable stochastic system in a ``quasi-equilibrium" state. For each example, we use various sets of initial conditions and run $10000$ rounds of simulations. Each simulation is stopped at the random time $T$ defined as follows:
\beqn
\label{tea}
T=\min\big\{k\in\nn\,:\min_j \Xn_k(j)\leq 0.05\big\}.
\feqn
Using this termination rule, we were able to complete all the simulation runs in a reasonable finite time. In each simulation run, we sample the value of the process at some random time ranging between $1000$ and $5000.$ At this intermediate sampling time,  we measure $d_{eq},$ the distance between the stochastic system and
the equilibrium $\bmz$. In addition, we record the composition of the last state of the Markov chain at the end of each simulation run.

\begin{example}
Our first example does not satisfy the threshold conditions. More precisely, we choose $\omega$ in \eqref{lf4} in such a way that $\frac{\omega}{1-\omega}=10^{-3}$ and we consider
the following symmetric $3\times 3$ matrix $\bA_1:$
\beq
\bA_1=
\left(
\begin{array}{ccc}
	1&20&45
	\\
	20&21&30
	\\
	45&30&1
\end{array}
\right)
\feq
with the equilibrium $\bmza=(0.24766355,  0.41121495,  0.3411215)$. Table \ref{tab-r1} reports the number of times that each type $j\in S_3$ was the type with abundance less than $0.05$ at the time of exiting the simulation runs.
\begin{figure}[H]
\begin{center}
\begin{tabular}{ c|c|c|c| }
\cline{2-4}
& \multicolumn{3}{|c|}{Number of times $\Xn_T(i)\leq 0.05$} \\
\cline{2-4}
 &$~~~i=1~~~$ & $~~~i=2~~~$ &$~~~i=3~~~$ \\
\hline
\multicolumn{1}{ |c|  }{For $\bXn_0=(0.8,0.1,0.1)$} &933& 6831& 2248 \\
\multicolumn{1}{ |c|  }{For $\bXn_0=(0.1,0.8,0.1)$} & 5991 & 164& 3903 \\
\multicolumn{1}{ |c|  }{For $\bXn_0=(0.1,0.1,0.8)$} & 3692 & 5940& 373 \\
\hline
\end{tabular}
\end{center}
\caption{The least abundant component of the terminal state $\bXn_T$ by the end of the simulations for
$\Gamma_i(\bx)=\frac{1000+A_1x(i)}{1000+\bx^T\bA_1\bx}x(i)$ and the quitting time $T$ introduced in \eqref{tea}.} \label{tab-r1}
\end{figure}
In addition, Fig~\ref{fig3} below depicts the distribution of $d_{eq}$ different initial values. Evidently,
the threshold condition plays an important role as the theory predicts in Theorem~\ref{xit}. The results in Fig.~\ref{tab-r1}
cannot be explained by the equilibrium value $\bmza$ only, rather they constitute an intricate result of the combined effect of this value
and the initial position of the Markov chain.
\end{example}
\begin{example}
For the second example, we let $\omega=1/2$ and use
the following matrix $\bA_2$:
\beq
\bA_2=
\left(
\begin{array}{ccc}
1&20&35
\\
20&21&30
\\
35&30&1
\end{array}
\right)
\feq
with the equilibrium $\bmzb=(0.0246913,  0.7345679,  0.2407407).$ This matrix is almost the same as $\bA_1,$
with the only difference that $45$ is changed to $35.$
The result of simulations for three different initial values are reported in Table \ref{tab-r2} and Fig.~\ref{fig4}, which are in complete agreement with the prediction of Theorem~\ref{xit}.

\begin{figure}[H]
\begin{center}
\begin{tabular}{ c|c|c|c| }
\cline{2-4}
& \multicolumn{3}{|c|}{Number of times $\Xn_T(i)\leq 0.05$} \\
\cline{2-4}
 &$~~~i=1~~~$ & $~~~i=2~~~$ &$~~~i=3~~~$ \\
\hline
\multicolumn{1}{ |c|  }{For $\bXn_0=(0.8,0.1,0.1)$} &10000& 0& 0 \\
\multicolumn{1}{ |c|  }{For $\bXn_0=(0.1,0.8,0.1)$} & 10000 & 0& 0 \\
\multicolumn{1}{ |c|  }{For $\bXn_0=(0.1,0.1,0.8)$} & 10000 & 0& 0 \\
 \hline
\end{tabular}
\end{center}
\caption{The least abundant component of the terminal state $\bXn_T$ by the end of the simulations for
$\Gamma_i(\bx)=\frac{1+A_2x(i)}{1+\bx^T\bA_2\bx}x(i)$ and the quitting time $T$ introduced in \eqref{tea}.}
\label{tab-r2}
\end{figure}
\end{example}
The difference between the two examples is two-fold: (1) by increasing the value
of $\omega$ we increase the influence of the ``selection matrix" $\bA$ comparing to the neutral selection, and (2) by changing $\bA_1$ to $\bA_2$
we replace the equilibrium $\bmza$ with fairly uniform distribution of ``types" by the considerably less balanced
$\bmzb,$ increasing the threshold in \eqref{hc} by a considerable margin.
\\
\begin{figure*}[ht!]
		\centering
		\begin{tabular}{cc}
			\includegraphics[width=0.473\textwidth]{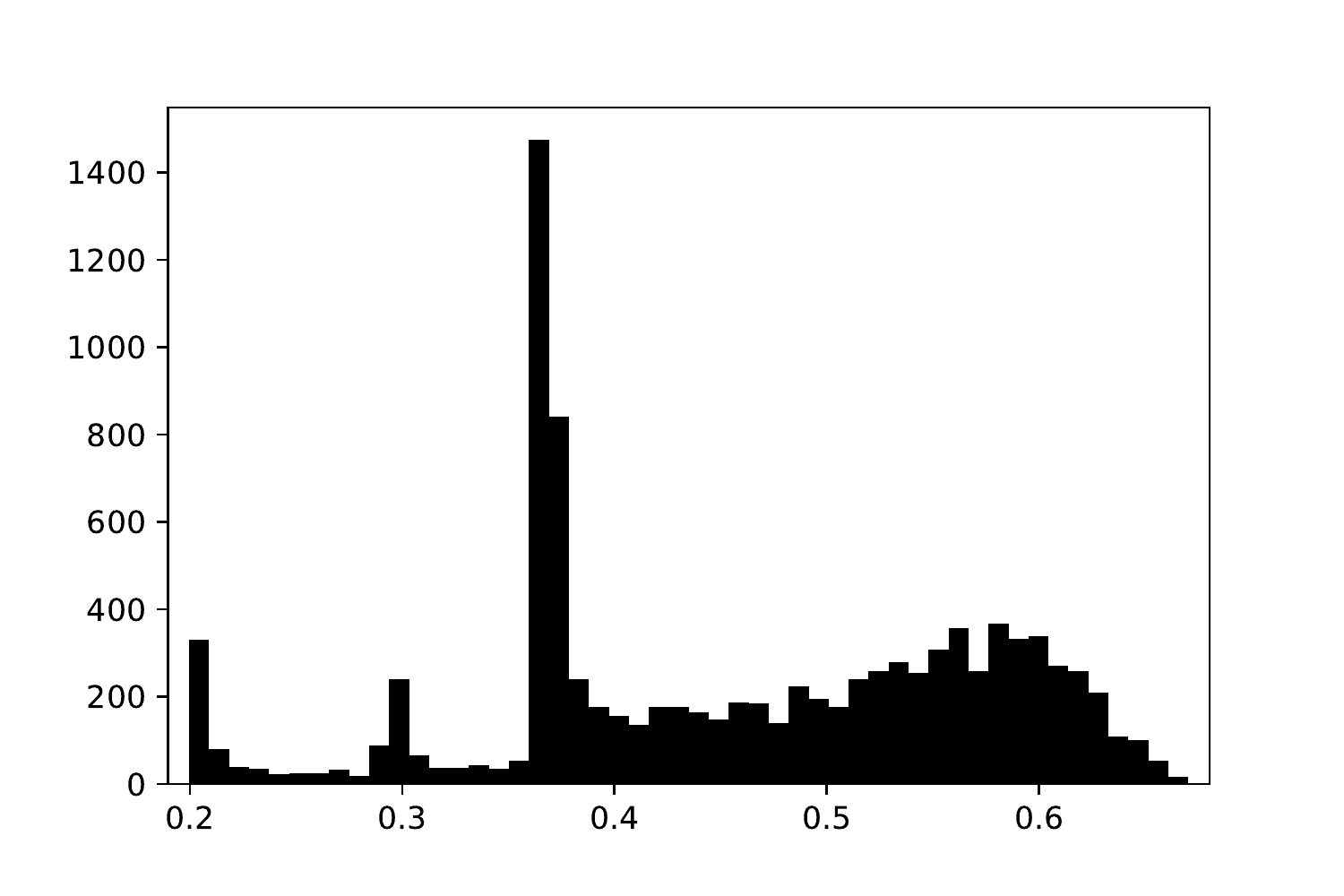}  &
			\includegraphics[width=0.473\textwidth]{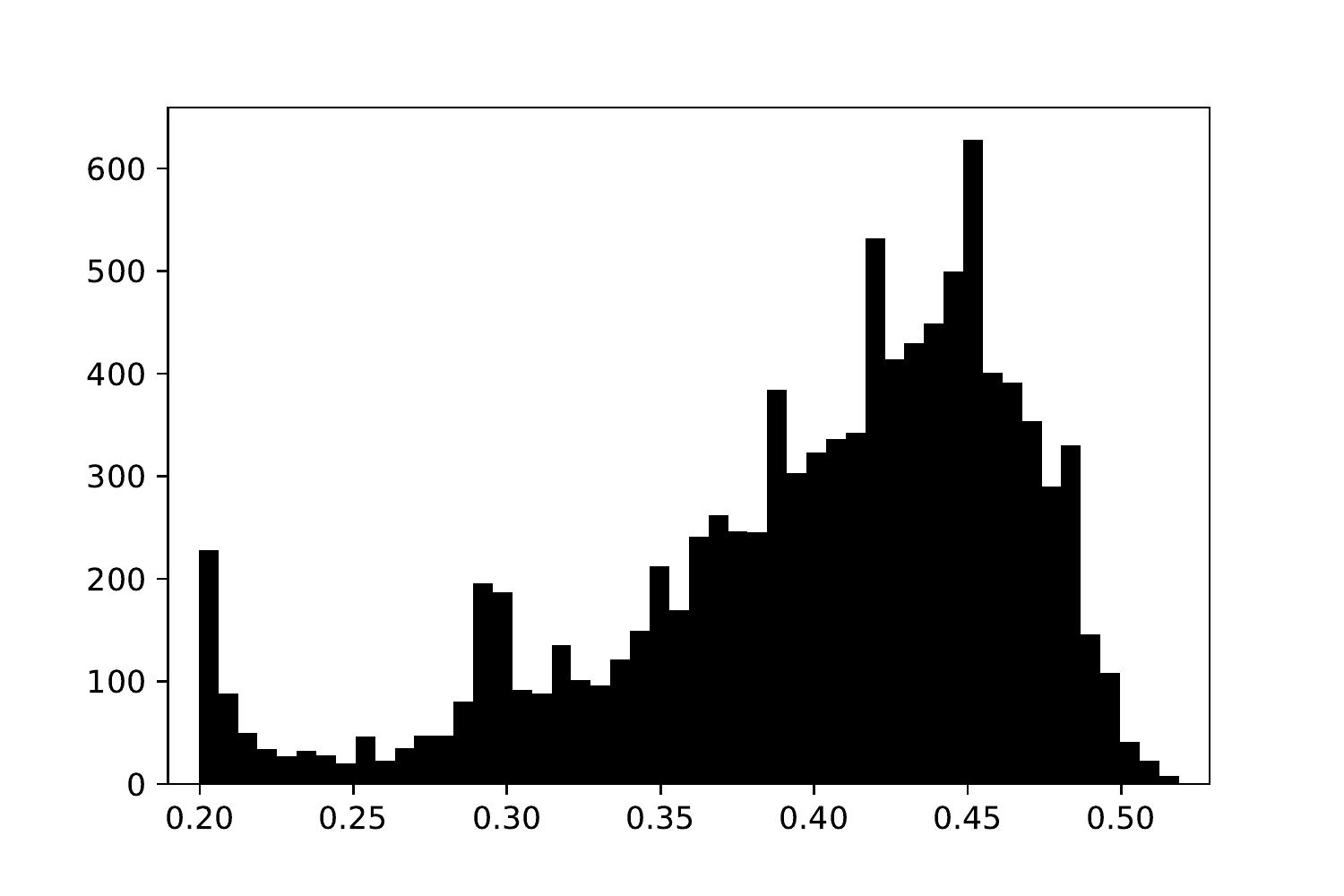} \\
			(a) $d_{eq}$  for $\bXn_0=(0.8,0.1,0.1)$& (b) $d_{eq}$  for $\bXn_0=(0.1,0.8,0.1)$ \\ [6pt]
    	\end{tabular}
		\begin{tabular}{c}
			\includegraphics[width=0.473\textwidth]{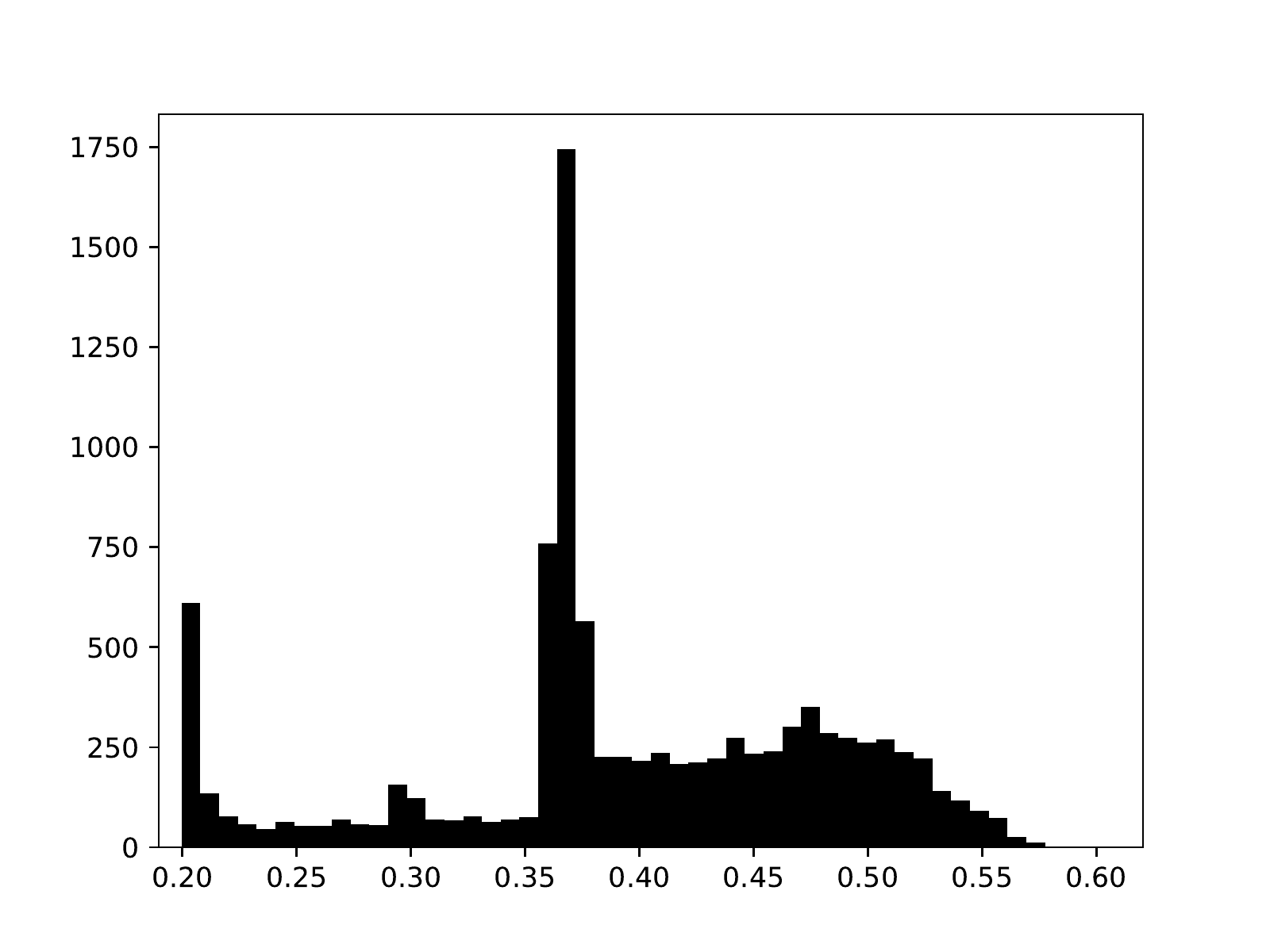} \\
			(c) $d_{eq}$  for $\bXn_0=(0.1,0.1,0.8)$ \\ [6pt]
		\end{tabular}
		\caption{Simulations of the trajectories of $\bpsi_k$ and $\bXn_k$ for $\Gamma_i(\bx)=\frac{1000+A_1x(i)}{1000+\bx^T\bA_1\bx}x(i).$
The $x$-axis represents $d_{eq},$ and the height of a histogram bar corresponds to the number of occurrences of
$d_{eq}$ in a simulation of 10,000 runs.}
		\label{fig3}
	\end{figure*}	
\begin{figure*}[ht!]
		\centering
		\begin{tabular}{cc}
			\includegraphics[width=0.473\textwidth]{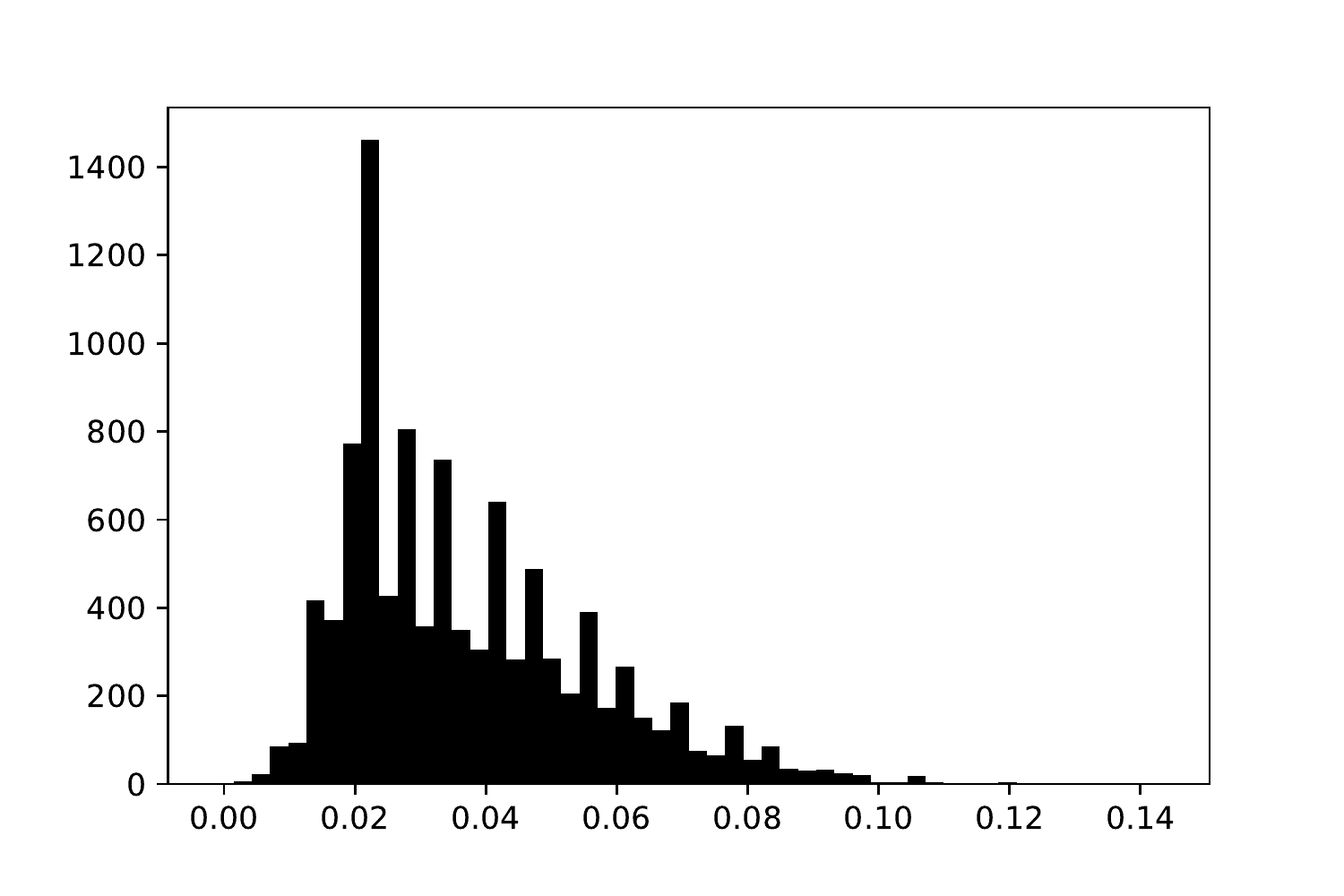}  &
			\includegraphics[width=0.473\textwidth]{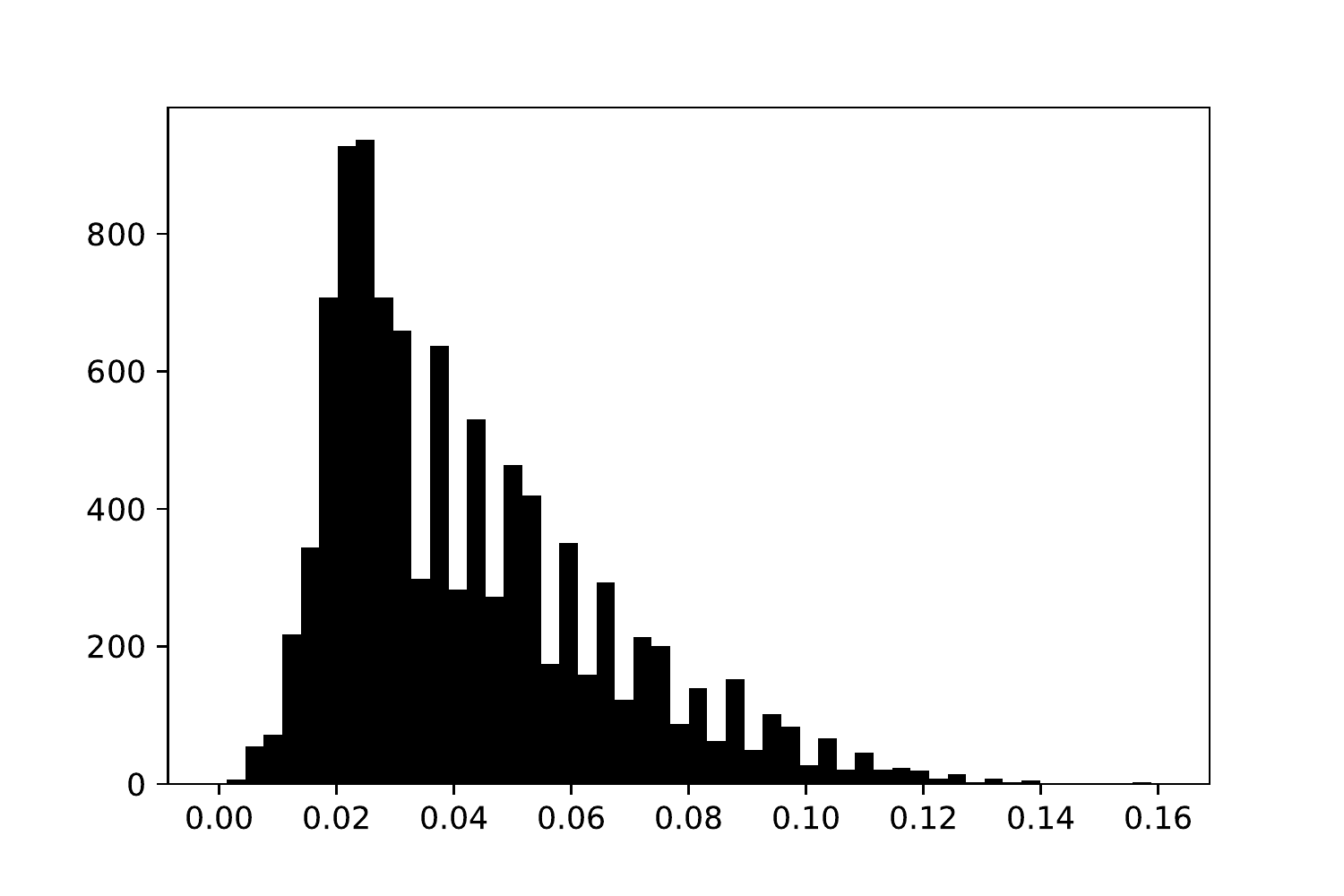}  \\
			(a) $d_{eq}$  for $\bXn_0=(0.8,0.1,0.1)$ & (b) $d_{eq}$  for $\bXn_0=(0.1,0.8,0.1)$\\ [6pt]
		\end{tabular}
        \begin{tabular}{c}
			\includegraphics[width=0.473\textwidth]{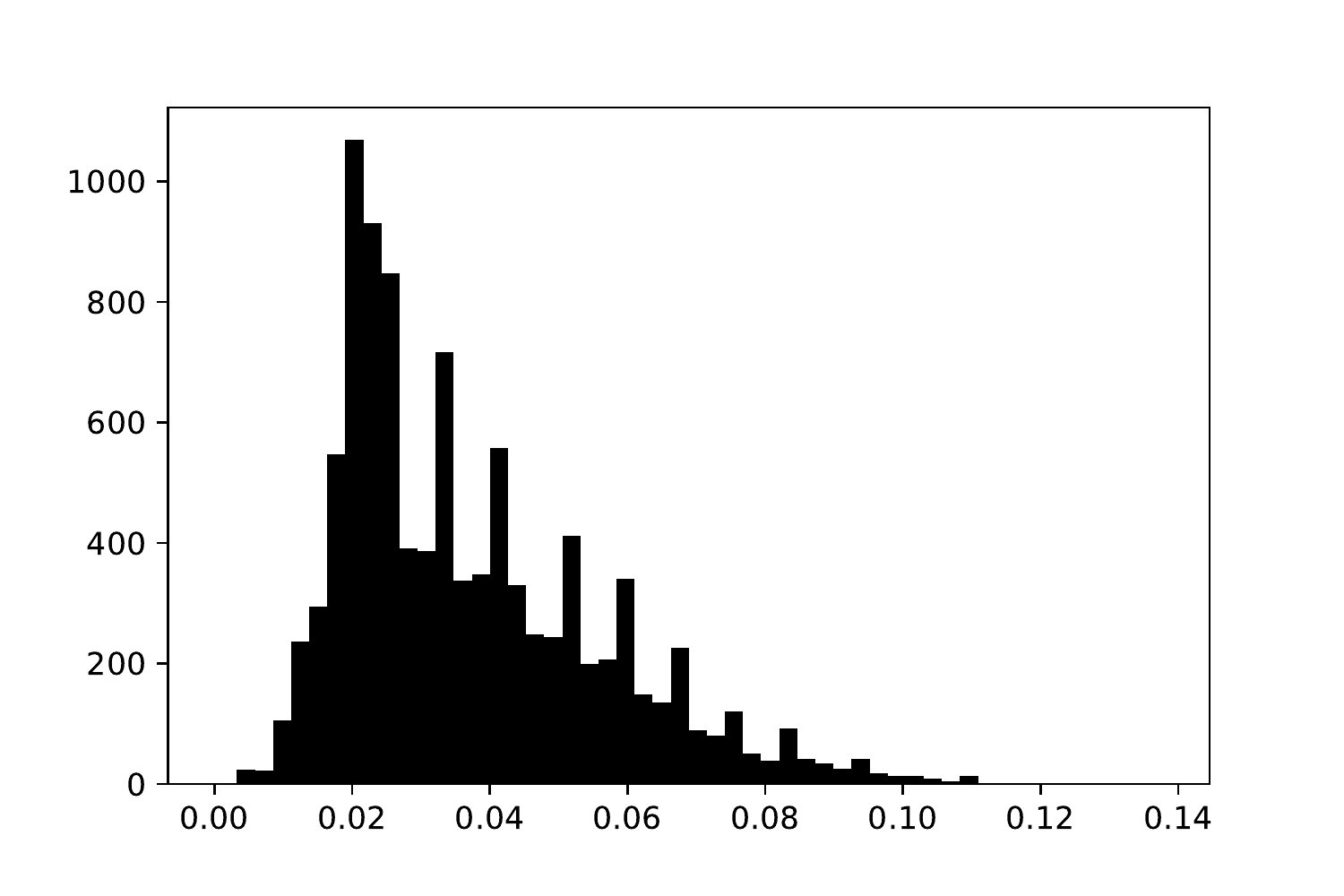} \\
			(c) $d_{eq}$  for $\bXn_0=(0.1,0.1,0.8)$ \\ [6pt]
		\end{tabular}
		\caption{Simulations of the trajectories of $\bpsi_k$ and $\bXn_k$ for $\Gamma_i(\bx)=\frac{1+A_2x(i)}{1+\bx^T\bA_2\bx}x(i).$
The $x$-axis represents $d_{eq},$ and the height of a histogram bar corresponds to the number of occurrences of
$d_{eq}$ in a simulation of 10,000 runs.}
		\label{fig4}
	\end{figure*}
\\
Next we provide an auxiliary ``contraction near the equilibrium" type result, further  illustrating the mechanism of convergence to equilibrium in Theorem~\ref{xit}.
\begin{proposition}
\label{radius}
Let Assumption~\ref{assume7} hold. Then the spectral radius of the Jacobian of $\bGamma$ at $\bmz$ is strictly less than one.
\end{proposition}
The proof of the proposition is included in Section~\ref{pr}.
This result is instrumental in understanding various applications including but not limited to mixing times \cite{imix} and the random sequence $\bu_k$ defined in the statement of Theorem~\ref{da}. In the later case, as we mentioned before, it implies that the limiting sequence $\bU_k$ in \eqref{vepr} converges to a stationary distribution.
\par
The following corollary is
an application directly related to the metastability of the system.  Let $\mu_N$ be the \textit{quasi-stationary distribution} of the Markov chain $\bXn_k$ on $\Delta_M^\circ.$ Then, there exists
$\lambda_N\in (0,1)$ such that $\mu_NP_{\circ,N}=\lambda_N \mu_N,$ where $P_{\circ,N}$ is  transition matrix of $\bXn_k$
restricted to $\dmn^\circ$  \cite{seneta1, q}.  Iterating, we obtain that for all $k\in\zz_+$ and $\by\in \dmn^\circ,$
\beqn
\label{ln1}
\mu_N P_{\circ,N}^k(\by)=\sum_{\bx\in \dmn^\circ} \mu_N(\bx)P\big(\bXn_k=\by\,\big|\,\bXn_0=\bx\big)=\lambda_N^k \mu_N(\by).
\feqn
Summing up over all $\by\in \dmn^\circ,$ we obtain
\beq
\lambda_N^k= P_{\mu_N}(\nu_N>k),
\feq
where $\nu_N$ is the hitting time introduced  in \eqref{nu}.
In particular,
\beqn
\label{ln}
1-\lambda_N=\sum_{\bx\in \dmn^\circ,y\in \partial(\dmn)} \mu_N(\bx)P_N(\bx,\by).
\feqn
\begin{corollary}
\label{metac}
Let Assumption~\ref{assume7} hold. Then, there exist constants $C>0$ and $b>0$ such that $\lambda_N\geq 1-Ce^{-b N}$ for all
$N\in\nn.$
\end{corollary}
A short proof of the corollary is included in Section~\ref{pmet} below. The result in the corollary and the comparison of \eqref{ln} with \eqref{239} in the proof of Theorem~\ref{xit} given below, motivate the following hypothesis:
\begin{conj}
\label{coca}
Let the conditions of Theorem~\ref{xit} hold. Then:
\item [(i)] The quasi-stationary distribution $\mu_N$ converges, as $N\to\infty,$ to the degenerate distribution supported on the single point $\bmz.$
That is, $\lim_{N\to \infty} \mu_N(\caln)=1$ for any open neighborhood of the equilibrium $\bmz.$
\item [(ii)] $\lim_{N\to\infty} \frac{\alpha^N}{1-\lambda_N}=1.$
\end{conj}
The first part of Conjecture~\ref{coca} is a natural property of stochastic models obtained as a perturbation of asymptotically stable
deterministic dynamical system (cf. \cite{quasi}). Recall $V_M$ from \eqref{vm} and define
\beqn
\label{vm1}
V_M^*=\{\be_j: j\in J^*\}\subset V_M.
\feqn
To illustrate the second part of the conjecture, write:
\beq
P_{\mu_N}(\cale_N)&=&\sum_{k=0}^\infty \sum_{\bx\in \dmn} P_{\mu_N}\big(\bXn_j=\bx, \nu_N>k\big) P\big(\bXn_{k+1}\in V_M^*\,|\,\bXn_k=\bx\big)
\\
&=&
\sum_{k=0}^\infty \sum_{\bx\in \dmn} \mu_N(\bx) P_{\mu_N}(\nu_N>k) P\big(\bXn_{k+1}\in V_M^*\,|\,\bXn_k=\bx\big)
\\
&=&
\sum_{k=0}^\infty \sum_{\bx\in \dmn} \mu_N(\bx) \lambda_N^k P\big(\bXn_1\in V_M^*\,|\,\bXn_0=\bx\big)
\\
&=&
\frac{1}{1-\lambda_N} \sum_{\bx\in \dmn} \mu_N(\bx) P\big(\bXn_1\in V_M^*\,|\,\bXn_0=\bx\big)
\\
&=&
\frac{1}{1-\lambda_N} P_{\mu_N}\big(\bXn_1\in J^*\big)=\frac{P_{\mu_N}\big(\bXn_1\in V_M^*\big)}{P_{\mu_N}\big(\bXn_1\in V_M\big)},
\feq
where the equality in the last step follows from \eqref{ln}. Thus, Conjecture~\ref{coca} is almost equivalent to the following:
\begin{conj}
\label{coca1}
Let the conditions of Theorem~\ref{xit} hold. Then, $\lim_{N\to\infty} P_{\mu_N}(\cale_N)=1.$
\end{conj}
A weak form of the hypothesis that $\lim_{N\to\infty} \frac{\alpha^N}{1-\lambda_N}=1$ is
\begin{conj}
Let the conditions of Theorem~\ref{xit} hold. Then, $\lim_{N\to\infty} \frac{1}{N}\log(1-\lambda_N)=\alpha.$
\end{conj}
\section{Proofs}
\label{proofs}
This section is devoted to the proof of our main results stated in Section~\ref{mainr}. As before, transition matrix of the Markov chain $\bXn$ is denoted
by $P_N.$ We write $P_{\mu_N}(A_n)$ for the probability of an event $A_n$ measurable with respect to the $\sigma$-algebra generated by $(\bXn_k)_{k\geq 0}$
if $\mu_N$ is the distribution of $\bXn_0.$ If $\mu_N$ is concentrated on a single point $\bx\in\dmn,$ we simplify this notation to $P_x(A_N).$ We use the standard notations $f_n=o(g_n)$ and $f_n=O(g_n)$ to indicate that, respectively, the sequence $f_n$ is ``ultimately smaller" than $g_n,$ that is $\lim_{n\to\infty} \frac{f_n}{g_n}=0,$ and $f_n$ and $g_n$ are of the same order, that is $\lim_{n\to\infty} \frac{f_n}{g_n}=$ exists
and it is not zero.
\subsection{Proof of Theorem~\ref{da}}
\label{pda}
Let $\zeta:=\bigl\{\zx_{k,i}:\bx\in\Delta_M,i\in\nn,k\in\zz_+\bigr\}$ be a collection of independent random variables distributed as follows:
\beqn
\label{z}
P\bigl(\zx_{k,i}=j\bigr)=\Gamma_j(\bx),\qquad j\in S_M.
\feqn
For $\bx\in\Delta_M,i\in\nn,k\in\zz_+,$ let $\bxxk$ be a random $M$-dimensional vector such that
\beq
\xx_{k,i}(j)=
\left\{\begin{array}{ll}
1&\mbox{\rm if}~\zx_{k,i}=j,
\\
0&\mbox{\rm otherwise}.
\end{array}
\right.
\feq
For given parameters $k,i,$ the family of random vectors $\bigl\{\bxxk:\bx\in\Delta_M\}$ is a ``white noise" random field.
Without loss of generality we can assume that
\beqn
\label{c}
\bZn_{k+1}=\sum_{i=1}^N {\bm \xi}_{k,i}^{(\bXn_k)}.
\feqn
Thus, by the law of large numbers, the conditional distribution of $\bXn_{k+1}=\frac{\bZn_{k+1}}{N},$ given $\bXn_k,$ converges weakly
to a degenerate distribution concentrated on the single point $\bGamma\bigl(\bXn_k\bigr).$ This implies part (a) of the theorem
by a result of \cite{karr} adopted to a non-homogeneous Markov chain setting
(cf. Remark~(i) in \cite[p.~60]{bdet}).
\par
To prove part (b) of the theorem, write
\beq
\bun_{k+1}=\frac{1}{\sqrt{N}}\sum_{i=1}^N \Bigl(\bxxk-\bGamma\bigl(\bXn_k\bigr)\Bigr)+\sqrt{N}\Bigl(\bGamma\bigl(\bXn_k\bigr)-\bpsi_{k+1}\Bigr).
\feq
The first terms converges in distribution, as $N$ goes to infinity, to the degenerate Gaussian distribution defined in the statement
of the theorem (see, for instance, Theorem~14.6 in \cite{vas}). Furthermore, Taylor expansion for $\bGamma$ gives
\beq
\sqrt{N}\Big(\bGamma\bigl(\bXn_k\bigr)-\bpsi_{k+1}\Big)=\sqrt{N}\Big(\bD_x(\bpsi_k) \big(\bXn_k-\bpsi_k\bigr)+O(1/N)\Big)
=\bD_x(\bpsi_k)\bun_k+o(1),
\feq
which implies the result. \hfill\hfill\qed
\subsection{Proof of Theorem \ref{th1}}
\label{pth1}
We have
\beq
\bigl\|\bXn_k-\bpsi_k\bigr\|&=&\bigl\|\bXn_k-\bGamma(\bXn_{k-1})+\bGamma(\bXn_{k-1})-\bGamma(\bpsi_{k-1})\bigr\|
\\
&\leq&
\bigl\|\bXn_{k}-\Gamma(\bXn_{k-1})\bigr\|+\bigl\|\Gamma(\bXn_{k-1})-\bGamma(\bpsi_{k-1})\bigr\|
\\
&\leq&
\bigl\|\bXn_{k}-\Gamma(\bXn_{k-1})\bigr\|+\rho \bigl\|\bXn_{k-1}-\bpsi_{k-1}\|.
\feq
For $m\in\nn,$ we set $\bun_m:=\bXn_{m+1}-\bGamma(\bXn_{m}).$ Iterating the above inequality, we obtain that
\beq
\bigl\|\bXn_{k}-\bpsi_{k}\bigr\|
\leq
\sum_{m=0}^{k-1} \rho^m \bigl\| \bun_{k-m-1}\bigr\|
\leq
c_k^{-1}\max_{0\leq m\leq k-1} \bigl\|\bun_m\bigr\|,
\feq
where $c_k$ is defined in \eqref{ck}. Therefore,
\beq
&&
P\Bigl( \max_{1\leq k\leq K}\bigl\|\bXn_k-\bpsi_k \bigr\| >\veps\Bigr)
\leq
P\Bigl(\max_{0\leq m\leq K-1}\|\bun_m\| >\veps c_K\Bigr)
\\
&&
\qquad
\leq
\sum_{m=0}^{K-1} P\Bigl(\|\bun_m\|>\veps c_K \Bigr)
=
\sum_{m=0}^{K-1} P\Bigl(\|\bXn_{m+1}-\bGamma(\bXn_m)\|>\veps c_K \Bigr).
\feq
Given $\bXn_m,$ the conditional distribution of the random vector $\bZn_{m+1}=N\bXn_{m+1}$ is the multinomial $\calm\bigl(N,\bGamma(\bXn_m)\bigr).$
Therefore, applying Hoeffding's inequality \cite{hinq} to conditionally binomial random variables $X_{m+1}(i),$ we obtain that for any $\eta>0,$
\beqn
\nonumber
P\Bigl(\|\bXn_{m+1}-\bGamma(\bXn_m)\|>\eta\Bigr)&\leq& \sum_{i=1}^M P\Bigl(\bigl|\Xn_{m+1}(i)-\Gamma_i(\bXn_m)\bigr|>\eta\Bigr)
\\
\label{bound11}
&\leq&
2M\exp\Bigl(-\frac{N\eta^2}{2}\Bigr).
\feqn
Therefore,
\beqn
\label{bound1}
P\Bigr( \max_{1\leq k\leq K}\bigl\|\bXn_k-\bpsi_k\bigr\| >\veps \Bigl)&\leq&
2KM\exp\Bigl(-\frac{\veps^2 c_K^2 N}{2}\Bigr).
\feqn
To complete the proof observe that $c_K=\frac{1-\rho}{1-\rho^K}>1-\rho$ if $\rho<1.$ \hfill\hfill \qed
\subsection{Proof of Theorem~\ref{th3}}
\label{pth3}
It follows from \eqref{bound1} that
\beqn
\label{hoeff}
P\Bigl(\|\bXn_{k+1}-\Gamma(\bXn_k)\|>(1-\rho)\veps\Bigr)\leq  2M\exp\Bigl(-\frac{N(1-\rho)^2\veps^2}{2}\Bigr).
\feqn
Furthermore, by \eqref{lip}, conditionally on the event that $\tau_N(\veps)>k,$ we have
\beq
\bigl\|\bGamma(\bXn_k)-\bpsi_{k+1}\bigr\|\leq \rho\|\bXn_k-\bpsi_k\|\leq \rho\veps.
\feq
Combining the last  inequality with \eqref{hoeff}, we obtain that with probability one,
\beq
P\bigl(\tau_N(\veps)>k+1\,\bigl|\,\tau_N(\veps)>k \bigl)\geq 1- 2M\exp\Bigl(-\frac{N(1-\rho)^2\veps^2}{2}\Bigr),
\feq
which implies the claim of the theorem.\hfill\hfill\qed
\subsection{Proof of Proposition~\ref{key}}
\label{keyp}
Recall $\psi_k$ from \eqref{psi}. For $\caln\subset \Delta_M,$ denote by $\tau(\bx,\caln)$ the first hitting time of the set
$\caln$ by $\bpsi_k$ with $\bpsi_0=\bx:$
\beq
\tau(\bx,\caln)=\inf\{k\in\zz_+:\bpsi_k(\bx)\in\caln \}.
\feq
Further, let $t(\bx,\caln)$ be the time when $\bpsi_k$ become trapped within $\caln:$
\beq
t(\bx,\caln)=\inf\{n\in\zz_+:\bpsi_k(\bx)\in\caln~\forall\, k\geq n \}.
\feq
In both the definitions we use the regular convention that $\inf\emptyset =+\infty.$ Note that, since $\bpsi_k$ converges to $\bmz,$
both $\tau(\bx,\caln)$ and $t(\bx,\caln)$ are finite for all $\bx\in\Delta_M^\circ.$
\begin{lemma}
\label{fx}
For any $\theta>0$ and an open neighborhood $\caln$ of $\bmz,$ there is  an open neighborhood $U_\bx$ of $\bx$ and a constant $T_\bx\in\nn$
such that $t(\by,\caln)\leq T_\bx$ for all $\by\in U_\bx.$
\end{lemma}
\begin{proof}[Proof of Lemma~\ref{fx}]
Under the conditions stated in Assumption~\ref{assume7}, $\bGamma$ is a diffeomorphism \cite{losakin}. In particular, $\bGamma$ is a locally Lipschitz map.
Hence, we may apply Theorem~1.2 in \cite{lyapa} to conclude that the interior equilibrium $\bmz$ is asymptotically stable for
the dynamical system $\bpsi_k.$ That is, for any $\veps>0$ there exists $\delta>0$ such that
\beqn
\label{asa}
\|\bpsi_0-\bmz\|\leq\delta \quad \Longrightarrow \quad \|\bpsi_k-\bmz\|\leq \veps~\,\forall\,k \geq 0.
\feqn
Equivalently, for any open neighborhood $\caln$ of $\bmz$ there exists an open neighborhood of $\bmz,$ say $\calu,$ included in $\caln$
such that $\bGamma(\calu)\subset \calu$ (that is $\calu$ is forward-invariant for $\bGamma$).
\par
Let $\calu$ be a forward-invariant open neighborhood of $\bmz$ included in $\caln.$
Since $\calu$ is an open set and $\bpsi_{\tau(\bx,\calu)}(\by)$ is a continuous function of the initial state $\by,$ one can choose an
open neighborhood $V_\bx$ of $\psi_{\tau(\bx,\calu)}(\bx)$ and a constant $\veps_\bx>0$ such that $V_\bx\subset \calu$ and,
furthermore, $\|\by-\bx\|<\veps_\bx$ implies $\bpsi_{\tau(\bx,\calu)}(\by)\in V_\bx.$ Set $U_\bx=\big\{\by\in\Delta_M^\circ: \|\by-\bx\|<\veps_\bx\big\}.$
\end{proof}
\begin{corollary}
\label{f}
Let $\caln$ be any open neighborhood of $\mz$ and $K$ be a compact subset of $\Delta_M^\circ.$ Then
there exists $T=T(K,\caln)\in\nn$ such that $t(\bx,\caln)\leq T$ for all $\bx\in K.$
\end{corollary}
\begin{proof}[Proof of Corollary~\ref{f}]
Let $U_\bx$ and $T_\bx$ be as in Lemma~\ref{fx}.
Consider a cover of $K$ by the union of open sets $\bigcup_{x\in K} U_\bx.$ Since $K$ is compact, we can choose a finite subcover,
say $\bigcup_{j=1}^m U_{\bx_j}.$ Set $T(K,\caln)=\max_{1\leq j\leq m} t(\bx_j,\caln).$
\end{proof}
\begin{remark}
\label{rp1}
The claim is a ``uniform shadowing property" of $\bGamma.$ Recall that
a sequence of points $(\bx_k)_{k\leq T}$ where $T\in\nn\bigcup \{+\infty\}$ is a $\delta$-pseudo orbit of the dynamical system $\bpsi_k$
if $\|\bGamma(\bx_k)-\bx_{k+1}\| \leq \delta$ for $k=0,\ldots,T-1$ (if $T$ is finite then the definitions of $\delta$-pseudo orbit and $\delta$-chain are equivalent) and that a $\delta$-pseudo orbit is said to be shadowed by a true orbit $\big(\bpsi_k(\bx)\big)_{k\in \zz_+}$ if $\|\bpsi_k(\bx)-\bx_k\|\leq \veps$ for all $k \geq 0.$
The proposition is basically saying that since $\bpsi_k$ is asymptotically stable, then 1) any infinite $\delta$-pseudo orbit starting at $\bx_0$ is
$\veps$-shadowed by the true orbit $\big(\bpsi_k(\bx_0)\big)_{k\geq 0}$ provided that $\delta>0$ is small enough; and 2) the property (more precisely, a
possible value of the parameter $\delta$ given $\veps>0$) is uniform over compact sets with respect to the initial point $\bx_0.$
\end{remark}
\subsection{Proof of Theorem~\ref{xit}}
\label{exit}
\begin{proof}[Proof of (i)]
Let $\odin_A$ denote the indicator of the event $A,$ that is $\odin_A=1$ if $A$ occurs and $\odin_A=0$ otherwise.
For $\theta$ introduced in the statement of the theorem and $\caln_\theta$ defined in \eqref{nd}, let $L_N$ be the number of times the Markov chain $\bXn$ \textit{returns} to $\caln_\theta$ before it hits the boundary $\partial(\Delta_M)$ for the first time. That is,
\beqn
\label{trout}
L_N=\sum_{k=1}^{\sigma_N} \one{\bXn_k\in \caln_\theta},
\feqn
where
\beq
\sigma_N=\inf\big\{k>0:\bXn_k\in \partial(\dmn\big)\}
\feq
is the first hitting time of the boundary. Since for a fixed $N\in\nn,$ the state space of the Markov chain $\bXn$ is finite and
\beq
P_N\big(\bx,\partial\big(\dmn\big)\big)=\sum_{\by\in \partial(\dmn)}P_N(\bx,\by)>0
\feq
for all $\bx\in\dmn,$ we have:
\beq
P_\bx(0\leq L_N <\infty)=P_\bx(1\leq \sigma_N<\infty)=1\qquad \forall\,\bx\in \dmn^\circ.
\feq
Note that while counting the number of \textit{returns} (rather than \textit{visits}) to $\caln_\theta$ in \eqref{trout}, we don't account for the initial state $\bXn_0$ even if
$\bXn_0\in \caln_\theta.$
\par
Denote
\beq
G_N=\{\sigma_N>0\},\qquad  \qquad \ol G_N=\{\sigma_N=0\},
\feq
and
\beqn
\label{xg}
\xi_N=\inf\{k\geq 0: \bXn_k\in \caln_\theta\}.
\feqn
That is $G_N$ is the event that $\bXn$ returns to $\caln_\theta$ at least once, $\ol G_N$ is its complement, and $\xi_N$ is the time of the first visit (rather than the first return, which means
$\bXn_0$ counts in \eqref{xg} if it belongs to $\caln_\theta$). Note that $P_\bx(\xi_N=+\infty)>0$ for all $\bx\in\dmn$ because the chain can escape to the boundary of the simplex before it visits to $\caln_\theta,$
and the boundary is an absorbing set for $\bXn.$
\par
Let $\eta>0$ satisfy \eqref{eta1}, and recall $K_{\theta,\eta}$ from \eqref{etad}.
Proposition~\ref{key} implies that there exist (deterministic, non-random) real constant $\veps_\theta>0$
and integer $T>0$ such that
\beq
\bXn_0\in K_{\theta,\eta}~\mbox{\rm and}~\big\|\bXn_{k+1}-\bGamma\big(\bXn_k\big)\big\|<\veps_\theta~
\mbox{\rm for}~k \leq T-1\quad \Rightarrow \quad \bXn_T\in \caln_\theta.
\feq
It follows from \eqref{bound11} that
\beqn
\label{g}
P_\bx(\ol G_N)\leq 2T Me^{-\frac{N\veps_\theta^2}{2}}\qquad\qquad  \forall\,\bx\in K_{\theta,\eta}.
\feqn
For $\bx\in K_{\theta,\eta},$ write
\beq
P_\bx(\cale_N)=P_\bx(\cale_N; G_N)+P_\bx(\cale_N; \ol G_N),
\feq
and notice that
\beq
\lim_{N\to\infty}P_\bx(\ol G_N)=0
\feq
by virtue of \eqref{g}. Moreover,
\beqn
\label{simsim}
P_\bx(\cale_N; G_N)=P_\bx(\cale_N\,|\,G_N)P_\bx(G_N)=P_{\mu_{N,\bx}}(\cale_N)P_\bx(G_N),
\feqn
where $\mu_{N,\bx}$ is the distribution of $\bXn_{\xi_N}$ under the conditional measure $P_\bx(\cdot\,|\,G_N),$ with $\xi_N$ introduced in \eqref{xg}.
Hence, in order to prove the first part of the theorem, it suffices to show that for any sequence of (discrete) probability measures $\mu_N,$ each supported on
$\caln_\theta\bigcap \dmn$ with the corresponding $N,$ we have
\beq
\lim_{N\to\infty} P_{\mu_N}(\cale_N)=1
\feq
Since
\beq
P_{\mu_N}(\cale_N)=\sum_{\by\in\caln_\theta\bigcap \dmn} P_\by(\cale_N)\,\mu_n(\by),
\feq
it suffices to show that
\beq
\lim_{N\to\infty} P_\by(\cale_N)=1,
\feq
and the convergence is uniform over $\by\in \caln_\theta.$ More precisely,
\begin{lemma}
\label{delta3}
For any $\veps>0$ there exists an integer $N_0=N_0(\veps)>0$ such that
\beq
P_\by(\cale_N)\geq 1-\veps
\feq
for all $N>N_0$ and $\by\in \caln_\theta \bigcap \dmn.$
\end{lemma}
\begin{proof}[Proof of Lemma~\ref{delta3}]
Recall $L_N$ from \eqref{trout}. If $L_N\geq 1,$ let $S_{N,1},\ldots, S_{L_N,N}$ be successive return times of the Markov chain $\bXn$
to $\caln_\theta.$ Namely, we set $S_{0,N}=0$ and define recursively,
\beq
S_{k,N}=\inf\{j>S_{k-1,N}: \bXn_j\in \caln_\theta\}.
\feq
As usual, we assume that $\inf \emptyset =+\infty.$ Let $S_{L_N+1,N}$ be the first hitting time of the boundary:
\beq
S_{L_N+1,N}=\inf\{j>S_{L_N,N}:\bXn_j\in \partial(\Delta_M)\}.
\feq
Finally, for $k\in\zz_+$ define
\beq
\bYn_k:=\left\{
\begin{array}{ll}
\bXn_{S_{k,N}}~&\mbox{\rm if}~k\leq L_N+1
\\
[2mm]
\bXn_{S_{L_N+1,N}}~&\mbox{\rm if}~k> L_N+1.
\end{array}
\right.
\feq
Then $\big(\bYn_j\big)_{j\in\zz_+}$ is a Markov chain on $\big(\caln_\theta \bigcup \partial(\Delta_M)\big)\bigcap \dmn$
with all states on the boundary $\partial\big(\dmn\big)$ being absorbing. Denote by $Q_N$ transition kernel of this chain.
\par
Recall $V_M^*$ from \eqref{vm1}. We then have for any $\bx\in \caln_\theta \bigcap \dmn:$
\beqn
\nonumber
P_\bx(\cale_N)&\geq& \sum_{m=0}^\infty \sum_{\by\in \caln_\theta\bigcap \dmn} \sum_{\bz\in V_M^*} Q_N^m(\bx,\by)P_N (\by,\bz)\geq (1-\alpha-\theta)^N
\sum_{m=0}^\infty \sum_{\by\in \caln_\theta\bigcap \dmn} Q_N^m(\bx,\by)
\\
\label{239}
&=&
(1-\alpha-\theta)^N \sum_{k=0}^\infty P_\bx(L_N \geq k)=(1-\alpha-\theta)^N\big(1+E_\bx(L_N)\big),
\feqn
where $L_N$ is the number of returns to $\caln_\theta$ introduced in \eqref{trout}.
The above formula bounds $P_\bx(\cale_N)$ from below by the probability of the event ``$\bXn$ hits the boundary at one of the vertices
of the set $J^*,$ moreover, the last site visited by $\bXn$ before passing to the boundary lies within $\caln_\theta$".
The index $m$ in \eqref{239} counts the number of returns to $\caln_\theta,$ $\by$ is the last visited state at $\Delta_M^\circ,$ and the estimate
\beq
P_N(\by,\bz)\geq (1-\alpha-\theta)^N,\qquad \qquad \by\in \caln_\theta,\, \bz\in J^*,
\feq
is a direct implication of \eqref{model} and the definition of $\caln_\theta.$
\par
To complete the proof of Lemma~\ref{delta3}, we will now decompose an (almost surely) certain event ``$\bXn$ hits the boundary eventually" in a way similar
to \eqref{239} and compare two estimates. Let (cf. \eqref{etad})
\beq
\partial_{\theta,\eta}:=\big\{x\in\Delta_M^\circ\backslash\caln_\theta:\max_{j\not\in V_M^*}x(j)< \eta\big\},
\feq
so that the interior of the simplex partitions as $\Delta_M^\circ=\caln_\theta \bigcup K_{\theta,\eta}\bigcup \partial_{\theta,\eta}.$ Similarly to \eqref{239}, for the probability to eventually hit the boundary starting from a point $\bx\in \caln_\theta\bigcap \dmn,$ we have:
\beqn
\nonumber
1&\leq &
\sum_{m=0}^\infty \sum_{\by\in \caln_\theta\bigcap \dmn} \sum_{\bz\in V_M^*} Q_N^m(\bx,\by)P_N (\by,\bz)
\\
\nonumber
&&
\quad
+
\sum_{m=0}^\infty \sum_{\by\in \caln_\theta\bigcap \dmn} \sum_{\bz\in \partial_{\theta,\eta}\backslash V_M^*} Q_N^m(\bx,\by)P_N (\by,\bz)
\\
\nonumber
&&
\quad
+
\sum_{m=0}^\infty \sum_{\by\in \caln_\theta\bigcap \dmn}
\sum_{z\in \dmn\backslash (N_\theta\bigcup \partial_{\theta,\eta})} Q_N^m(\bx,\by)P_N (\by,\bz)P_\bz\big(\ol G_N\big)
\\
\nonumber
&\leq &
\sum_{m=0}^\infty \sum_{\by\in \caln_\theta\bigcap \dmn} \sum_{\bz\in V_M^*} Q_N^m(\bx,\by)P_N(\by,\bz)
\\
\label{239a}
&&
\quad
+
(1-\beta+\theta)^{\eta N}\big(1+E_\bx(L_N)\big)
+
\big(1+E_\bx(L_N)\big)\cdot 2M Te^{-\frac{N\veps_\theta^2}{2}},
\feqn
where we used \eqref{g} to estimate $P_\bz\big(\ol G_N\big).$
\par
It follows from \eqref{239} and the assumptions of theorem that
\beq
&&
\frac{(1-\beta+\theta)^{\eta N}\big(1+E_\bx(L_N)\big)
+
\big(1+E_\bx(L_N)\big)\cdot 2MT e^{-\frac{N\veps_\theta^2}{2}}}
{\sum_{m=0}^\infty \sum_{\by\in \caln_\theta\bigcap \dmn} \sum_{\bz\in V_M^*} R_N^m(\bx,\by)P_N (\by,\bz)}
\\
&&
\qquad
<
\frac{(1-\alpha-\theta)^N\big(1+E_\bx(L_N)\big)
+
\big(1+E_\bx(L_N)\big)\cdot 2MT (1-\alpha-\theta)^N}
{\sum_{m=0}^\infty \sum_{\by\in \caln_\theta\bigcap \dmn} \sum_{\bz\in V_M^*} R_N^m(\bx,\by)P_N (\by,\bz)}
\longrightarrow_{N\to\infty} 0.
\feq
Hence, taking now in account both \eqref{239} and \eqref{239a},
\beq
\lim_{N\to\infty} P_\bx(\cale_N)=\lim_{N\to\infty}\bigg(
\sum_{m=0}^\infty \sum_{\by\in \caln_\theta\bigcap \dmn} \sum_{\bz\in V_M^*} R_N^m(\bx,\by)P_N (\by,\bz)\bigg)
=1.
\feq
The proof of Lemma~\ref{delta3} is complete.
\end{proof}
This completes the proof of part (i) of the theorem.
\end{proof}
$\mbox{}$
\begin{proof}[Proof of (ii)]
To show that $\lim_{N\to\infty}E(\nu_N)=+\infty,$ write, similarly to \eqref{simsim},
\beq
E(\nu_N)&\geq&E(\nu_N;G_N)=E_{\mu_N}(\nu_N)P(G_N),
\feq
where $\mu_N$ is the distribution of $\bXn_{\xi_N}$ under the conditional measure $P(\cdot\,|\,G_N),$ with $\xi_N$ introduced in \eqref{xg}.
Using Proposition~\ref{key} with a compact set $K$ that includes $\bx_0$ (the limit of the initial state of $\bXn$ introduced in the conditions of the theorem) and leveraging the same argument that we employed in order to obtain \eqref{g}, one can show that $\lim_{N\to\infty} P(G_N)=1.$  Furthermore,
the estimates in \eqref{239a} show that $L_N,$ the number of returns to $\caln_\theta$ under $P_{\mu_N}$ is stochastically dominated  from
below  by a geometric variable with probability of success $e^{-\gamma N}$ for any $\gamma>0$ sufficiently small. Thus,
\beq
E_{\mu_N}(\nu_N) \geq E_{\mu_N}(L_N)\geq e^{\gamma N}.
\feq
Therefore, $\lim_{N\to\infty} E_{\mu_N}(\nu_N)=+\infty,$ and the proof of the second part of the theorem is complete.
\end{proof}
\subsection{Proof of Proposition~\ref{radius}}
\label{pr}
Write
\beq
\Gamma_i(\bx)=\frac{1+B x(i)}{1+ \bx^T\bB\bx}x(i),
\feq
where $\bB=\frac{\omega}{1-\omega}\bA.$ By the definition of the interior equilibrium, there is a constant $r>0$ such that
\beqn
\label{ins}
B\mz(i)=r\qquad \forall\,i\in S_M.
\feqn
Therefore, for $\bz=\bmz+ \bu\in \Delta_M^\circ,$ as $\bu$ approaches zero in $\rr^M,$ we have:
\beq
\Gamma_i(\bz)=\frac{1+r+B u(i)}{1+r+\bu^T\bB \bmz+\bmz^T \bB\bu}\big(\mz(i)+u(i)\big) +o(\|\bu\|).
\feq
Since $\bB^T=\bB$ and $\sum_{i\in S_M}u(i)=0,$ we have:
\beq
\bmz^T\bB\bu=\bu^T \bB \bmz=r\sum_{i\in S_M}u(i)=0.
\feq
Thus, as $\bu$ approaches $\bo\in\rr^M,$
\beq
\Gamma_i(\bz)&=&\frac{1+r+B u(i)}{1+r}\big(\mz(i)+u(i)\big) +o(\|u\|)
\\
&=&
\big(\mz(i)+u(i)\big)+\frac{\mz(i)\cdot Bu(i)}{1+r}+o(\|u\|).
\feq
Since $\bGamma$ is a diffeomorphism \cite{losakin}, we can compare this formula with the first-order Taylor expansion of $\bGamma$ around $\bmz.$
It follows that the Jacobian matrix of $\bGamma$ evaluated at $\bmz$ is given by
\beqn
\label{525}
D_{ij}:=\frac{\partial \Gamma_i}{\partial x_j}(\bmz)=\delta_{ij} +\frac{\mz(i)B_{i,j}}{1+r},\qquad i.j\in S_M,
\feqn
where $\delta_{ij}$ is the Kronecker symbol. It is easy to verify that
\beq
\bH:=\frac{1+r}{1+2r}\bD
\feq
is a transpose of a strictly positive stochastic matrix, namely $H_{ij}>0$ for all $i,j\in S_M$ and $\sum_{i\in S_M} H_{ij}=1$ for all $j\in S_M.$
\par
Note that $\bmz$ is a left eigenvector of $\bH^T$ corresponding to the principle eigenvalue $1.$  Indeed,
assuming $i\neq j,$ we deduce from \eqref{525} and the fact that $\bB^T=\bB:$
\beq
\mz(i) H^T_{ij}=\frac{1+r}{1+2r}\frac{\mz(i)B_{i,j}\mz(j)}{1+r}=\mz(j) H^T_{ji}.
\feq
Thus, stochastic matrix $\bH^T$ is transition kernel of a reversible Markov chain with stationary distribution $\bmz.$
From the reversibility of the Markov transition matrix $\bH^T,$ it follows that all the eigenvalues of $\bH,$ and hence also of $\bD,$ are real numbers (see, for instance, \cite{brem}).
\par
Recall $W_M$ from \eqref{seq}. The asymptotic stability of the equilibrium (see the paragraph above \eqref{asa}) implies that
the largest absolute value of an eigenvalue of the linear operator $\bD$ restricted to the linear space $W_M$ is at most one \cite{perko}. It follows from \eqref{525} that if $\lambda$ is this eigenvalue and $|\lambda|=1,$ then there is an eigenvalue $\eta$ of
the matrix with entries $\frac{\mz(i)B_{i,j}}{1+r}$ such that $\eta\in\{-2,0\}.$ To complete the proof of the proposition it remains to show that
this is indeed impossible. Toward this end, denote by $\bH_\bchi$ the matrix with entries $\frac{\mz(i)B_{i,j}}{1+r}$ and observe that
\beq
\det (\bH_\bchi)=(1+r)^{-M}\prod_{i=1}^M \mz(i)\cdot \det \bB.
\feq
Since $\bB$ is invertible by our assumptions, this shows that $0$ cannot be an eigenvalue of $\bH_\chi.$ Furthermore, if $\eta=-2$
is an eigenvalue of $\bH_\bchi,$ then there exists $\bnu\in W_M$ such that
\beq
\bmz\circ \bB\bnu=-2(1+r)\bnu,
\feq
where $\circ$ denotes the element-wise product introduced in \eqref{ep}. Recall now that $\bB=h\bA,$ where
$h=\frac{\omega}{1-\omega}.$ Plugging this expression into the previous formula, we obtain
\beqn
\label{cnn}
\bmz\circ \bA\bnu=-2(1+r)h^{-1}\bnu
\feqn
Now, choose a constant $\beta>1$ so close to one that $\witi \bA=\beta \bA$ satisfies the conditions of Theorem~\ref{bsa} in Appendix~B, namely $\witi \bA$ is invertible
and has exactly one positive eigenvalue.
\par
Denote $\witi \bB:=h\witi \bA=\beta \bB.$ The analogue of \eqref{ins} for $\witi \bB$ reads
\beq
\witi B\mz(i)=\beta B\mz(i)=r\beta\qquad \forall\,i\in S_M.
\feq
In this sense, $r\beta$ plays the same role for the pair $\witi \bA$ and $\witi \bB$ as $r$ does for
$\bA$ and $\bB.$ Moreover, \eqref{cnn} implies that
\beq
\frac{1}{1+r\beta }\bmz \circ \witi \bB\bnu=-2\frac{\beta(1+r)}{1+r\beta} \bnu.
\feq
Since $\big|\frac{\beta(1+r)}{1+r\beta}\big|>1,$
it follows that the spectral radius of the Jacobian of the vector field
$\witi \bGamma(\bx):=\bx\circ \frac{1+h\witi \bA\bx}{1+h\bx\witi \bA\bx}$ evaluated at $\bmz$ is greater than one.
This is a clear contradiction, since matrix $\witi \bA$ is assumed to satisfy the conditions of Theorem~\ref{bsa}, and hence the spectral radius must be at most one.
\hfill\hfill\qed
\subsection{Proof of Corollary~\ref{metac}}
\label{pmet}
For $\veps>0,$ denote $\calun_\veps:=\big\{\bx\in \dmn:\|\bx-\bmz\|<\veps\big\}.$
Since $\bGamma$ is a diffeomorphism under the conditions stated in Assumption~\ref{assume7} \cite{losakin},
Proposition~\ref{radius} implies that there exist $\veps>0$ and $\rho\in(0,1)$ such that the spectral radius of the Jacobian is strictly less than
$\rho$ within $\calun_\veps.$ Thus, $\bGamma(\calun_\veps)\subset \calun_{\rho\veps}.$
Without loss of generality we may assume that $N$ is so large that both $\calun_{\rho\veps}$ and
$\calun_\veps\backslash \calun_{\rho\veps}$ are nonempty. Indeed, for any $N_0\in\nn$ and $b>0$ the claim is trivially true for all $N\leq N_0$
if we choose $C>0$ in such a way that $1-Ce^{-bN_0}<0.$
\par
Since $\bGamma(\calun_\veps)\subset \calun_{\rho\veps}$ and $P_N(\bx,\by)>0$ for all $\bx,\by\in\dmn^\circ$
(and hence, by virtue of \eqref{ln1}, the quasi-stationary distribution $\mu_N$ puts strictly positive weights on all $\bx\in\dmn^\circ$),
for suitable $C>0$ and $b>0$ we have:
\beqn
\nonumber
\lambda_N \mu_N\big(\calun_{\rho\veps}\big)&=&\mu_N P_{\circ,N}\big(\calun_{\rho\veps}\big)=\sum_{\bx\in \dmn^\circ} \sum_{\by\in \calun_{\rho\veps}} \mu_N(x)P_N(\bx,\by)
\\
\label{cruc}
&\geq& \sum_{\bx\in \calun_\veps} \sum_{\by\in \calun_{\rho\veps}} \mu_N(\bx)P_N(\bx,\by)
\geq \mu_N\big(\calun_\veps)\big(1-Ce^{-b N}\big)
\\
\nonumber
&>&
\mu_N\big(\calun_{\rho \veps}\big)\big(1-Ce^{-b N}\big),
\feqn
which implies the claim. To obtain the crucial second inequality in \eqref{cruc}
\beq
\sum_{\by\in \calun_{\rho\veps}} P_N(\bx,\by)=P_N\big(\bx,\calun_{\rho\veps}\big)\geq 1-Ce^{-bN}
\feq
one can rely on the result of Proposition~\ref{key} and use an argument similar to the one which led us to \eqref{g}.

\hfill\hfill\qed
\section{Conclusion}
\label{conclude}
In this paper we examined the long-term behavior of a generalized Wright-Fisher model. Our results hold for a broad class
of Wright-Fisher models, including several instances that have been employed in experimental biology applications.
Our primary motivation for this generalized Wright-Fisher model came from research conducted on
gut gene-microbial ecological networks \cite{pro}.
\par
The main results of the paper are:
\begin{itemize}
\item[-] A maximization principle for deterministic replicator dynamics, an analogue of Fisher's fundamental theorem of natural
selection, stated in Theorem~\ref{mdns}.
In this theorem, we show that the average of a suitably chosen reproductive fitness of the mean-field model is non-decreasing with time,
and is strictly increasing outside of a certain chain-recurrent set. The proof relies on a ``fundamental theorem of dynamical systems" of Conley \cite{conley}.
\item[-] Proposition~\ref{cava}, which is an implication of the result in Theorem~\ref{mdns} for stochastic dynamics of a class
of  Wright-Fisher models.
\item[-] Gaussian approximation scheme constructed for the generalized Wright-Fisher model in Theorem~\ref{da}.
\item[-] Theorem~\ref{th1}, which gives an exponential lower bound on the first time that the error of mean-field approximation exceeds a given threshold.
\item[-] Theorem~\ref{xit}, which specifies the route to extinction in the case that the following three conditions are satisfied:
1) the mean-field dynamical system has a unique interior equilibrium; 2) boundary of the probability simplex is repelling for the mean-field map;
and 3) the stochastic mixing effect (in certain rigorous sense) dominates extinction forces.
\item[-] Proposition~\ref{radius} that elaborates on the mechanism of convergence to the equilibrium in the mean-field model of Theorem~\ref{xit}.
\end{itemize}
We next outline several directions for extensions and future research.
\begin{enumerate}
\item The competitive exclusion principle predicts that only the fittest species will survive in a competitive environment.
In \cite{hutch}, Hutchinson addressed the \textit{paradox of the plankton} which is a discrepancy between this principle and the observed diversity of the plankton.
One of the suggested resolutions of this paradox is that the environmental variability might promote the diversity.
Several mathematical models of species persistence in time-homogeneous environments have been proposed \cite{benspers, changing, re,  re-tumor, coex, re1}.
In order to maintain the diversity, environmental factors must fluctuate within a broad range supporting a variety of contradicting trends and thus promoting the survival of various species. The theory of permanence and its stochastic counterpart, the theory of stochastic persistence \cite{benspers}, is an alternative (but not necessarily mutually excluding)
mathematical framework able to explain the coexistence of interacting species in terms of topological properties of the underlying dynamical system.
Extension of our results to systems in fluctuating environments seems to be a natural direction for future research with important implications regarding the dynamical interactions
between microbial communities, host gene expression, and nutrients in the human gut.

\item We believe that Theorem~\ref{xit} can be extended to a broad class of stochastic population models with mean-field vector-field $\bGamma$ promoting
cooperation between the particles (cells/microorganisms in applications). Among specific biological mechanisms promoting forms of cooperation are:
public goods \cite{nonlim, spgoods, altruism, pgoods5, pgoods, tale}, relative nonlinearity of competition \cite{nonl3, updates, nlpromote}, population-level randomness \cite{altruism, stat65, kroumis}, and diversification of cells into distinct ecological niches \cite{tinkering, aggreg}. Public goods models are of a special interest for us because of a possibility of an application to  to the study of a gene-microbial-metabolite ecological network in the human gut \cite{pro}. A public good is defined in \cite{nonlim} as any benefit that is ``simultaneously non-excludable (nobody can be excluded from its consumption) and non-rivalrous (use of the benefit by one particle does not diminish its availability to another particle)". For instance, an alarm call against a predator is a pure public good while natural resources (e.\,g. the air) and enzyme production in bacteria can be approximated to public goods \cite{nonlim}. Public goods games in microbial communities have been considered by several authors, see \cite{nonlim, mgame9, spgoods, tale} for a review. An important feature of the public goods game is the intrinsic non-linearity of the fitness \cite{nonlim, tale}.

\item In contrast to non-rivalrous interactions captured by public goods games, competitive interaction between cells can lead to negative frequency dependence,
see \cite{negac, nlinear} and references therein. A classical example of negative frequency dependence is animals foraging on several food patches \cite{nlinear}.
We believe it is of interest to explore the path to extinction of this class of models in the spirit of analysis undertaken in Section~\ref{exits}.

\item An interesting instance of the Wright-Fisher model with an replicator update rule was introduced in \cite{kroumis}.
This is a mathematical model for cooperation in a phenotypic space, and is endowed with a reach additional structure
representing phenotypic traits of the cells in addition to the usual division into two main categories/types, namely cooperators and defectors.
In the exact model considered in \cite{kroumis}, the phenotypic space is modelled as a Euclidean space $\zz^d.$ However, a variation
with finite population (for instance, restricting the phenotypic space to a cube in $\zz^d$) can be studied in a similar manner. Extending the result of Theorem~\ref{xit}
to a variation of this Wright-Fisher model is challenging, but could potentially contribute to a better understanding of
mechanisms forming extinction trajectory in stochastic population models.

\item In Section~\ref{exits} we investigated the phenomenon of extinction of a long-lived closed stochastic population.
The manner in which the population is driven to extinction depends on whether the extinction states are repellers or attractors of the mean-field dynamics
(cf. \cite{assaf, assaf1, metapark}).
In the latter case the extinction time is short because with a high probability the stochastic trajectory follows the mean-field path closely
for a long enough time to get to a close proximity (in a suitable phase space) of the extinction set, and then it fixates there via small noise fluctuations.
In contrast, in the former scenario, attractors of the deterministic dynamics act as traps or pseudo-equilibria for the stochastic model and thus create metastable stochastic states.
While Theorem~\ref{xit} is concerned with a model of this kind, Conjecture~\ref{conja} indicates the possibility of extending this result to a situation where all
attractors of the mean-field model belong to the boundary of the simplex $\Delta_M.$

\item The proof of Theorem~\ref{da} relies on a representation of $\bXn_k$ as partial sums of i.\,i.\,d. random indicators (see equations \eqref{z} and \eqref{c}), and can be
extended beyond the i.\,i.\,d. setting. In particular, it can be generalized to cover the Cannings exchangeable model and models in fluctuating environments. We believe that
with a proper adjustment, the result also holds for a variation of the model with $N$ depending on time \cite{gompert, hulu, isuka} (this is, for instance, well known for
branching processes with immigration). This theorem can also be extended to the \textit{paintbox models}  \cite{paintbox} and to a class of Cannings models in phenotypic spaces with weak and fast decaying dependency between phenotypes \cite{kroumis}.

\item Finally, proving Conjectures~\ref{coca} and~\ref{coca1} would shed further light on the relations among metastability, quasi-stationary distributions,
and the extinction trajectory for the underlying Wright-Fisher model.
\end{enumerate}
\section*{Acknowledgements}
We would like to thank Ozgur Aydogmus for a number of illuminating discussions. This work was supported by the
Allen Endowed Chair in Nutrition and Chronic Disease Prevention, and the National Institute of Health (R35-CA197707 and P30E5029067).
A.~R. is supported by the Nutrition, Statistics \& Bioinformatics Training Grant T32CA090301.

\section{Appendices}
\subsection*{Appendix A. Absorption states and stochastic equilibria}
In this appendix we explore the connection between
basic geometric properties of the vector field $\bGamma$ and the structure of closed communication classes of the Markov chain $\bXn.$ In other words, we study the configuration of absorbing states and identify possible supports of stationary distributions according to the properties of $\bGamma$.  We remark that the results of this section are included for the sake of completeness and are not used anywhere else in the reminder of the paper. We have made the section self-contained and included the necessary background on the Markov chain theory.
\par
The main result of the appendix is stated in Theorem~\ref{bc} where a geometric characterization of recurrent classes is described. The general structure
of recurrent classes of $\bXn$ in this theorem holds universally for the Wright-Fisher model \eqref{model} with any update function $\bGamma.$
\par
Recall \eqref{sjn}. The discussion in this section is largely based on the fact that
\beqn
\label{gr}
P_N(\bx,\by)>0 \quad
\Longleftrightarrow \calc(\by)\subset \calc\big(\bGamma(\bx)\big)\quad
\Longleftrightarrow
\quad P_N(\bx,\bz)>0~\mbox{\rm for all}~\bz\in \Delta_{[\calc(\by)],M}.
\feqn
This simple observation allows one to relate the geometry of zero patterns in the vector field to communication properties, and hence asymptotic behavior, of $\bXn.$
\par
From the point of view
of the model's adequacy in potential applications, \eqref{gr} is a direct consequence of the fundamental assumption of the Wright-Fisher model
that  particles in the population update their phenotypes independently of each other and follow the same stochastic protocol. Mathematically, this
assumption is expressed in identity \eqref{c}, whose right-hand side is a sum of independent and identically distributed random indicators.
To appreciate this feature of the Wright-Fisher model, we remark that while \eqref{gr}
holds true for an arbitrary model \eqref{model}, that feature does not hold true, in general, for two most natural generalization
of the Wright-Fisher model, namely the Cannings exchangeable model \cite{ewens} and the pure-drift Generalized Wright-Fisher process of \cite{app17}.
\par
Recall that two states $\bx,\by\in \dmn$ of the finite-state Markov chain
$\bXn$ are said to communicate if there exist $i,j\in\nn$ such that $P_N^i(\bx,\by)>0$ and $P_N^j(\by,\bx)>0.$
If $\bx$ and $\by$ communicate, we write $\bx\leftrightarrow \by.$ The binary relation $\leftrightarrow$ partitions the state space $\dmn$ into a finite number of disjoint equivalence classes, namely $\bx,\by$ belong to the same class if and only if $\bx\leftrightarrow \by.$  A
state $\bx\in \dmn$ is called absorbing if $P_N(\bx,\bx)=1,$ recurrent if $\bXn$ starting at $\bx$ returns to $\bx$ infinitely often with probability one,
and transient if, with probability one, $\bXn$ starting at $\bx$ revisits $\bx$ only a finite number of times. It turns out that each communication class
consists of either 1) exactly one absorbing state; or 2) recurrent non-absorbing states only; or 3) transient states only. No other behavior is possible due
to zero-one laws enforced by the Markov property \cite{lawler}.  A communication class is called closed or recurrent if it belongs to the first or second category, and transient if it belongs to the third one. The chain is called irreducible if there is only one (hence, recurrent) communication class, and aperiodic if $P_N^j(\bx,\by)>0$ for some $j\in\nn$ and all $\bx,\by\in\dmn.$ If $C_1,\ldots,C_r$ are disjoint recurrent classes of $\bXn,$ then the general form of its stationary distribution is $\pi(\bx)=\sum_{i=1}^r \alpha_i \pi_i(\bx)$ where $\pi_i$ is the unique stationary distribution supported on $C_i$ (that is $\pi_i(\bx)>0$ if and only if $\bx\in C_i$) and $\alpha_i$ are arbitrary non-negative numbers adding up to one. If the chain is irreducible, then the stationary distribution is unique and is strictly positive on the whole state space $\dmn.$ An irreducible finite-state Markov chain is aperiodic if and only if $\pi(\bx)=\lim_{k\to\infty} P_N^k(\by,\bx)$ for any $\bx,\by\in\dmn,$ where $\pi$ is the unique stationary distribution of the chain \cite{lawler}.
\par
Let $\cala_N \subset \Delta_M$ denote the set of absorbing states of the Markov chain
$\bXn.$ It readily follows from \eqref{model} that $\cala_N\subset V_M,$ and a vertex $\be_j\in V_M$ is an absorbing (fixation) state  if and only if $\bGamma(\be_j)=\be_j.$ In particular, $\cala_N$ depends only on the update rule $\bGamma$ and is independent of the population size $N.$
We summarize this observation as follows:
\beqn
\label{abs}
\cala_N=\big\{\be_j\in V_M:\bGamma(\be_j)=\be_j\big\}\qquad \qquad \forall\, N\in\nn.
\feqn
The following lemma asserts that if $\be_j$ is an absorbing state of $\bXn,$ then the type
$j$ cannot be mixed into any stochastic equilibrium of the model, namely any other than $\be_j$ state $\bx\in \dmn$ with $x(j)>0$ is transient.
\begin{lemma}
If $\bx\in \dmn$ is a recurrent state of $\bXn$ and $x(j)>0,$ then $\be_j$ is
a recurrent state and belongs to the same closed communication class as $\bx.$
\end{lemma}
\begin{proof}
By \eqref{abs}, a state $\bx\in\dmn$ cannot be absorbing if $\bx\neq \be_j$ and $x(j)>0.$
If $\bx$ is recurrent and non-absorbing, then there is a state $\by\in \dmn$ that belongs to the same
closed communication class as $\bx$ and such that $P_N(\by,\bx)>0.$ In view of \eqref{model} this implies that $\Gamma_j(\by)>0$
and hence $P_N(\by,\be_j)>0.$ Since the state $\by$ is recurrent, $\be_j$ must belong to the same communication class as $\by$ and $\bx,$
and in particular is recurrent. The proof is complete.
\end{proof}
Recall \eqref{sjn}. The same argument as in the above lemma shows that if an interior point $\bx$ of a simplex $\djn$ is recurrent
for $\bXn,$ then the whole simplex $\djn$ belongs to the (recurrent) closed communication class of $\bx.$
One can rephrase this observation as follows:
\begin{theorem}
\label{bc}
Any recurrent class of the Markov chain $\bXn$ has the form of a simplicial complex $\bigcup_\ell \Delta_{[J_\ell],N},$ where $J_\ell$
are (possibly overlapping) subsets (possibly singletons) of $S_M.$
\end{theorem}
With applications in mind, we collect some straightforward implications of this proposition in the next corollary.
\begin{corollary}
\label{bcc}
\item [(i)] If there is a stationary distribution $\pi$ of $\bXn$ and $J\subset S_M$ such that $\pi(\bx)>0$ for some
interior point $\bx\in\djn^\circ,$ then $\pi(\by)>0$ for all $\by\in \djn.$  Furthermore, the following holds in this case
unless $\Gamma(\by)\cdot \by=0$ (that is, $\by$ and $\bGamma(\by)$ are orthogonal) for all $\by\in \djn:$
\begin{itemize}
\item [(a)] Let $C$ be the (closed) communication class to which the above $\bx$ (and hence the entire $\djn$) belongs. Then
the Markov chain $\bXn$  restricted to $C$ is aperiodic.
\item [(b)] Let $T$ be the first hitting time of $C,$ namely $T=\inf\{k\in\zz_+:\bXn_k\in C\}.$
As usual, we use here the convention that $\inf \emptyset =+\infty.$ Then
\beq
\pi(\by)=\lim_{k\to\infty} P(\bXn_k=\by|\bXn_0=\bz,T<\infty)
\feq
for all $\by\in C$ and $\bz\in\dmn.$ In particular,
$\pi(\by)=\lim_{k\to\infty}P_N^k(\bz,\by)$ whenever $\by,\bz\in C.$
\end{itemize}
\item [(ii)] A stationary distribution $\pi$ of $\bXn$ such that $\pi(\bx)>0$ for some interior point
$\bx\in \dmn^\circ$ exists if and only if $\bXn$ is irreducible and aperiodic,
in which case the stationary distribution is unique and $\pi(\by)=\lim_{k\to\infty} P^k_N(\bx,\by)>0$ for any $\bx,\by\in\dmn.$
\item [(iii)] Let $B=\big\{\exists~k\in\nn: \bXn_k\in \partial\big(\dmn\big)~\mbox{\rm for all}~m\geq k\big\}$
be the event that the trajectory of $\bXn$ reach eventually the boundary of the simplex and stays there forever.
Then,  either the Markov chain $\bXn$ is irreducible and aperiodic or $P(B)=1$ for any state $\bXn_0.$
\item [(iv)] If $\bGamma(\be_j)=\be_j$ for some $j\in S_M,$ then $\lim_{k\to\infty} \Xn_k(j)\in\{0,1\},$ \as,
for any initial distribution of $\bXn.$
\end{corollary}
We will only prove the claim in (a) of part (i), since the rest of the corollary is an immediate result of a direct combination of Theorem~\ref{bc}
and general properties of finite-state Markov chains reviewed in the beginning of the section.
\begin{proof}[Proof of Corollary~\ref{bcc}-(i)-a]
Suppose that $\bGamma(\by)\cdot \by>0$ for some $\by\in \djn.$ Let $\bz$ be the projection of $\by$ into $\Delta_{[I],M},$
where $I:=\{j\in J:y(j)\Gamma_j(\by)>0\}.$ That is,
\beq
z(i)=
\left\{
\begin{array}{ll}
y(i)&\mbox{\rm if}~y(i)\Gamma_i(\by)>0,
\\
0&\mbox{\rm otherwise}.
\end{array}
\right.
\feq
Then $\bz\in \Delta_{[I],M}^\circ$ and $P_N(\bz,\bz)>0.$ Since $\bz\in \djn,$ by Proposition~\ref{bc}
$\bz$ belongs to the same closed communication class as $\bx.$ The asserted aperiodicity of $\bXn$ restricted to $C$
is a direct consequence of the existence of the $1$-step communication loop at $\bz$ \cite{lawler}.
\end{proof}
\begin{example}
[\cite{entropy}]
If $\Gamma_j(\bx)=x(j)^\alpha f(\bx)$ for some $j\in S_M,$ function $f:\Delta_M\to \rr_+$ such that $f(\be_j)=1,$
and a constant $\alpha>0,$  then $\bGamma(\be_j)=\be_j.$ By part (iv) of Corollary~\ref{bcc}, this fact alone suffices to conclude
that $\lim_{k\to\infty} \Xn_k(j)\in\{0,1\},$ \as, for this specific type $j.$
\end{example}
A simple sufficient condition for the Markov chain $\bXn$ to be irreducible and aperiodic is that $\bGamma(\Delta_M)\subset
\Delta_M^\circ,$ in which case $P_N(\bx,\by)>0$ for any $\bx,\by\in \dmn.$ If $\bGamma\big(\dmn^\circ\big)\subset
\dmn^\circ$ but $\bGamma(\be_j)=\be_j$ for all $j\in S_M,$ the model can be viewed as a multivariate analog of what is termed in \cite{kimurac} as a Kimura class of Markov chains.  The latter describes evolution dynamics of a population that initially consists of two phenotypes, each one eventually either vanishes or takes over the entire population.
\subsection*{Appendix B. Permanence of the mean-field dynamics}
In this appendix, we focus on Wright-Fisher models with a permanent mean-field dynamics $\bGamma.$ The permanence means
that the boundary of the state space $\partial(\Delta_M)$ is a repeller for the mean-field dynamics, implying that all types under the mean-field dynamics will ultimately survive \cite{garay, HS, hutson, ckperm}. This condition is particularly relevant to the stochastic Wright-Fisher model because, as stated in Appendix~A, all absorbing states (possibly an empty set) of the Markov chain $\bXn$ lie within the boundary set $\partial(\Delta_M).$ Of course, in the absence of mutations, any vertex of $\Delta_M$ is an absorbing state for the Markov chain $\bXn,$ and the latter is going to eventually fixate on one of the absorbing states
with probability one (see Appendix~A for formal details). However, we showed in Section~\ref{mainr}
that a permanent mean-field dynamics induces metastability and a prolonged coexistence of type.
\begin{definition}[\cite{HS}]
$\bGamma$ is permanent if
there exists $\delta>0$ such that
\beq
\liminf_{k\to\infty} \psi_k(i)>\delta \qquad \forall\,i\in S_M,\bpsi_0\in \Delta_M^\circ.
\feq
\end{definition}
\par
See \cite{garay, coexit} and references therein for sufficient conditions for the permanence.
It turns out that under reasonable additional assumptions, $\bGamma$ is permanent for the fitness defined in either \eqref{lf4} or \eqref{ef4}.
More precisely, we have (for a more general version of the theorem, see \cite[Theorem~11.4]{garay} and \cite[Theorem~5]{rewhs}):
\begin{theorema}
[\cite{garay}]
\label{ept}
Let $\bA$ be an $M\times M$ matrix which satisfies the following condition: There exists $\by\in \Delta_M^\circ$ such that
\beq
\by^T\bA \bz> \bz^T\bA\bz
\feq
for all $\bz\in \partial(\Delta_M)$ solving the fixed-point equation $\bz=\frac{1}{\bz^T\bA\bz}\big(\bz\circ \bA\bz\big).$
\par
Then:
\item [(i)] There exists $\omega_0>0$ such that $\bGamma$ given by \eqref{replica} and \eqref{lf4} is permanent for all $\omega\in(0,\omega_0).$
\item [(ii)] $\bGamma$ given by \eqref{replica} and \eqref{ef4} is permanent for all $\beta>0.$
\end{theorema}
\par
In Section~\ref{exits} considered a stochastic Wright-Fisher model with a permanent mean-field map $\bGamma.$ This model has a global attractor, consisting of a unique equilibrium point. A sufficient condition for such a dynamics is given by the following theorem \cite{losakin}.
Recall \eqref{seq}.
\begin{theorema}
[\cite{losakin}]
\label{bsa}
Let Assumption~\ref{assume7} hold. Then the following holds true:
\begin{itemize}
\item [(a)]  Recall $\bpsi_k$ from \eqref{psi}. If $\bpsi_0\in \Delta_M^\circ,$ then
$\lim_{k\to\infty} \bpsi_k=\bmz.$
\item [(b)] $\bGamma$ is a diffeomorphism.
\end{itemize}
\end{theorema}
Note that the weak selection condition (cf. \cite{mweaks, paintbox, wfnowak}) imposed in part (i) of Theorem~\ref{ept} is not required
in the conditions of Theorem~\ref{bsa}.
\par
The evolutionary games with $\bA=\bA^T$ are called \textit{partnership games} \cite{hoffa}. We refer to \cite{partner, signal, partner3} for applications of partnership games in evolutionary biology. Condition \eqref{seq1} implies the existence of a unique evolutionary stable equilibrium for the evolutionary game defined by the payoff matrix $\bA.$  For a symmetric reversible matrix $\bA$ this condition holds if and only if $\bA$ has exactly one positive eigenvalue \cite{kingman, mandel}. Condition \eqref{seq3} ensures that the equilibrium is an interior point.  Note that since the equilibrium is unique, the conditions of Theorem~\ref{ept} are automatically satisfied, that is $\bpsi_k$ is permanent. Finally, under the conditions of Theorem~\ref{bsa}, the average payoff $\bx^T\bA\bx$ is a Lyapunov function \cite{kingman, mandel}:
\beq
\bpsi_k^T\bA\bpsi_k<\bpsi_{k+1}^T\bA\bpsi_{k+1}<\bmz^T \bA\bmz\qquad \qquad \forall\,k\in\zz_+,
\feq
as long as $\bpsi_0\in \Delta_M^\circ$ and $\bpsi_0\neq \bmz.$


{\small
\begin{thebibliography}{239}
\providecommand{\url}[1]{{#1}}
\providecommand{\urlprefix}{URL }
\expandafter\ifx\csname urlstyle\endcsname\relax
\providecommand{\doi}[1]{DOI~\discretionary{}{}{}#1}\else
\providecommand{\doi}{DOI~\discretionary{}{}{}\begingroup
\urlstyle{rm}\Url}\fi

\bibitem{chaose}
M.~Akhmet and M.~O.~Fen,
\emph{Input-output mechanism of the discrete chaos extension},
In V.~Afraimovich, J.~A.~T.~Machado, and J.~Zhang (Eds.),
\emph{Complex Motions and Chaos in Nonlinear Systems}, 203--233,
Nonlinear Syst. Complex., Vol. 15, Springer, 2016.
\filbreak

\bibitem{mweaks}
T.~Antal, A.~Traulsen, H.~Ohtsuki, C.~E.~Tarnita, and M.~A.~Nowak,
\emph{Mutation-selection equilibrium in games with multiple strategies},
J. Theoret. Biol. \textbf{258} (2009), 614--622.
\filbreak

\bibitem{nonlim}
M.~Archetti and I.~Scheuring,
\emph{Review: game theory of public goods in one-shot social dilemmas without assortment},
J. Theoret. Biol. \textbf{299} (2012), 9--20.
\filbreak

\bibitem{nonl3}
H.~Arnoldt, M.~Timme, and S.~Grosskinsky,
\emph{Frequency-dependent fitness induces multistability in coevolutionary dynamics},
J.~R.~Soc. Interface. \textbf{9} (2012), 3387--3396.
\filbreak

\bibitem{e}
M.~Assaf and B.~Meerson,
\emph{Noise enhanced persistence in a biochemical regulatory network with feedback control},
Phys. Rev. Lett. \textbf{100} (2008), 058105.
\filbreak

\bibitem{assaf}
M.~Assaf and B.~Meerson,
\emph{Extinction of metastable stochastic populations},
Phys. Rev. E \textbf{81} (2010), 021116.
\filbreak


\bibitem{assaf1}
M.~Assaf and M.~Mobilia,
\emph{Fixation of a deleterious allele under mutation pressure and finite selection intensity},
Phys. Rev. E \textbf{275} (2011), 93--103.
\filbreak

\bibitem{leto}
V.~Auletta, D.~Ferraioli, F.~Pasquale, and G.~Persiano,
\emph{Metastability of logit dynamics for coordination games},
Algorithmica \textbf{80} (2018), 3078--3131.
\filbreak

\bibitem{ozgur}
O.~Aydogmus,
\emph{On extinction time of a generalized endemic chain-binomial model},
Math. Biosci. \textbf{279} (2016), 38--42.
\filbreak

\bibitem{deva}
J.~Banks, J.~Brooks, G.~Cairns, G.~Davis, and P.~Stacey,
\emph{On Devaney's definition of chaos},
Amer. Math. Monthly \textbf{99} (1992), 332--334.
\filbreak

\bibitem{statph1}
N.~H.~Barton and J.~B.~Coe,
\emph{On the application of statistical physics to evolutionary biology},
J. Theoret. Biol. \textbf{259} (2009), 317--324.
\filbreak

\bibitem{stat24}
J.~P.~Barton, N.~Goonetilleke, T.~C.~Butler, B.~D.~Walker, A.~J.~McMichael, and A.~K.~Chakraborty,
\emph{Relative rate and location of intra-host HIV evolution to evade cellular immunity are predictable},
Nat. Commun. \textbf{7} (2016), 11660.
\filbreak

\bibitem{rfit}
W.~F.~Basener,
\emph{Limits of chaos and progress in evolutionary dynamics}, In
R.~J.~Marks II, M.~J.~Behe, W.~A.~Dembski, B.~L.~Gordon, and J.~C.~Sanford (Eds.),
\emph{Biological Information. New Perspectives}, 87--104,
Proceedings of the Symposium (Cornell University, 2011), World Scientific, 2013.
\filbreak

\bibitem{fmuta}
W.~F.~Basener and J.~C.~Sanford,
\emph{The fundamental theorem of natural selection with mutations},
J. Math. Biol. \textbf{76} (2018), 1589--1622.
\filbreak

\bibitem{wfc}
N.~Beerenwinkel, T.~Antal, D.~Dingli, A.~Traulsen, K.~W.~Kinzler, V.~E.~Velculescu, B.~Vogelstein, and M.~A.~Nowak,
\emph{Genetic progression and the waiting time to cancer},
PLoS Comput. Biol. \textbf{3} (2007), 2239--2246.
\filbreak

\bibitem{benspers}
M.~Bena\"{\i}m and S.~J.~Schreiber,
\emph{Persistence and extinction for stochastic ecological models with internal and external variables},
J. Math. Biol. \textbf{79} (2019), 393--431.
\filbreak

\bibitem{dns1}
O.~Bernardi and A.~Florio,
\emph{Existence of Lipschitz continuous Lyapunov functions strict outside the strong chain recurrent set},
Dyn. Syst. \textbf{34} (2019), 71--92.
\filbreak

\bibitem{birch}
J.~Birch,
\emph{Natural selection and the maximization of fitness},
Biol. Rev. \textbf{91} (2016), 712--727.
\filbreak

\bibitem{p3}
L.~Block and J.~E.~Franke,
\emph{The chain recurrent set, attractors, and explosions},
Ergodic Theory Dynam. Systems \textbf{5} (1985), 321--327.
\filbreak

\bibitem{paintbox}
F.~Boenkost, A.~Gonz\'{a}lez Casanova, C.~Pokalyuk, and A.~Wakolbinger,
\emph{Haldane's formula in Cannings models: The case of moderately weak selection},
2019, preprint is available at \url{https://arxiv.org/abs/1907.10049}.
\filbreak

\bibitem{lyapa}
N.~Bof, R.~Carli, and L.~Schenato,
\emph{Lyapunov theory for discrete time systems},
2018, technical report is available at \url{https://arxiv.org/abs/1809.05289}.
\filbreak

\bibitem{bovier}
A.~Bovier and F.~den~Hollander,
\emph{Metastability: A potential-Theoretic Approach},
Springer, 2015.
\filbreak

\bibitem{fisher5}
A.~S.~Bratus, Y.~S.~Semenov, and A.~S.~Novozhilov,
\emph{Adaptive fitness landscape for replicator systems: to maximize or not to maximize},
Math. Model. Nat. Phenom. \textbf{13} (2018), art. 25.
\filbreak

\bibitem{brem}
P.~Br\'{e}maud,
\emph{Markov Chains: Gibbs Fields, Monte Carlo Simulation, and Queues},
Springer, 1998.
\filbreak

\bibitem{negac}
M.~J.~E.~Broekman,  H.~C.~Muller-Landau,  M.~D. Visser,  E.~Jongejans, S.~J.~Wright, and  H.~de~Kroon,
\emph{Signs of stabilisation and stable coexistence},
to appear in Ecol. Lett. (2019).
\filbreak

\bibitem{nlinear}
M.~Broom and J.~Rycht\'{a}\v{r},
\emph{Game-Theoretical Models in Biology},
CRC Press, 2013.
\filbreak

\bibitem{bdet}
F.~M.~Buckley and P.~K.~Pollett,
\emph{Limit theorems for discrete-time metapopulation models},
Probab. Surv. \textbf{7} (2010), 53--83.
\filbreak

\bibitem{ad4}
J.~Cai, T.~Tan, and S.~H.~J.~Chan,
\emph{Bridging traditional evolutionary game theory and metabolic models for predicting Nash equilibrium of microbial metabolic interactions},
2019, preprint is available at \url{https://www.biorxiv.org/content/early/2019/05/14/623173.full.pdf}.
\filbreak

\bibitem{changing}
O.~Carja, U.~Liberman, and M.~W.~Feldman,
\emph{Evolution in changing environments: Modifiers of mutation, recombination, and migration},
Proc. Natl. Acad. Sci. U.\,S.\,A. \textbf{111} (2014), 17935--17940.
\filbreak

\bibitem{petopgenes}
A.~F.~Caulin, T.~A.~Graham, L.~S.~Wang,  and C.~C.~Maley,
\emph{Solutions to Peto's paradox revealed by mathematical modelling and cross-species cancer gene analysis},
Philos. Trans. Roy. Soc. London Ser. B \textbf{370} (2015), 20140222.
\filbreak

\bibitem{mgame9}
M.~Cavaliere, S.~Feng, O.~S.~Soyer, and J.~I.~Jim\'{e}nez,
\emph{Cooperation in microbial communities and their biotechnological applications},
Environ. Microbiol. \textbf{19} (2017), 2949--2963.
\filbreak

\bibitem{scale}
F.~A.~C.~C.~Chalub and M.~O.~Souza,
\emph{The frequency-dependent Wright-Fisher model: diffusive and non-diffusive approximations},
J. Math. Biol. \textbf{68} (2014), 1089--1133.
\filbreak

\bibitem{kimurac}
F.~A.~C.~C.~Chalub and M.~O.~Souza,
\emph{On the stochastic evolution of finite populations},
J. Math. Biol. \textbf{75} (2017), 1735--1774.
\filbreak

\bibitem{pro}
R.~S.~Chapkin et al.,
\emph{Propagation of a signal within a gut gene-microbial ecological network with an interface layer},
work in progress.
\filbreak

\bibitem{stat21}
H.~Chen and M.~Kardar,
\emph{Mean-field computational approach to HIV dynamics on a fitness landscape},
2019, preprint is available at \url{https://www.biorxiv.org/content/10.1101/518704v2}.
\filbreak

\bibitem{updates}
P.~Chesson,
\emph{Updates on mechanisms of maintenance of species diversity},
J. Ecol. \textbf{106} (2018), 1773--1794.
\filbreak

\bibitem{c}
J.~R.~Christie and M.~Beekman,
\emph{Selective sweeps of mitochondrial DNA can drive the evolution of uniparental inheritance},
Evolution \textbf{71} (2017), 2090--2099.
\filbreak

\bibitem{spgoods}
J.~S.~Chuang, O.~Rivoire, and S.~Leibler,
\emph{Simpson's paradox in a synthetic microbial system},
Science (N.~S.) \textbf{323} (2009), 272--275.
\filbreak

\bibitem{conley}
C.~Conley,
\emph{Isolated Invariant Sets and the Morse index},
C.\,B.\,M.\,S. Regional Conference Series in Math, Vol. 38, Amer. Math. Soc., 1978.
\filbreak

\bibitem{altruism}
G.~W.~A.~Constable, T.~Rogers, A.~J.~McKane, and C.~E.~Tarnita,
\emph{Demographic noise can reverse the direction of deterministic selection},
Proc. Natl. Acad. Sci. U.\,S.\,A. \textbf{113} (2016), E4745--E4754.
\filbreak

\bibitem{stat65}
C.~Coron, S.~M\'{e}l\'{e}ard, and D.~Villemonais,
\emph{Impact of demography on extinction/fixation events},
J. Math. Biol. \textbf{78} (2019), 549--577.
\filbreak

\bibitem{ctao}
R.~Cressman and Y.~Tao,
\emph{The replicator equation and other game dynamics},
Proc. Natl. Acad. Sci. U.\,S.\,A. \textbf{111} (2014), 10810--10817.
\filbreak

\bibitem{stat33}
I.~Cvijovi\'{c}, B.~H.~Good, and M.~M.~Desai,
\emph{The effect of strong purifying selection on genetic diversity},
Genetics \textbf{209} (2018), 1235--1278.
\filbreak

\bibitem{seneta1}
E.~V.~Doorn and P.~Pollett,
\emph{Quasi-stationary distributions for discrete state models},
European J. Operat. Res. \textbf{230} (2013), 1--14.
\filbreak

\bibitem{wfc1}
R.~S.~Datta, A.~Gutteridge, C.~Swanton, C.~C.~Maley, and T.~A.~Graham,
\emph{Modelling the evolution of genetic instability during tumour progression},
Evol. Appl. \textbf{6} (2013), 20--33.
\filbreak

\bibitem{re}
A.~Dean and N.~M.~Shnerb,
\emph{Stochasticity-induced stabilization in ecology and evolution},
2019, preprint is available at \url{https://www.biorxiv.org/content/10.1101/725341v1}.
\filbreak

\bibitem{app17}
R.~Der, C.~L.~Epstein, and J.~B.~Plotkin,
\emph{Generalized population models and the nature of genetic drift},
Theoret. Population Biol. \textbf{80} (2011), 80--99.
\filbreak

\bibitem{landscape}
J.~A.~G.~M.~de~Visser and J.~Krug,
\emph{Empirical fitness landscapes and the predictability of evolution},
Nat. Rev. Genet. \textbf{15} (2014), 480--490.
\filbreak

\bibitem{stat35}
R.~C.~Dewar, W.~B.~Sherwin, E.~Thomas, C.~E.~Holleley, and R.~A.~Nichols,
\emph{Predictions of single-nucleotide polymorphism differentiation between two populations in terms of mutual information},
Mol. Ecol. \textbf{20} (2011), 3156--66.
\filbreak

\bibitem{fmoran}
M.~Doebeli, Y.~Ispolatov, and B.~Simon,
\emph{Towards a mechanistic foundation of evolutionary theory},
eLife \textbf{6} (2017), e23804.
\filbreak

\bibitem{stat28}
M.~dos~Reis,
\emph{How to calculate the non-synonymous to synonymous rate ratio of protein-coding genes under the Fisher-Wright mutation-selection framework},
Biol. Lett. \textbf{11} (2015), 20141031.
\filbreak

\bibitem{fisher11}
A.~W.~F.~Edwards,
\emph{Analysing nature's experiment: Fisher's inductive theorem of natural selection},
Theoret. Population Biol. \textbf{109} (2016), 1--5.
\filbreak

\bibitem{ewens}
W.~J.~Ewens,
\emph{Mathematical Population Genetics (I. Theoretical Introduction)}, 2nd ed.,
Interdisciplinary Applied Mathematics Series, Vol.~27, Springer, 2012.
\filbreak

\bibitem{elf}
W.~J.~Ewens and S.~Lessard,
\emph{On the interpretation and relevance of the Fundamental Theorem of Natural Selection},
Theoret. Population Biol. \textbf{104} (2015), 59--67.
\filbreak

\bibitem{am}
A.~Fathi and P.~Pageault,
\emph{Aubry-Mather theory for homeomorphisms},
Ergodic Theory Dynam. Systems \textbf{35} (2015), 1187--1207.
\filbreak

\bibitem{dns6}
A.~Fathi and P.~Pageault,
\emph{Smoothing Lyapunov functions},
Trans. Amer. Math. Soc. \textbf{371} (2019), 1677--1700.
\filbreak

\bibitem{quasi}
M.~Faure and S.~Schreiber,
\emph{Quasi-stationary distributions for randomly perturbed dynamical systems},
Ann. Appl. Probab. \textbf{24} (2014), 553--598.
\filbreak

\bibitem{fergi}
T.~S.~Ferguson,
\emph{Mathematical Statistics: A Decision Theoretic Approach},
Academic Press, 1967.
\filbreak

\bibitem{p1}
J.~E.~Franke and J.~F.~Selgrade,
\emph{Abstract $\omega$-limit sets, chain recurrent sets, and basic sets for flows},
Proc. Amer. Math. Soc. \textbf{60} (1976), 309--316.
\filbreak

\bibitem{franks}
J.~Franks,
\emph{Notes on chain recurrence and Lyapunonv functions},
(unpublished) lecture notes, available at \url{https://arxiv.org/pdf/1704.07264.pdf}.
\filbreak

\bibitem{garay}
B.~M.~Garay and J.~Hofbauer,
\emph{Robust permanence for ecological differential equations, minimax, and discretizations},
SIAM J. Math. Anal. \textbf{34} (2003), 1007--1039.
\filbreak

\bibitem{stat18}
V.~Garcia, E.~C.~Glassberg, A.~Harpak, and M.~W.~Feldman,
\emph{Clonal interference can cause wavelet-like oscillations of multilocus linkage disequilibrium},
J. R. Soc. Interface \textbf{15} (2018), 20170921.
\filbreak

\bibitem{gavr}
S.~Gavrilets,
\emph{High-dimensional fitness landscapes and speciation},
In M.~Pigliucci and G.~B.~M\"{u}ller (Eds.),
\emph{Evolution - The Extended Synthesis}, 45--79,
MIT Press, 2010.
\filbreak

\bibitem{multip}
C.~S.~Gokhale and A.~Traulsen,
\emph{Evolutionary multiplayer games},
Dyn. Games Appl. \textbf{4} (2014), 468--488.
\filbreak

\bibitem{gompert}
Z.~Gompert,
\emph{Bayesian inference of selection in a heterogeneous environment from genetic time-series data},
Mol. Ecol. \textbf{25} (2016), 121--34.
\filbreak

\bibitem{tinkering}
B.~H.~Good, S.~Martis, and O.~Hallatschek,
\emph{Directional selection limits ecological diversification and promotes
ecological tinkering during the competition for substitutable resources},
2018, preprint is available at \url{https://www.biorxiv.org/content/10.1101/292821v1}.
\filbreak

\bibitem{indirect}
L.~Goudenège and P.-A.~Zitt,
\emph{A Wright-Fisher model with indirect selection},
J. Math. Biol. \textbf{71} (2015), 1411--1450.
\filbreak

\bibitem{hulu}
N.~Grosjean and T.~Huillet,
\emph{Wright-Fisher-like models with constant population size on average},
Int. J. Biomath. \textbf{10} (2017), 1750078.
\filbreak

\bibitem{nlpromote}
N.~Harmand, V.~Federico, T.~Hindre, and T.~Lenormand,
\emph{Nonlinear frequency-dependent selection promotes long-term coexistence between bacteria species},
Ecol. Lett. \textbf{22} (2019), 1192--1202.
\filbreak

\bibitem{entropy}
M.~Harper and D.~Fryer,
\emph{Stationary stability for evolutionary dynamics in finite populations},
Entropy \textbf{18} (2016), paper no. 316.
\filbreak

\bibitem{fitn}
D.~L.~Hartle and A.~G.~Clark,
\emph{Principles of Population Genetics}, 4th ed.,
Sinauer, 2007.
\filbreak

\bibitem{c1}
A.~Hintze, R.~S.~Olson, C.~Adami, and R.~Hertwig,
\emph{Risk sensitivity as an evolutionary adaptation},
Sci. Rep. \textbf{5} (2015), 8242.
\filbreak

\bibitem{stat62}
A.~Hobolth and J.~Sire\'{e}n,
\emph{The multivariate Wright-Fisher process with mutation: Moment-based analysis and inference using a hierarchical Beta model},
Theoret. Population Biol. \textbf{108} (2016), 36--50.
\filbreak

\bibitem{hinq}
W.~Hoeffding,
\emph{Probability inequalities for sums of bounded random variables},
J. Amer. Stat. Assoc. \textbf{58} (1963), 13--30.
\filbreak

\bibitem{fishm}
J.~Hofbauer,
\emph{The selection mutation equation},
J. Math. Biol. \textbf{23} (1985), 41--53.
\filbreak

\bibitem{hoffa}
J.~Hofbauer,
\emph{Deterministic evolutionary game dynamics},
In K.~Sigmund (Ed.), \emph{Evolutionary Game Dynamics}, 61--79, Proc. Sympos. Appl. Math., Vol. 69,
AMS Short Course Lecture Notes, Amer. Math. Soc., 2011
\filbreak

\bibitem{HS}
J.~Hofbauer and K.~Sigmund,
\emph{Evolutionary Games and Population Dynamics},
Cambridge University Press, 1998.
\filbreak

\bibitem{rewhs}
J.~Hofbauer and K.~Sigmund,
\emph{Evolutionary game dynamics},
Bull. Amer. Math. Soc. (N.~S.) \textbf{40} (2003), 479--519.
\filbreak

\bibitem{iwf}
J.~Hofrichter, J.~Jost, and T.~Tran,
\emph{Information Geometry and Population Genetics.
The Mathematical Structure of the Wright-Fisher Model}, Springer, 2017.
\filbreak

\bibitem{dnoise1}
B.~Houchmandzadeh,
\emph{Fluctuation driven fixation of cooperative behavior},
Biosystems \textbf{127} (2015), 60--66.
\filbreak

\bibitem{partner}
S.~Hummert, C.~Glock, S.~N. Lang, C.~Hummert, C.~Skerka, P.~F.~Zipfel, S.~Germerodt, and S.~Schuster,
\emph{Playing ``hide-and-seek" with factor H: game-theoretical analysis of a single nucleotide polymorphism},
J.~R.~Soc.~Interface \textbf{15} (2018), 20170963.
\filbreak

\bibitem{mgame7}
S.~Hummert, K. Bohl, D.~Basanta, A.~Deutsch, S.~Werner, G.~Thei{\ss}en, A.~Schroeter, and S.~Schuster,
\emph{Evolutionary game theory: cells as players},
Mol. BioSyst. \textbf{10} (2014), 3044.
\filbreak

\bibitem{hurley1}
M.~Hurley,
\emph{Lyapunov functions and attractors in arbitrary metric spaces},
Proc. Amer. Math. Soc. \textbf{126} (1998), 245--256.
\filbreak

\bibitem{hutch}
G.~E.~Hutchinson,
\emph{The paradox of the plankton},
Am. Nat. \textbf{95} (1961), 137--145.
\filbreak

\bibitem{signal}
S.~Huttegger, B.~Skyrms, P.~Tarr\`{e}s, and E.~Wagner,
\emph{Some dynamics of signaling games},
Proc. Natl. Acad. Sci. U.\,S.\,A. \textbf{111} (2014), 10873--10880.
\filbreak

\bibitem{hutson}
V.~Hutson and K.~Schmitt,
\emph{Permanence and the dynamics of biological systems},
Math. Biosci. \textbf{111} (1992), 1--71.
\filbreak

\bibitem{eagle}
D.~L.~Iglehart,
\emph{Extreme values in the $GI\slash G\slash 1$ queue},
Ann. Math. Statist. \textbf{43} (1972), 627--635.
\filbreak

\bibitem{isuka}
M.~Iizuka,
\emph{Effective population size of a population with stochastically varying size},
J. Math. Biol. \textbf{61} (2010), 359--375.
\filbreak

\bibitem{wfnowak}
L.~A.~Imhof and M.~A.~Nowak,
\emph{Evolutionary game dynamics in a Wright-Fisher process},
J. Math. Biol. \textbf{52} (2006), 667--681.
\filbreak

\bibitem{iwasa}
Y.~Iwasa and F.~Michor,
\emph{Evolutionary dynamics of intratumor heterogeneity},
PLoS One. \textbf{6} (2011), e17866.
\filbreak

\bibitem{stat19}
S.~John and S.~Seetharaman,
\emph{Exploiting the adaptation dynamics to predict the distribution of beneficial fitness effects},
PLoS One \textbf{11} (2016), e0151795.
\filbreak

\bibitem{ckperm}
Y.~Kang and P.~Chesson,
\emph{Relative nonlinearity and permanence},
Theoret. Population Biol. \textbf{78} (2010), 26--35.
\filbreak

\bibitem{kms}
S.~Karlin and J.~McGregor,
\emph{Direct product branching processes and related Markov chains},
Proc. Nat. Acad. Sci. U.\,S.\,A. \textbf{51} (1964), 598--602.
\filbreak

\bibitem{karr}
A.~F.~Karr,
\emph{Weak convergence of a sequence of Markov chains},
Z. Wahrsch. Verw. Gebiete \textbf{33} (1975), 41--48.
\filbreak

\bibitem{kingman}
J.~F.~C.~Kingman,
\emph{A mathematical problem in population genetics},
Proc. Camb. Phil. Soc. \textbf{57} (1961), 574--582.
\filbreak

\bibitem{knerman}
F.~C.~Klebaner and O.~Nerman,
\emph{Autoregressive approximation in branching processes with a threshold},
Stochastic Process. Appl. \textbf{51} (1994), 1--7.
\filbreak

\bibitem{stat46}
B.~Koopmann, J.~M\"{u}ller, A.~Tellier, and D.~\v{Z}ivkovi\'{c},
\emph{Fisher–Wright model with deterministic seed bank and selection},
Theoret. Population Biol. \textbf{114} (2017), 29--39.
\filbreak

\bibitem{kroumis}
D.~Kroumi and S.~Lessard,
\emph{Evolution of cooperation in a multidimensional phenotype space},
Theoret. Population Biol. \textbf{102} (2015), 60--75.
\filbreak

\bibitem{pgoods5}
S.~Kurokawa and Y.~Ihara,
\emph{Emergence of cooperation in public goods games},
Proc. R. Soc. B \textbf{276} (2009), 1379--1384.
\filbreak

\bibitem{landim}
C.~Landim,
\emph{Metastable Markov chains},
Probab. Surveys \textbf{16} (2019), 143--227.
\filbreak

\bibitem{lawler}
G.~F.~Lawler,
\emph{Introduction to stochastic processes}, 2nd ed., Chapman \& Hall/CRC, 2006.
\filbreak

\bibitem{re-tumor}
D.~Li and F.~Cheng,
\emph{The extinction and persistence of tumor evolution influenced by external fluctuations and periodic treatment},
to appear in Qual. Theory Dyn. Syst. (2019)
\filbreak

\bibitem{coex}
L.~Li and P.~Chesson,
\emph{The effects of dynamical rates on species coexistence in a variable environment: The paradox of the plankton revisited},
Am Nat. \textbf{188} (2016), 46-58.
\filbreak

\bibitem{losakin}
V.~Losert and E.~Akin,
\emph{Dynamics of games and genes: discrete versus continuous time},
J. Math. Biol. \textbf{17} (1983), 241--251
\filbreak

\bibitem{re1}
A.~Mahdipour-Shirayeh, A.~H.~Darooneh, A.~D.~Long, N.~L.~Komarova, and M.~Kohandel,
\emph{Genotype by random environmental interactions gives an advantage to non-favored minor alleles},
Sci. Rep. \textbf{7} (2017), 5193.
\filbreak

\bibitem{mandel}
S.~P.~H.~Mandel,
\emph{The stability of a multiple allelic system},
Heredity \textbf{13} (1959), 289--302.
\filbreak

\bibitem{pgoods}
A.~McAvoy, N.~Fraiman, C.~Hauert, J.~Wakeley, and M.~A.~Nowak,
\emph{Public goods games in populations with fluctuating size},
Theoret. Population Biol. \textbf{121} (2018), 72--84.
\filbreak

\bibitem{q}
S.~M\'{e}l\'{e}ard and D.~Villemonais,
\emph{Quasi-stationary distributions and population processes},
Probab. Surv. \textbf{9} (2012), 340--410.
\filbreak

\bibitem{nagy}
T.~Nagylaki,
\emph{The Gaussian approximation for random genetic drift},
in S.~Karlin and E.~Nevo (Eds.),
\emph{Evolutionary Processes and Theory}, 629--642,
Academic Press, 1986.
\filbreak

\bibitem{app1}
T.~Nagylaki,
\emph{Models and approximations for random genetic drift},
Theoret. Population Biol. \textbf{37} (1990), 192--212.
\filbreak

\bibitem{norton}
D.~E.~Norton,
\emph{The fundamental theorem of dynamical systems},
Comment. Math. Univ. Carolin. \textbf{36} (1995), 585--597.
\filbreak

\bibitem{stat47}
A.~Nourmohammad, J.~Otwinowski, and J.~B.~Plotkin,
\emph{Host-pathogen coevolution and the emergence of broadly neutralizing antibodies in chronic infections},
PLoS Genet. \textbf{12} (2016), e1006171.
\filbreak

\bibitem{bionowak}
M.~A.~Nowak and K.~Sigmund,
\emph{Evolutionary dynamics of biological games},
Science \textbf{303} (2004), 793--799.
\filbreak

\bibitem{rugged}
U.~Obolski, Y.~Ram, and L.~Hadany,
\emph{Key issues review: evolution on rugged adaptive landscapes},
Rep. Prog. Phys. \textbf{81} (2018), 012602.
\filbreak

\bibitem{forr}
H.~A.~Orr,
\emph{Fitness and its role in evolutionary genetics},
Nat. Rev. Genet. \textbf{10} (2009), 531--539.
\filbreak

\bibitem{dns3}
P.~Pageault,
\emph{Conley barriers and their applications: chain-recurrence and Lyapunov functions},
Topology Appl. \textbf{156} (2009), 2426--2442.
\filbreak

\bibitem{imix}
I.~Panageas, P.~Srivastava, and N.~K.~Vishnoi,
\emph{Evolutionary dynamics in finite populations mix rapidly},
In R.~Krauthgamer (Ed.),
\emph{Proceedings of the Twenty-Seventh Annual ACM-SIAM Symposium on Discrete Algorithms}, 480--497,
SIAM, 2016.
\filbreak

\bibitem{fisheri1}
C.~Papadimitriou and G.~Piliouras,
\emph{From Nash equilibria to chain recurrent sets: solution concepts and topology},
In M.~Sudan (Ed.),
\emph{ITCS'16—Proceedings of the 2016 ACM Conference on Innovations in Theoretical Computer Science}, 227--235, ACM, New York, 2016.
\filbreak

\bibitem{clonal}
S.-C.~Park and J.~Krug,
\emph{Clonal interference in large populations},
Proc. Natl. Acad. Sci. U.\,S.\,A. \textbf{104} (2007), 18135.
\filbreak

\bibitem{metapark}
H.~J.~Park and A.~Traulsen,
\emph{Extinction dynamics from metastable coexistences in an evolutionary game},
Phys. Rev. E \textbf{96} (2017), 042412.
\filbreak

\bibitem{dapr4}
C.~Parra-Rojas, J.~D.~Challenger, D.~Fanelli, and A.~J.~McKane,
\emph{Intrinsic noise and two-dimensional maps: Quasicycles, quasiperiodicity, and chaos},
Phys. Rev. E \textbf{90} (2014), 032135.
\filbreak

\bibitem{perko}
L.~Perko,
\emph{Differential Equations and Dynamical Systems}, 3rd ed.,
Springer, 1991.
\filbreak

\bibitem{wfctunelling}
S.~R.~Proulx,
\emph{The rate of multi-step evolution in Moran and Wright-Fisher populations},
Theoret. Population Biol. \textbf{80} (2011), 197--207.
\filbreak

\bibitem{fitn1}
D.~C.~Queller,
\emph{Fundamental theorems of evolution},
Am. Nat. \textbf{189} (2017), 345--353.

\filbreak

\bibitem{rao}
C.~R.~Rao,
\emph{Linear Statistical Inference and Its Applications}, 2nd ed.,
Wiley, 2001.
\filbreak

\bibitem{aggreg}
M.~Rohlfs and T.~S.~Hoffmeist,
\emph{An evolutionary explanation of the aggregation model of species coexistence},
Proc. Biol. Sci. \textbf{270} (2003), 33--35.
\filbreak

\bibitem{ad1}
E.~Ruppin, J.~A.~Papin, L.~F.~de~Figueiredo, and S.~Schuster,
\emph{Metabolic reconstruction, constraint-based analysis and game theory to probe genomescale metabolic networks},
Curr. Opin. Biotechnol. \textbf{21} (2010), 502--510.
\filbreak

\bibitem{e1}
M.~S.~Samoilov snd A.~P.~Arkin,
\emph{Deviant effects in molecular reaction pathways},
Nat. Biotechnol. \textbf{24} (2006), 1235--40.
\filbreak

\bibitem{econegt}
W.~H.~Sandholm,
\emph{Population Games and Evolutionary Dynamics},
Series on Economic Learning and Social Evolution, Vol.~8,
MIT Press, 2010.
\filbreak

\bibitem{ska}
K.~A.~Schneider,
\emph{Maximization principles for frequency-dependent selection II: the one-locus multiallele case},
J. Math. Biol. \textbf{61} (2010), 95--132.
\filbreak

\bibitem{coexit}
S.~J.~Schreiber, J.~M.~Levine, O.~Godoy, N.~J.~B.~Kraft, and S.~P.~Hart,
\emph{Does deterministic coexistence theory matter in a finite world? Insights from serpentine annual plants},
2018, preprint is available at \url{https://www.biorxiv.org/content/10.1101/290882v1}.
\filbreak

\bibitem{zig}
K.~Sigmund,
\emph{A survey of replicator equations}.
In J.~L.~Casti and A.~Karlqvist (Eds.),
\emph{Complexity, Language, and Life: Mathematical Approaches}, 88--104,
Biomathematics, Vol. 16. Springer, 1986.

\bibitem{silver}
S.~Silverman,
\emph{On maps with dense orbits and the definition of chaos},
Rocky Mountain J. Math. \textbf{22} (1992), 353--375.
\filbreak

\bibitem{stat63}
N.~Simonsen Speed, D.~J.~Balding, and A.~Hobolth,
\emph{A general framework for moment-based analysis of genetic data},
J. Math. Biol. \textbf{78} (2019), 1727--1769.
\filbreak

\bibitem{c3}
M.~Spichtig and T.~J.~Kawecki,
\emph{The maintenance (or not) of polygenic variation by soft selection in heterogeneous environments},
Am. Nat. \textbf{164} (2004), 70--84.
\filbreak

\bibitem{mb8fit}
C.~Stephens,
\emph{Selection, drift, and the ``forces" of evolution},
Philosophy of Science \textbf{71} (2014), 550--570.
\filbreak

\bibitem{fitn3}
E.~I.~Svensson and T.~Connallon,
\emph{How frequency‐dependent selection affects population fitness, maladaptation and evolutionary rescue},
Evol. Appl. \textbf{12} (2019), 1243--1258.
\filbreak

\bibitem{mreview}
C.~E.~Tarnita,
\emph{The ecology and evolution of social behavior in microbes},
J. Exp. Biol. \textbf{220} (2017), 18--24.
\filbreak

\bibitem{taylor}
C.~Taylor and M.~Nowak,
\emph{Evolutionary game dynamics with non-uniform interaction rates},
Theoret. Population Biol. \textbf{69} (2006), 243--252.
\filbreak

\bibitem{fitco}
A.~Traulsen, J.~C.~Claussen, and C.~Hauert,
\emph{Coevolutionary Dynamics: From finite to infinite populations},
Phys. Rev. Lett. \textbf{95} (2005), 238701.
\filbreak

\bibitem{partner3}
A.~Traulsen and F.~A.~Reed,
\emph{From genes to games: cooperation and cyclic dominance in meiotic drive},
J. Theoret. Biol. \textbf{299} (2012), 120--125.
\filbreak

\bibitem{wfbio}
M.~Vallier, M.~Abou Chakra, L.~Hindersin, M.~Linnenbrink, A.~Traulsen, and J.~F.~Baines,
\emph{Evaluating the maintenance of disease-associated variation at the blood group-related gene B4galnt2 in house mice},
BMC Evol. Biol. \textbf{17} (2017), 187.
\filbreak

\bibitem{vas}
L.~Wasserman,
\emph{All of Statistics: A Concise Course in Statistical Inference},
Springer Texts in Statistics, Springer, 2004.
\filbreak

\bibitem{wfmratchet}
D.~Waxman and L.~Loewe,
\emph{A stochastic model for a single click of Muller's ratchet},
J. Theoret. Biol. \textbf{264} (2010), 1120--1132.
\filbreak

\bibitem{sinergy3}
S.~A.~West, S.~P.~Diggle, A.~Buckling, A.~Gardner, and A.~S.~Griffin,
\emph{The social lives of microbes},
Annu. Rev. Ecol. Evol. Syst. \textbf{38} (2007), 53--77.
\filbreak

\bibitem{mb5wf}
Q.~Zeng, S.~Wu, J.~Sukumaran, and A.~Rodrigo,
\emph{Models of microbiome evolution incorporating host and microbial selection},
Microbiome \textbf{5} (2017), 127.
\filbreak

\bibitem{tale}
Y.~Zhang, F.~Fu, T.~Wu, G.~Xie, and L.~Wang,
\emph{A tale of two contribution mechanisms for nonlinear public goods},
Sci. Rep. \textbf{3} (2013), 2021.
\filbreak

\bibitem{stat45}
T.~Zinger, M.~Gelbart, D.~Miller, P.~S.~Pennings, and A.~Stern,
\emph{Inferring population genetics parameters of evolving viruses using time-series data},
Virus Evol. \textbf{5} (2019), vez011.
\filbreak

\bibitem{mgame3}
A.~R.~Zomorrodi and D.~Segr\`{e},
\emph{Genome-driven evolutionary game theory helps understand the rise of metabolic interdependencies in microbial communities},
Nat. Commun. \textbf{8} (2017), 1563.
\filbreak


\end{thebibliography}
}
\end{document}